\documentclass{lms}
\usepackage{amssymb}
\usepackage{tikz}
\usepackage{tikz-cd}
\usetikzlibrary{matrix, arrows,  decorations.markings,  patterns,  plotmarks}
\usepackage{amsmath}
\usepackage{amssymb}
\usepackage{mathtools}
\usepackage{verbatim}
\usepackage{todonotes}

\usepackage[backend=bibtex,  maxcitenames=5, maxbibnames=9 ]{biblatex}
\renewbibmacro{in:}{} 
\usepackage{enumerate}
\usepackage{subcaption}

\usepackage{prettyref}

\newtheorem{theorem}{Theorem}[section]
\newtheorem{lemma}[theorem]{Lemma}
\newtheorem{prop}[theorem]{Proposition}
\newtheorem{corollary}[theorem]{Corollary}
\newnumbered{definition}[theorem]{Definition}
\newnumbered{notation}[theorem]{Notation}
\newnumbered{example}[theorem]{Example}
\newnumbered{conjecture}[theorem]{Conjecture}
\newnumbered{remark}[theorem]{Remark}
\newnumbered{assumption}[theorem]{Assumption}

\newrefformat{thm}{Theorem \ref{#1}}
\newrefformat{lemma}{Lemma \ref{#1}}
\newrefformat{prop}{Proposition \ref{#1}}
\newrefformat{def}{Definition \ref{#1}} 
\newrefformat{cor}{Corollary \ref{#1}}
\newrefformat{eq}{Equation \ref{#1}}
\newrefformat{conj}{Conjecture \ref{#1}}
\newrefformat{fig}{Figure \ref{#1}}
\newrefformat{app}{Appendix \ref{#1}}
\newrefformat{sec}{Section \ref{#1}}
\newrefformat{subsec}{Subsection \ref{#1}}
\newrefformat{subsub}{Subsubsection \ref{#1}}
\newrefformat{assum}{Assumption \ref{#1}}
\newrefformat{exam}{Example \ref{#1}}

\newcounter{specialtheorem}

\newcommand{\ZZ}{\mathbb{Z}}
\newcommand{\CC}{\mathbb{C}}
\newcommand{\RR}{\mathbb{R}}
\newcommand{\NN}{\mathbb{N}}

\newcommand{\im}{\text{ Im }}
\newcommand{\into}{\hookrightarrow}
\newcommand{\tensor}{\otimes}

\newcommand{\End}{\text{End}}

\DeclareMathOperator{\Map}{Map}

\DeclareMathOperator{\Crit}{Crit}

\DeclareMathOperator{\id}{id}

\DeclareMathOperator{\vol}{vol}

\DeclareMathOperator{\grad}{grad}

\DeclareMathOperator{\Hess}{Hess}

\newcommand{\CP}{\mathbb{CP}}

\DeclareMathOperator{\val}{val}
\DeclareMathOperator{\Flux}{Flux}

\DeclareMathOperator{\Hull}{Hull}
\DeclareMathOperator{\avg}{avg}

\newcommand{\nw}{w}
\newcommand{\opm}
{{\bullet/\circ}}
\renewcommand{\Im}{\text{Im}}
\renewcommand{\Re}{\text{Re}}

    \DeclareFontFamily{U}{wncy}{}
    \DeclareFontShape{U}{wncy}{m}{n}{<->wncyr10}{}
    \DeclareSymbolFont{mcy}{U}{wncy}{m}{n}
    \DeclareMathSymbol{\Sh}{\mathord}{mcy}{"58}

\DeclareMathOperator{\Fuk}{Fuk}

\DeclareMathOperator{\Coh}{Coh}

\newcommand{\CF}{{CF^\bullet}}

\newcommand{\CM}{{CM^\bullet}}

\graphicspath{{figures/}}
\extraline{
	A portion of this work was completed at ETH Z\"urich.
	This work was partially supported by 
		NSF grants DMS-1406274 and 
      DMS-1344991;
   EPSRC Grant EP/N03189X/1;
	and by a Simons Foundation grant.}
\classno{53D37  (Primary), 53D12, 14T05}

\begin{document}
\title{Tropical Lagrangian Hypersurfaces are Unobstructed}
\author{Jeff Hicks}
\maketitle
\begin{abstract}
	We produce for each tropical hypersurface $V(\phi)\subset Q=\RR^n$ a Lagrangian $L(\phi)\subset (\CC^*)^n$ whose moment map projection is a tropical amoeba of $V(\phi)$. When these Lagrangians are admissible in the Fukaya-Seidel category, we show that they are unobstructed objects of the Fukaya category, and mirror to sheaves supported on complex hypersurfaces in a toric mirror. 
\end{abstract}

\section{Introduction}
\subsection{SYZ Fibrations and Mirror Symmetry}
\label{subsec:syz}
Mirror symmetry was proposed by physicists as a duality for Calabi-Yau manifolds which interchanges the symplectic geometry on a space $X$ with the complex geometry on a mirror space $\check X$ \cite{candelas1991pair}. 
One version of this duality comes from the homological mirror symmetry conjecture, which advances  that the categories $\Fuk(X)$ and $D^b\Coh(\check X)$ are equivalent as triangulated $A_\infty$ categories \cite{kontsevich1994homological}.
Here, $\Fuk(X)$ is the derived Fukaya category of $X$, whose objects are Lagrangian submanifolds $L$ and whose morphisms are given by Floer cochain groups $CF^\bullet(L_0, L_1)$.
On the mirror side, the complex space $\check X$ has the derived category of coherent sheaves $D^b\Coh(\check X)$.
Though the analytic and algebraic technicalities of working with Fukaya categories are substantial, this conjecture has been proven on a growing number of examples. 
A first example outside of Calabi-Yau manifolds are toric varieties $\check X_\Sigma$ with fan $\Sigma$, which are mirror to Landau Ginzburg models $((\CC^*)^n, W_\Sigma)$ \cite{hori2000mirror}. \cite{abouzaid2009morse}  proves that  $\Fuk((\CC^*)^n, W_\Sigma)$ is mirror to $D^b\Coh(\check X_\Sigma)$.

Separately from homological mirror symmetry, an approach to constructing mirror pairs $(X, \check X)$ was proposed by \citeauthor{strominger1996mirror}, who posited that mirror spaces $X$ and $\check X$ carry dual torus fibrations \cite{strominger1996mirror}. 
In this framework,  $X\to Q$ and $\check X \to Q$ are almost toric Lagrangian fibrations over a common base $Q$.
The data of an affine structure on $Q$ gives rise to a symplectic structure on $X$, and a complex structure on $\check X$.
These fibrations provide a mechanism for mirror symmetry where the symplectic geometry of $X$ and complex geometry of $\check X$ are mutually compared to the  tropical geometry of the base $Q$ in the so-called large complex structure limit \cite{gross2003affine,gross2013logarithmic}. 
From this perspective, both the symplectic geometry of $X$ and complex geometry on $\check X$ may be compared to tropical geometry on $Q$, recovering mirror symmetry. 

The correspondence between complex geometry on $\check X$ and tropical geometry on $Q$ can be understood by replacing the defining polynomials for an affine variety with the corresponding tropical polynomials \cite{mikhalkin2005enumerative, kontsevich2001homological}.
In particular, the image $\check{\val}(D)$ of complex subvarieties of $D\subset \check X$ can be described in certain examples as a tropical ``amoeba'' of a tropical variety $V(\phi)\subset Q.$

On the $A$-model, the lack of rigidity of Lagrangian submanifolds means that there is no reason for $\val (L)\subset Q$ to live near a tropical variety.
To obtain a well defined correspondence between Lagrangian submanifolds and tropical geometry one can use perspective from family Floer homology \cite{abouzaid2014family,fukaya2002Floer}, which provides a bridge between SYZ fibrations and the Fukaya category. 
Starting with the observation that fibers of the SYZ fibration $F_q\subset X$ are candidate mirrors to skyscraper sheaves of points on the space $\check X$,
family Floer theory associates to each Lagrangian submanifold $L\subset X$ a sheaf $\mathcal L$ on a rigid analytic mirror space $\check X^\Lambda$. 
The valuation of this sheaf gives us a tropical subvariety related to the original Lagrangian $L$. 
When $L$ is a section of $X\to Q$, the sheaf built is a line bundle.
The tropical support becomes all of $Q$. 

The recent parallel works of \cite{matessi2018Lagrangian,mikhalkin2018examples,mak2019tropically} provide methods for constructing \emph{tropical Lagrangian submanifolds} $L(\phi)\subset X$ whose image under $\val: X\to Q$ is nearby a tropical hypersurface $V(\phi)$.
In this paper, we show how such a tropical Lagrangian submanifold can be constructed via Lagrangian surgery. 
With this method of construction we show that tropical Lagrangians $L(\phi)$ are \emph{unobstructed by bounding cochain} and can be therefore considered as objects of the Fukaya category.
Provided an appropriate version of the Fukaya category exists, we additionally prove a homological mirror symmetry for these Lagrangians, thereby extending the intuition above to holomorphic sheaves supported on hypersurfaces.
\subsection{Summary of Results}

In \prettyref{sec:background}, we provide necessary background related to tropical geometry, Lagrangian cobordisms, and homological mirror symmetry for toric varieties.
The review of tropical geometry is included to fix notation, and introduce the non-self-intersecting property of a tropical hypersurface. 
The section on Lagrangian cobordisms may be safely skipped by experts who are familiar with the results of \cite{biran2013fukayacategories, fukaya2007chapter10}.
Our notation for Lagrangian cobordism follows \cite{haug2015torus}.
In \prettyref{subsec:toricbackground}, we review the monomial admissibility condition of \cite{hanlon2018monodromy}, and state homological mirror symmetry for toric varieties.

\prettyref{sec:hypersurfaces} associates to each tropical polynomial $\phi:Q=\RR^n\to \RR$ a Lagrangian submanifold $L(\phi)\subset X=(\CC^*)^n$ whose projection under the valuation map $\val: X\to Q$ lies near the tropical variety $V(\phi)\subset Q$.
These tropical Lagrangian submanifolds are constructed using Lagrangian surgery.
We prove that this construction is independent of choices made up to exact Lagrangian isotopy. 
In \prettyref{sec:unobstructedness} we show that the constructed tropical Lagrangians are unobstructed by bounding cochain, and are monomial-admissible Lagrangians.
Between these two sections, the main result of this paper can be summarized as (see \prettyref{thm:tropicalLagrangians} for details): \refstepcounter{specialtheorem} 
\begin{theorem*}[\Alph{specialtheorem}]
	\label{thm:existenceandunobstructedness}
	Let $V(\phi)$ be a tropical hypersurface of $\RR^n$ without self-intersections.  For every $\epsilon>0$ there exists a tropical Lagrangian $L(\phi)\subset (\CC^*)^n$ whose valuation projection is $\epsilon$-close to $V(\phi)$ in the Hausdorff metric.
	Furthermore, this Lagrangian is Floer-theoretically unobstructed.
\end{theorem*}

The proof of Theorem A is given modulo hypothesis stated in \prettyref{subsec:assumptions}, which relate to extending the definition of the pearly complex of \cite{charest2015floer} to the monomial-admissible setting.
We additionally include an outline of how domain-dependent perturbations or abstract perturbation techniques can be employed to remove these hypothesis.
In \prettyref{sec:tropicalmirrorsymmetry} we prove that the constructed Lagrangians $L(\phi)$ are mirror to structure sheaves of divisors in a mirror toric variety $\check X_\Sigma$ (see \prettyref{thm:tropicalLagrangianhms} for details)

\refstepcounter{specialtheorem}
\begin{theorem*}[\Alph{specialtheorem}]
	\label{thm:mirrorsymmetry}
	Let $\check X^\Lambda_\Sigma$ be a toric variety, and let $X=((\CC^*)^n, W_\Sigma)$ be its mirror Landau Ginzburg model.  Let $D$ be a base point-free divisor of $\check X^\Lambda_\Sigma$ with tropicalization given by the tropical variety $V(\phi)$.
	 The corresponding tropical Lagrangian $L(\phi)$ is homologically mirror to a  structure sheaf $\mathcal O_{D'}$, where $D$ and $D'$ are rationally equivalent.
\end{theorem*}

The construction of a bounding cochain for the tropical Lagrangian uses some general statements about filtered $A_\infty$ algebras and deforming cochains, a review of which is included in \prettyref{app:ainftyrefresher}.

This work is a portion of the author's PhD thesis, \cite{hicks2019tropical}, which additionally explores applications and examples of tropical Lagrangians in the context of homological mirror symmetry.
The tropical Lagrangians constructed in this paper are Hamiltonian isotopic to existing parallel constructions of tropical Lagrangians presented in \cite{matessi2018Lagrangian, mikhalkin2018examples,mak2019tropically}.
These constructions have been recently employed in the works of \cite{sheridan2018lagrangian,treumann2018kasteleyn}, which look at how tropical geometry can be used to obstruct the existence of certain unobstructed Lagrangian cobordisms, and how the combinatorics of dimers is related to these tropical Lagrangians.

\begin{acknowledgements}
I would like to thank my advisor Denis Auroux whose guidance and comments have helped me at every step of this project.
I would also like to thank Andrew Hanlon for walking me through the construction of the monomial admissible Fukaya-Seidel category, and an anonymous reviewer whose comments and suggestions greatly improved the exposition of this article. 
This project has also benefited from useful conversations with Diego Matessi, Nick Sheridan, and Ivan Smith.
Finally, I am especially grateful to Paul Biran and ETH Z\"urich for their hospitality while hosting me.

\end{acknowledgements}

\section{Background: Tropical Geometry, Lagrangian Cobordisms, and HMS for Toric Varieties}
\label{sec:background}
\subsection{Tropical Geometry}
\label{subsec:tropicalgeometry}
The tropical semiring $(\RR, \oplus , \odot)$ is the set $\RR\cup \{+\infty\}$ equipped with the following two binary operations
\begin{align*}
	x_1\oplus x_2 = & \min(x_1, x_2) \\
	x_1\odot  x_2 =  & x_1 + x_2.
\end{align*}
These two operations are called \emph{tropical plus} and \emph{tropical times} respectively, and they obey the distributive law.
Tropical polynomials of multiple variables describe  piecewise linear concave  functions $\phi: Q:=\RR^n\to \RR$ of rational slope.
It will frequently be useful for us to use the following characterization of tropical linear polynomials. 
\begin{prop}
	The tropical polynomials are exactly the piecewise linear concave functions with $d\phi(x)\in T^*_\ZZ\RR^n$ at all points where $\phi$ is differentiable.
\end{prop}
One can approximate tropical polynomials with regular polynomials via logarithms and the estimates
\begin{align}
	-\log_{1/q}(q^{x_1}+q^{x_2})\sim & \min({x_1}, {x_2})=x_1\oplus x_2 \\
	-\log_{1/q}(q^{x_1}q^{x_2})=     & x_1+x_2= x_1\odot x_2
	\label{eq:tropical}
\end{align}
for $q$ goes to $0$. 
We'll frequently describe tropical polynomials of two variables in terms of their tropical varieties by drawing a planar graph whose faces describe the domains of linearity of $\phi$.  This graph generalizes in higher dimensions to a stratification of $Q$ which describes many of the combinatorial properties of a tropical polynomial.
\begin{definition}
	Let $\phi: Q \to \RR$ be a tropical polynomial. Each monomial term in $\phi$ can be labelled by its exponent $v\in \ZZ^n$.  The \emph{linearity stratification} of $Q$ is the stratification 
	\[
		\emptyset \subset Q_0\subset \cdots \subset Q_n,
	\]
	where $p\in Q_k$ if and only if there is no $k+1$ dimensional open subset of an affine subspace $A\subset Q$, with  $p\in A$ on which the restriction $\phi|_A$ is a $k+1$-affine map.
	 Each stratum will be denoted $\underbar U_{\{v_i\}}$, where $\{v_i\}$ is the collection of monomial terms which achieve their minimum along the strata. We define the \emph{tropical variety} of $Q$ to be $V(\phi)=Q_{n-1}$, which describes the locus of non-linearity of $\phi$.\footnote{The $V$ in $V(\phi)$ should either stand for valuation, or variety.}
\end{definition}
One interpretation to the approximation given in equation (\ref{eq:tropical}) is that when $f:(\CC^*)^n\to \CC$ is a Laurent polynomial, the tropical variety $V(\phi)$ provides a dominating term approximation of $\val(f^{-1}(0))$.

There is an involution on smooth concave functions $f$ with convex domains $\Delta$ given by the Legendre transform. An analogous involution exists in the setting of tropical polynomials.
\begin{definition}
	Let $\phi= \bigoplus_{v} a_v\odot x^v$. 
	Let $\Delta_\phi^\ZZ$ be the set of integer points $v$ for which $\phi$ matches the monomial $a_v\odot x^v$ on an open subset.
	Define the \emph{Newton polytope} $\Delta_\phi\subset T^*_0Q$ to be the convex hull of $\Delta_\phi^\ZZ$. We define the \emph{Legendre transform $\check \phi(v)$} to be the minimal-fit concave piecewise linear function  to the data
	\[
		\check \phi(v)= a_v \text{ for all $v\in \Delta_\phi^\ZZ$}.
	\] 
	We will denote the linearity stratification induced by $\check \phi$ on the Newton polytope as 
	\[
		\Delta_\phi^\ZZ= \Delta^n_\phi\subset \cdots \subset \Delta^0_\phi= \Delta_\phi.
	\] 
	We will denote\footnote{It's worth pointing out that this naming convention is very different than the one for the strata of $Q$, but will make notation a lot easier in the future.} the stratum with vertices $v_i$ by $\underbar U^{\{v_i\}}$. 
\end{definition}
See \prettyref{fig:sometropicalcurves} for examples.
The Newton polytope is a lattice polytope which is determined by the leading order behavior of the tropical polynomial $\phi$. It will also be convenient for us to interpret this Newton polytope as the image of $d\phi$ under the projection $\pi_0: T^*\RR^n\to T^*_0\RR^n$ 
\[
	\Delta_\phi=\text{Hull}(\{\pi_0\circ d\phi(x)\;|\; \text{$\phi$ is differentiable at $x$}\}).
\]
As the Legendre transform contains the data of the coefficients of $\phi$, the polynomials $\phi$ and $\check \phi$ determine each other completely.
This relation is also reflected in the duality between the stratifications  $\underbar U_{\{v_i\}}$ and $\underbar U^{\{v_i\}}$. 

\begin{definition}
	We say that a non-maximal stratum $\underbar U_{\{v_i\}}$ is smooth if $\underbar U^{\{v_i\}}$ is a standard simplex.
	The self-intersection number of a strata $\underbar U_{\{v_i\}}$ is defined as the number of interior lattice points of $\underbar U^{\{v_i\}}$.
\end{definition}
When $\phi$ has no self-intersections, then $\Delta_\phi\cap \ZZ^n= \Delta^\ZZ_\phi$. 
When a tropical variety comes as the tropicalization of a family of complex curves, the self-intersection number gives the genus of the family of curves which degenerate in the family.
In our constructions of tropical Lagrangians, this self-intersection number will give the number of self-intersection points of our tropical Lagrangian (see discussion following \prettyref{def:tropicallagrangian}). 

\begin{example}
	We now look at a few examples which demonstrate the difference between smooth, non-smooth, and self-intersecting tropical hypersurfaces.
	Consider the tropical polynomial
	\[
		\phi_{T^2}^0(x_1,x_2) =x_1\oplus x_2\oplus (x_1x_2)^{-1}.
	\]
	This has 3 domains of linearity, where each of the monomial terms dominate. Notice that the dual stratum $U^{x_1, x_2, (x_1x_2)^{-1}}$ is not a primitive simplex, and so this vertex is non-smooth (see \prettyref{fig:nonsmoothtropicalelliptic}).
	As $(0,0)\in U^{x_1, x_2, (x_1x_2)^{-1}}$ is an interior lattice point of the simplex, this stratum has 1-self intersection.

	If we modify the coefficients of the monomials in the tropical polynomial to
	\[
		\phi_{T^2}^c(x_1, x_2) = x_1\oplus x_2 \oplus( c\odot x_1^{-1}x_2^{-1})\oplus 0
	\]
	with $c>0$, we get a smooth tropical curve  instead (see \prettyref{fig:smoothtropicalelliptic}).
	The dual stratification is a triangulation of the Newton polytope by primitive simplices, so this is an example of a smooth tropical curve with no self intersections.
	\label{exam:nonsmoothpuncturedtorus}

	Finally, we look at an example which highlights the difference between self-intersections and smoothness.
	Consider the tropical polynomial
	\[
		\phi_+(x_1, x_2)=1\oplus x_1 \oplus x_2 \oplus x_1x_2
	\]
	The moment polytope and curve are drawn in \prettyref{fig:smoothtropicalcross}.
	The stratum $U^{1, x_1, x_2, x_1x_2}$ is not smooth, as it is not a simplex.
	However, the stratum does not contain any self intersections either, as its interior is lattice point free. 
\end{example}
\begin{figure}
	\begin{subfigure}{\linewidth}
		\centering
		\includegraphics[scale=1]{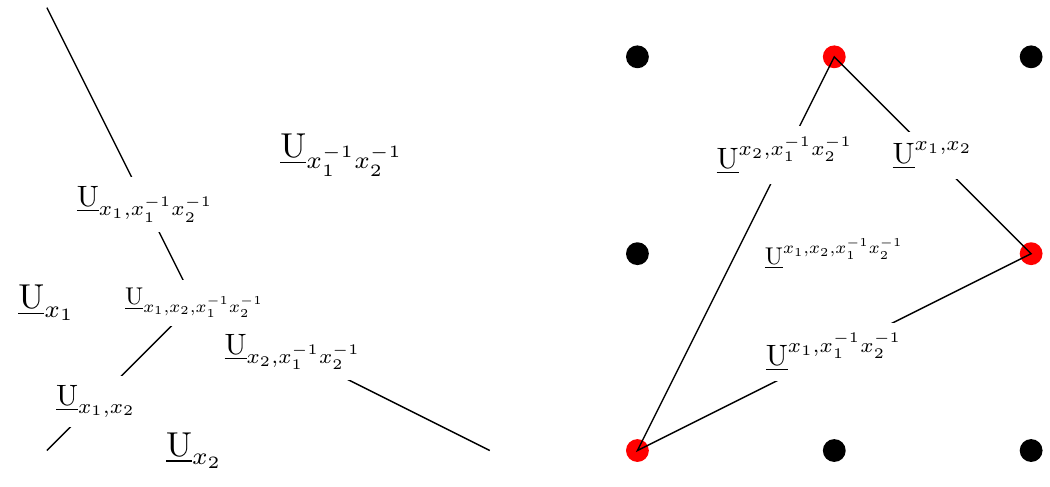}
		\caption{The non-smooth punctured immersed sphere with 3 punctures, $\phi_{T^2}^0(x_1,x_2) =x_1\oplus x_2\oplus (x_1x_2)^{-1}.$}
		\label{fig:nonsmoothtropicalelliptic}
	\end{subfigure}

	\begin{subfigure}{\linewidth}
		\centering
		\includegraphics[scale=1]{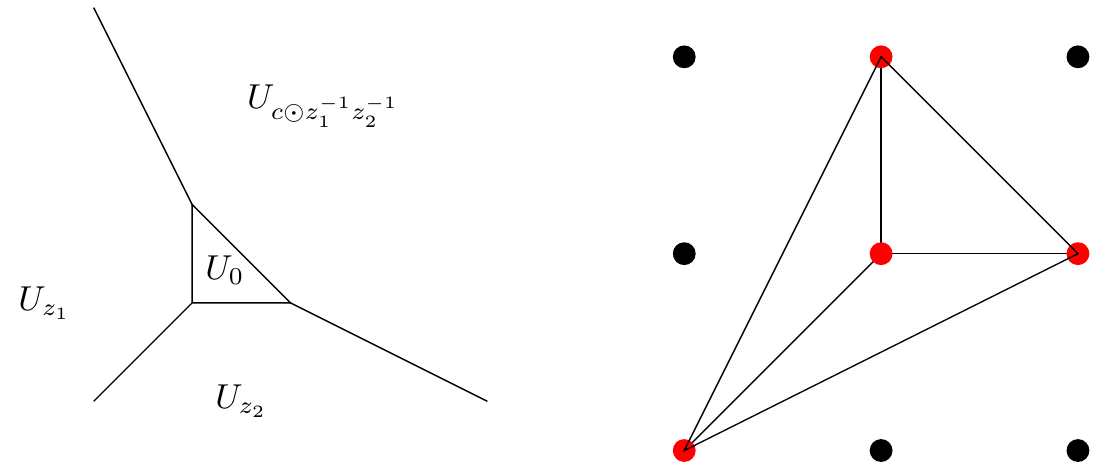}
		\caption{The smooth tropical 3-punctured torus, $\phi_{T^2}^c(x_1, x_2) = x_1\oplus x_2 \oplus( c\odot x_1^{-1}x_2^{-1})\oplus 0$. }
		\label{fig:smoothtropicalelliptic}
	\end{subfigure}

	\begin{subfigure}{\linewidth}
		\centering
		\includegraphics[scale=1]{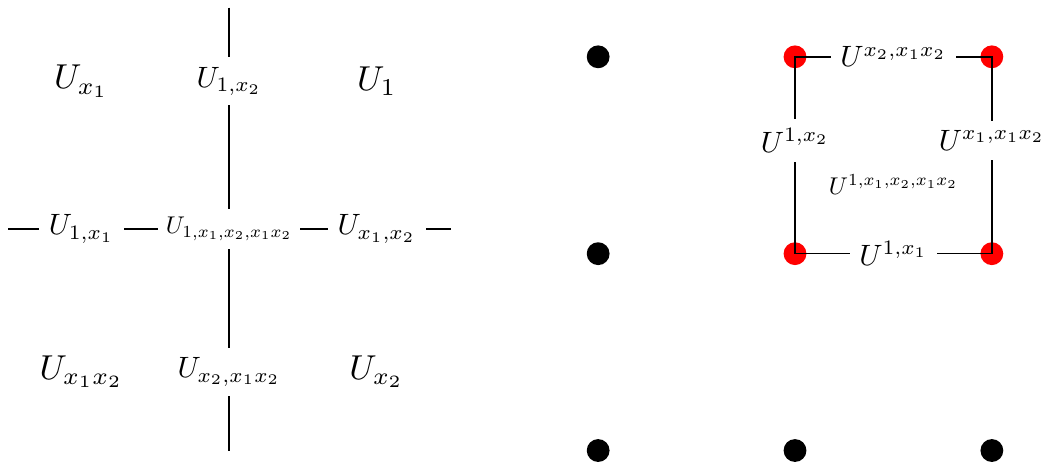}
		\caption{Non-smooth, self-intersection free sphere with 4 punctures, $\phi_+(x_1, x_2)=1\oplus x_1 \oplus x_2 \oplus x_1x_2$. }
		\label{fig:smoothtropicalcross}
	\end{subfigure}
	\caption{Some tropical curves, with $\Delta^\ZZ_\phi$ highlighted in red.}
	\label{fig:sometropicalcurves}
\end{figure}
\subsection{Lagrangian Surgery and Cobordisms}
\label{subsec:lagsurgery}
The bulk of this paper will revolve around constructing new Lagrangian submanifolds. Our most-used tool for these constructions is Lagrangian surgery, which allows us to make local modifications to a Lagrangian submanifold to obtain new Lagrangians.
By the Weinstein neighborhood theorem, we can locally model the transverse intersection of two Lagrangians by the zero section and cotangent fiber of $T^*\RR^n$. The model  \emph{ Lagrangian neck  }in $T^*\RR^n$ is the Lagrangian parameterized by
\begin{align*}
	I\times S^{n-1}\into & \RR^n\times \RR^n                                         \\
	(t, \hat u) \mapsto  & (r \cdot e^t \hat u,r \cdot  e^{-t}\hat u).
\end{align*}
Notice that this Lagrangian neck is asymptotic to the disjoint union of the zero section and $T^*_0\RR^n$. The parameter $r$ is called the \emph{neck radius} of the smoothing. Given a Lagrangian $L$ with a transverse self-intersection, one can replace a neighborhood of the self-intersection with a Lagrangian neck. In summary:
\begin{theorem}[\cite{polterovich1991surgery}]
	Let $L$ be a (not necessarily connected) Lagrangian submanifold with a transverse self-intersection point $p$. 
	Then there exists a smooth Lagrangian $L'_\nw$ which agrees with $L$ outside of the intersection point, and is modeled on the Lagrangian neck in a neighborhood of $p$.
\end{theorem}
There are many ways of doing this smoothing depending on the choice of neck inserted, however the Hamiltonian isotopy class of the surgery is dependent only on a single parameter $\nw$ called the neck width, which measures the flux swept by the Lagrangian isotopy as one decreases the neck radius to 0. 
Given Lagrangians $L_0$ and $L_1$ intersecting transversely at a single point $q$, we denote by $L_0\#_q^\nw L_1$ the Lagrangian connect sum, which is obtained by inserting a neck of width $\nw$ at a neighborhood of the intersection point $q$.
\footnote{Here, the neck we choose is consistent with the orientations from \cite{auroux2014beginner}, and opposite of the orientation chosen in \cite{biran2013lagrangian}.}
When the width of the neck inserted is unimportant, we will simply write $L_0\#_q L_1$. 
Lagrangian surgery also has an interpretation as an algebraic operation in the Fukaya category.
\begin{theorem}[\cite{fukaya2007chapter10}]
	Let $L_0, L_1$ be unobstructed Lagrangians intersecting at a unique point $q$. Then there is an exact triangle
	\[
		L_1\xrightarrow{q}L_0\to  L_0\#_q^\nw L_1
	\] 
	in the Fukaya category. \label{thm:surgerycone}
\end{theorem}
The proof of this theorem comes from a comparison of holomorphic triangles with corner at $q$, and holomorphic strips in $L_0\#_q^\nw L_1$ which are obtained from rounding this corner.

The topological operation of surgery can be understood via cobordisms. In the symplectic world,  there is an analogous notion of \emph{Lagrangian cobordism} relating  Lagrangian surgeries.
\begin{definition}[\cite{arnol1980lagrange}]
	Let $\{L^+_{i}\}_{i=0}^{k-1}, L^-$ be Lagrangian submanifolds of $X$. A \emph{Lagrangian cobordism} between $\{L^+_i\}$ and  $L^-$ is a Lagrangian $K\subset X\times \CC$ which satisfies the following conditions:
	\begin{itemize}
		\item \emph{Fibered over ends:} There exists constants $\{c^+_i\}, c^-\in \RR$ with $c^+_i< c^+_{i+1}$, as well as constants $t^-< t^+ \in \RR$ such that 
		\[
			K\cap \{(x, z)\;|\; \Re(z)\geq t^+\}=\bigsqcup_{i} L^+_i\times \{(t+ic_i^+)\;|\; t\geq t^+\} 
		\]
		\[
			K\cap \{(x, z)\;|\; \Re(z)\leq t^-\}= L^-\times \{(t+ic^-)\;|\; t\leq t^-\}.
		\]

		\item\emph{Compactness:}  The projection $\Im_z: K\to i \RR \subset \CC$ is bounded. 
	\end{itemize}
	We denote such a cobordism $K:(L_0^+, \ldots, L_{k-1}^+)\rightsquigarrow L^-$.
	\label{def:cobordism}
\end{definition}
\begin{remark}
	We follow the \emph{cohomological grading} of \cite{haug2015torus}, where a cobordism between $(L_0^+, \ldots, L_{k-1}^+)\rightsquigarrow L^-$ has input ends $L_0^+, \ldots, L_{k-1}^+$ with positive real values $\{c_i^+\}$, and end $L^-$ with negative real value $c^-$.  
	This is opposite to the convention from \cite{biran2013fukayacategories}.
	This can be slightly confusing when considering Lagrangian cobordisms, as the ``domain'' end of the cobordism is on the right in its projection to $\CC$. 
	\label{rem:cobordismgrading}
\end{remark}
Some simple examples of Lagrangian cobordisms include the trivial cobordism $L\times \RR$, or the suspension of a Hamiltonian isotopy.  Given Lagrangians $L_0, L_1$ intersecting transversely at a single point $q$, there exists a surgery trace Lagrangian cobordism $(L_0, L_1)\rightsquigarrow L_0\#_q L_1$. Just as Lagrangian surgery gave us a way to understand $L_0\#_q L_1$ as a mapping cone, there is a broad-reaching theorem which tells us how to relate cobordant Lagrangians as objects of the Fukaya category.
\begin{theorem}[\cite{biran2013lagrangian}]
	Let $K:(L_i^+)_{i=0}^{k-1}\rightsquigarrow L^-$ be an embedded monotone Lagrangian cobordism.
	Then there are $k$ objects $Z_0, \ldots, Z_{k-1}$ in the Fukaya category, with $Z_0=L_0^+$ and $Z_{k}\simeq L^-$ which fit into $k$ exact triangles 
	\[
		L_i^+\to Z_{i-1}\to Z_i\to L_i^+[1].
	\]
	In particular when $k=2$, we have an exact triangle
	\[
		L_1^+\to L_0^+\to L^-.
	\]
	In the case where $k=1$, we have an isomorphism
	\[
		0 \to L^+\to L^-\to 0.	
	\]
	\label{thm:cobordisms}
\end{theorem}
This can be restated as a relation between the Lagrangian cobordism category of $X$, and a category which describes triangular decompositions of objects in the Fukaya category.
The proof of this theorem computes for test objects $L\in \Fuk(X)$ the Floer cohomology $\CF(L\times\RR,K)$ in two different ways.
\subsection{Mirror Symmetry for Toric Varieties}
\label{subsec:toricbackground}
Let $\check X_\Sigma$ be a toric variety given by the fan $\Sigma\subset \RR^n= Q$. We review some notation and concepts from \cite{cox2011toric} related to line bundles, divisors and tropical geometry.
Each lattice generator  $v$ of a ray of $\Sigma$ gives a torus equivariant divisor $D_v$ of $\check X_\Sigma$. In the setting where $\check X_\Sigma$ is smooth we may express any linear equivalence class of a divisor as a sum  $\sum_{v\in \Sigma} a_v D_v$.
An \emph{integral support function} for $\Sigma$ is a function $\phi: Q\to \RR$ which is linear on each cone of $\Sigma$, and  is integral on the fan in the sense that  $\phi(\Sigma \cap \ZZ^n)\subset \ZZ$.
To each divisor class $[D]$, we may associate an integral support function $\phi_{[D]}$ which is determined by the values
\[
	\phi_{[D]}(v)=a_v.
\]
Properties of the line bundle $\mathcal O(D)$ can be read from the support function: the line bundle is \emph{base point- free} whenever $\phi_{[D]}$ is concave, and \emph{ample} if and only if $\phi_{[D]}$ is strictly concave. The function $\phi_{[D]}$ is piecewise linear, so when  $D$ is base point free the function $\phi_{[D]}$ is a maximally degenerate tropical polynomial (in the sense that every stratum of the tropical variety contains the origin).
The tropical variety of the support function of a base point free line bundle can be related to the valuation projection of the corresponding divisor of the line bundle.  Let $D$ be a base point free divisor transverse to the toric anticanonical divisor, and let $\Delta_{\phi_{[D]}}$ be the Newton polytope of $\phi_{[D]}$. Then over the open torus of $\check X_\Sigma$, each choice of constants $c_v$ defines a polynomial
\begin{align*}
	f_D:(\CC^*)^n\to& \CC\\
	f_{D}:=&\sum_{v\in \Delta_{\phi_D}} c_v z^v
\end{align*}
which gives a section of the line bundle over the compactification $\mathcal O_D\to \check X_\Sigma \supset (\CC^*)^n.$ There exists a choice of constants so that $f_D^{-1}(0)=D$.
We now consider Laurent polynomials $f_D=\sum_{v\in \Delta_{\phi_D}} c_v z^v$, where the constants $c_v$ are elements of the Novikov field. The valuation projection of the variety $f_D=0$ is a tropical hypersurface defined by a tropical polynomial
$\phi_D$, which we call the tropicalization of $f_D$.  $\phi_D$ is a deformation of $\phi_{[D]}$.
In the complex setting the tropical variety $V(\phi_{D})\subset \RR^n$ approximates the image of $D$ under the moment map projection $\val: \check X_\Sigma \to \RR^n$.
The observation that the strata of the tropical variety $V(\phi)$ meet the toric boundary of the moment polytope of $X_\Sigma$ transversely is compatible with the existence of a compactification for the  variety $f^{-1}(0)\subset (\CC^*)^n$ inside of $\check X_\Sigma$.

Mirror symmetry for toric varieties is based on an understanding of how compactifications modify the mirror construction.
It is an expectation in mirror symmetry that compactification on $\check X$ corresponds to the incorporation of a superpotential $ W:  X\to \CC$ in the mirror and vice versa.
One proposed method for constructing mirror spaces for toric varieties is to consider $\check X_\Sigma$ as a compactification of $\check X_\Sigma\setminus D_\Sigma= (\CC^*)^n$, where $D_\Sigma$ is the toric anticanonical divisor.
A choice of symplectic form for $\check X_\Sigma$ picks out the coefficients of a Laurent polynomial,
\[
	W_\Sigma=\sum_{v\in \Sigma}c_v z^v: (\CC^*)^n\to \CC
\]
the \emph{Hori-Vafa} superpotential for $\check X_\Sigma$ \cite{hori2000mirror}.
\begin{notation}
	For the remainder of this paper, $X=\check X = (\CC^*)^n, Q=\RR^n$ , $\check X_\Sigma$ is a toric variety determined by a fan $\Sigma$, and $W_\Sigma: X\to \CC$ is the mirror Hori-Vafa superpotential.
\end{notation}
This Hori-Vafa superpotential provides a taming condition for non-compact Lagrangian submanifolds of $X$.
We use the notion of a \emph{monomial division} as a taming condition for its particularly clean description in $X$.
\begin{definition}[\cite{hanlon2018monodromy}]
	Let $W_\Sigma:X\to \CC$ be a Laurent polynomial whose monomials are indexed by the rays of a fan $\Sigma$. A \emph{monomial division} $\Delta_\Sigma$ for $W_\Sigma=\sum_{\alpha\in A} c_\alpha z^\alpha$ is an assignment of a closed set $C_\alpha \subset Q$ to each monomial in $W_\Sigma$ such that the following conditions hold:
	\begin{itemize}
		\item
		      The $C_\alpha$ cover the complement of a compact subset of $Q=\RR^n$
		\item
		      There exist constants $k_\alpha \in \RR_{>0}$ so that for all $z$ with $\val(z)\in C_\alpha$ the expression 
		      \[
			      \max_{\alpha\in A} (|c_\alpha z^\alpha|^{k_\alpha})
		      \]
		      is always achieved by $|c_\alpha z^\alpha|^{k_\alpha}$.
		\item
		      $C_\alpha$ is a subset of the open star of the ray $\alpha$ in the fan $\Sigma$.
	\end{itemize}
	A Lagrangian $L\subset X$ is \emph{$\Delta_\Sigma$-monomially admissible} if over $\val^{-1}(C_\alpha)$ the argument of $c_\alpha z^\alpha$ restricted to $L$ is zero outside of a compact set. \label{def:admissiblitycondition}
\end{definition}
Given $W_\Sigma$, there is often a preferred type of monomial division, the \emph{tropical division}, with covering regions defined by
\[C_\alpha:=\{|c_\alpha z^\alpha|\geq (1-\delta) \max_{\beta\in A}(|c_\beta z^\beta|)\]
for some fixed $\delta\in [0, 1]$.
The data of a monomial division allows the construction of a \emph{monomial admissible Fukaya-Seidel category}.
\begin{theorem}[\cite{hanlon2018monodromy}]
	Given $\Delta_\Sigma$ a monomial division for $W_\Sigma$, there exists an $A_\infty$ category $\Fuk_{\Delta_\Sigma}(X, W_\Sigma)$ whose objects are $\Delta_\Sigma$-admissible Lagrangians, and whose morphism spaces are defined by localizing an $A_\infty$ pre-category $\Fuk^\to(X)$ with morphisms:
	\[
	\hom(L_0, L_1)= \CF(L_0, \theta(L_1)).
	\]
	Here $\theta$ is an admissible Hamiltonian perturbation.
	The  higher composition maps $m^d$ in this precategory are given by counts of punctured holomorphic disks.
\end{theorem}

The Lagrangians considered in the setting of \cite{hanlon2018monodromy} do not bound holomorphic disks.
\begin{theorem}[\cite{hanlon2018monodromy}]
	Let $\check X_\Sigma$ be a smooth complete toric variety with Hori-Vafa superpotential $W_\Sigma$. Let $Tw^\pi \mathcal P_{\Delta_\Sigma}(X, W_\Sigma)$ be the category of twisted complexes generated by the full subcategory of $\Fuk_{\Delta_\Sigma}(X, W_\Sigma)$ consisting of tropical Lagrangian sections. The $A$ and $B$-models
	\[
		Tw^\pi\mathcal P_{\Delta_\Sigma}(X, W_\Sigma)\simeq D^b\Coh(\check X_\Sigma)
	\]
	are quasi-equivalent $A_\infty$  categories.
	\label{thm:mirrorfortoric}
\end{theorem}
The equivalence is proven with  \prettyref{thm:hmstorics}, which characterizes the Floer cohomology between tropical Lagrangian sections (\prettyref{def:tropicallagrangiansection}).
\section{Construction of Tropical Lagrangians in $X$}
\label{sec:hypersurfaces}
The goal of this section is to  construct for each tropical polynomial $\phi: Q\to \RR$ a Lagrangian $L(\phi)$ whose projection $\val(L(\phi))$ is $\epsilon$-close to $V(\phi)$ in the Hausdorff metric. Our construction will be rooted in the language of Lagrangian cobordisms, giving us a path to prove a homological mirror symmetry statement for $L(\phi)$.
We additionally prove that the Lagrangian $L(\phi)$ is an unobstructed object of the Fukaya category. 
\subsection{Surgery Profiles}
We will need an explicit surgery profile to build $L(\phi)$.
\begin{prop}
	\label{prop:surgeryprofile}
	Let $f_0:\RR^n\to \RR$ be the constant function $f_0=0$, and let $f_1:\RR^n\to \RR$ be a smooth convex function.
	Let $U$ be the region where  $df_1=0$
	and suppose we have normalized $f_1$ so that $f_1(U)=0$. 
	Additionally, suppose that for given $\epsilon>0$, there exists a constant $c_\epsilon$ so that $f_1(x)< 2c_\epsilon$ implies that $x\in B_\epsilon(U)$. 

	Consider the Lagrangian sections $L_0=df_0$, and $L_1=df_1$ in $T^*\RR^n$. There exists a Lagrangian  $L_0\#_U L_1$  in a small neighborhood of the symmetric difference
	\[
		L_0\#_U L_1 \subset B_\epsilon( (L_0\sqcup L_1)\setminus (L_0 \cap L_1)).
	\]
	Furthermore, there exists a Lagrangian cobordism $K$ with ends $(L_0, L_1)\rightsquigarrow L_0\#_U L_1.$
	 We call $L_0\#_U L_1$ the \emph{Lagrangian surgery at $U$ }.
\end{prop}
\begin{proof}
	We first give a description of a Lagrangian $L_0\#_U L_1$ which satisfies the desired properties.
	By convexity, $f_1>0$ on the complement of $U$.
	Let $r_\epsilon:\RR_{>\epsilon}\to \RR$ and $s_\epsilon: \RR_{>\epsilon}\to \RR$ be functions satisfying the following properties:
	\begin{itemize}
		\item
		      $r_\epsilon(t)=t$ for $t\geq 2c_\epsilon$ and $s_\epsilon(t)=0$ for $t\geq 2c_\epsilon$.
		\item
		      $r_\epsilon'(c_\epsilon)=s_\epsilon'(c_\epsilon)=\frac{1}{2}$
		\item
		      $r_\epsilon(t)$ is convex, while $s_\epsilon(t)$ is concave.
		\item
		      The concatenation of curves $(t, r_\epsilon'(t))$ and $(t, s_\epsilon'(t))$ is a smooth plane curve.
	\end{itemize}

	The profiles of these functions are drawn in \prettyref{fig:surgeryprofile}. 
	A quantity which we will later use is the \emph{neck width} of the surgery profile curves, which is defined as 
	\[\nw_{r_\epsilon,s_\epsilon}:=-r_{\epsilon}(c_{\epsilon})-s_{\epsilon}(c_\epsilon).\]
	Consider the Lagrangian submanifolds given by the graphs
	\[
		d(r_\epsilon \circ f_1) \;\;\;\;\;\;\; d(s_\epsilon  \circ f_1)
	\] defined as sections over the domain $f_1\geq c_\epsilon$.
	
	The union of these two charts is a smooth Lagrangian submanifold, which is our definition of $L_0\#_U L_1$.
	The profile is only defined where $f_1\geq c_\epsilon$, so  $L_0\#_U L_1$ is disjoint from the set $L_0\cap L_1$.
	As $r_\epsilon$ is the identity for $f_1>2c_\epsilon$, we have that $d(r_\epsilon \circ f_1)=L_1$ on the region where $f_1(x)>2c_\epsilon$.
	A similar statement can be made about $s_\epsilon$.
	These observations (along with the choice of $c_\epsilon$) give us that $L_0\#_U L_1$ is contained a small neighborhood of the symmetric difference. 
	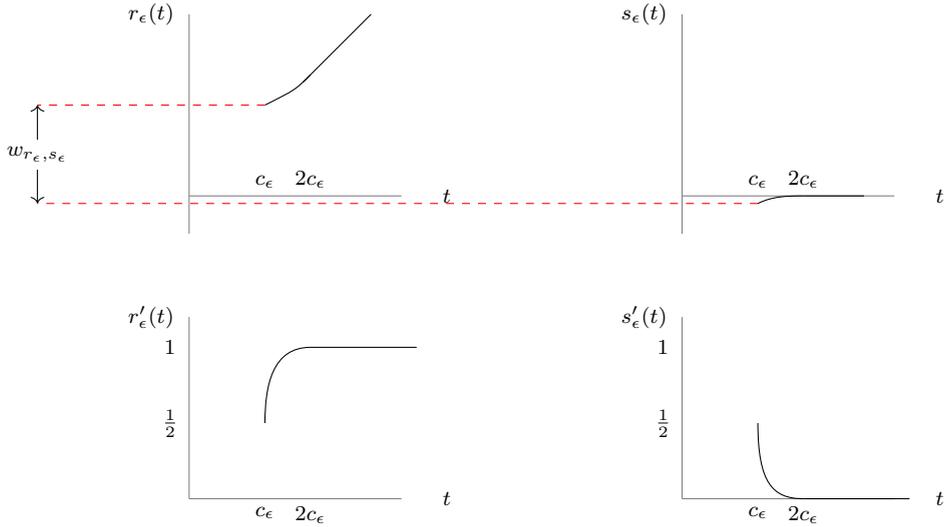
\begin{figure}
			\begin{tikzpicture}
	\begin{scope}[]
		\draw[gray] (-1, 0)-- (1.8,0);
		\draw[gray] (-1,2.4) -- (-1,-0.5);
		\node at (-1.5,2.4) {$s_\epsilon(t)$};
		\node at (2.4,0) {$t$};

		\node[above] at (0.6,0) {$2c_\epsilon$};
	
	\node[above]at (0,0) {$c_\epsilon $};
		
		\draw (1.4,0) -- (0.6,0);
	\draw (0.6,0) .. controls (0.4,0) and (0.2,0) .. (0,-0.1);
	
	\end{scope}
	\begin{scope}[shift={(0,-4)}]
		\draw[gray] (-1,0)-- (1.8,0);
		\draw[gray] (-1,2.4) -- (-1,0);
		\node at (-1.5,2.4) {$s_\epsilon'(t)$};
		\node at (2.4,0) {$t$};

		\node[below] at (0.6,0) {$2c_\epsilon$};
	
	\node[below] at (0,0) {$c_\epsilon$};
		
		\draw (2,0) -- (0.6,0);
	\draw (0.6,0) .. controls (0.2,0) and (0,0.2) .. (0,1);
	
	\end{scope}
	\node at (-1.25,-2) {1};
	\node at (-1.25,-3) {$\frac{1}{2}$};

	\begin{scope}[shift={(-6.5,0)}]
	
	\begin{scope}[]
		\draw[gray] (-1,0)-- (1.8,0);
		\draw[gray] (-1,2.4) -- (-1,-0.5);
		\node at (-1.5,2.4) {$r_\epsilon(t)$};
		\node at (2.4,0) {$t$};

		\node[above] at (0.6,0) {$2c_\epsilon$};
		
		\draw (1.4,2.4) -- (0.5,1.5);
	\draw (0.6,1.6) .. controls (0.4,1.4) and (0.4,1.4) .. (0,1.2);
	
	\node[above] at (0,0) {$c_\epsilon $};
	\end{scope}
	\begin{scope}[shift={(0,-4)}]
		\draw[gray] (-1,0)-- (1.8,0);
		\draw[gray] (-1,2.4) -- (-1,0);
		\node at (-1.5,2.4) {$r_\epsilon'(t)$};
		\node at (2.4,0) {$t$};

		\node[below] at (0.6,0) {$2c_\epsilon$};
		
	\node[below] at (0,0) {$c_\epsilon $};
		\draw (2,2) -- (0.6,2);
	\draw (0.6,2) .. controls (0.1,2) and (0,1.5) .. (0,1);
	
	\end{scope}
	\node at (-1.25,-2) {1};
	\node at (-1.25,-3) {$\frac{1}{2}$};
	
	\end{scope}
	\draw[dashed, red] (-6.5,1.2) -- (-9.5,1.2);
	\draw[dashed, red] (0,-0.1) -- (-9.5,-0.1);
	\draw[<->] (-9.5,1.2) -- (-9.5,-0.1);
	\node[fill=white] at (-9.5,0.55) {$\nw_{r_\epsilon,s_\epsilon}$};
	\end{tikzpicture}
		\caption{Some profiles for constructing $L_0\#_U^\epsilon L_1$}
		\label{fig:surgeryprofile}
	\end{figure}

	It remains to show that $L_0 , L_1$ and $L_0\#_U L_1$ fit into a cobordism. This cobordism will be constructed as a Lagrangian surgery in one dimension higher.  Let $\tilde f_0: \RR^{n+1}\to \RR$ be the constant zero function, and let function
	\begin{align*}
		\tilde f_1: \RR^n\times \RR \to & \RR          \\
		(x, t) \mapsto                  & f_1(x)+ g(t)
	\end{align*}
	where the function $g(t)$ satisfies the following properties:
	\begin{itemize}
		\item
		      $g(t)$ is convex
		\item
		      $dg|_{t<-\epsilon}=0$ and $dg|_{t>0}=1$.
	\end{itemize}
	We will now take the surgery of sections $\tilde L_0=  d\tilde f_0$ and $ \tilde L_1 =d\tilde f_1$.
	Let $U$ be the region where $df_1=0$.  $\tilde L_0$ and $\tilde L_1$ agree on the region $ \tilde U=U\times (-\infty, -\epsilon)$.
	As the intersection over this region is defined by the intersection of convex primitives, we may use our previous construction to define the surgery cobordism
	\[
		K:=\tilde L_0\#_{\tilde U} \tilde L_1.
	\]
	Since $K$ agrees with $\tilde L_0$ and $\tilde L_1$ outside of a small neighborhood of $\tilde U$,
	\[
		K|_{t>0}=  ((L_0\times \{0\} )\sqcup( L_1\times \{dt\}))\times \RR_{t>0}.
	\]
	As the function $g(t)$ is constant on each $t$-fiber, we obtain that $K|_{t=-\epsilon}=L_0\#_{U} L_1$ and conclude
	\[
		K|_{t<-\epsilon} =(L_0\#_U L_1)\times\{0\}\times \RR_{t<- \epsilon}.
	\]
	These are the conditions required for $K$ to be a Lagrangian cobordism between $L_0, L_1$ and $L_0\#_U L_1$.
\end{proof}
\begin{figure}
	\centering
	\includegraphics{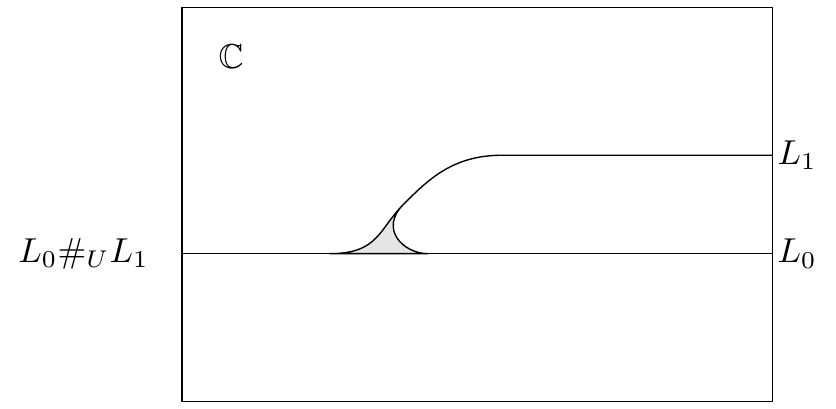}
	\caption[The projection of the surgery cobordism to the $\CC$-parameter.]{The projection of the surgery cobordism to the $\CC$-parameter.
	The curve defining the upper boundary of this projection is parameterized by $z=t+i\frac{dg}{dt}$.}
\end{figure}
\begin{remark}
	We could have chosen $f_1$ to be concave and still have had a surgery construction. However the existence of a surgery cobordism depends on the convexity of $f_1$.
	If instead we had started with a concave primitive $f_1$, the function
	\begin{align*}
		\tilde f_1: \RR^n\times \RR \to & \RR \\
		(x, y) \mapsto& f_1(x)+ g(t)
	\end{align*}
	is neither concave or convex, and we are unable to construct our cobordism using the surgery profile previously described.
	This manifests itself as an ordering on the ends of the cobordism.
	In particular, there is usually no Lagrangian cobordism with ends $(L_1, L_0)$ and $(L_0\#_UL_1)$. 
\end{remark}
In the setting where $U$ is a closed ball, this is nothing different than a special choice of neck on the standard Lagrangian connect sum cobordism. 
The construction above provides an alternative definition to the surgery trace cobordism considered in \cite{biran2013lagrangian}.
In the non-compact setting, we get an interpretation of what it means to take the connect sum of Lagrangians which agree on a non-compact set.

This operation was dependent on a number of choices--- in particular, a choice of profile function $r_\epsilon$ and $s_\epsilon$. 
Different choices for these parameters give isotopic Lagrangian submanifolds. 
However, just as in the setting of transverse surgery of Lagrangian submanifolds, the Hamiltonian isotopy class of these Lagrangians can be made independent of choices up to the neck width.

Recall, if $i_t:L\times [0, t_0]\to X$ is a Lagrangian isotopy, the flux of the isotopy is the class 
\begin{align*}
	\Flux_{i_t}\in H^1(L) & & \Flux_{i_t}(c):= \int_{S^1\times[0, t_0]}(i_t\circ c)^* \omega.
\end{align*}

We can use this to associate a flux class to a Lagrangian surgery. 
Fix choices of profile functions $r_{\epsilon_0}, s_{\epsilon_0}$.
Extend this to a smooth family of surgery profiles  $r_{\epsilon}, s_{\epsilon}$ , dependent on the choice of parameter $\epsilon\in(0, \epsilon_0)$.
As $\epsilon$ approaches zero, $r_\epsilon$ approaches the identity and $s_{\epsilon}$ approaches the constant zero function.
This smooth family of surgery profiles defines a Lagrangian homotopy $i_\epsilon:L_0\# L_1\times(0,\epsilon_0)\to X$.
\begin{prop}
	The flux class $\Flux_{i_\epsilon}$ is determined by the neck width $\nw_{r_\epsilon, s_\epsilon}$.
	\label{prop:neckwidth}
\end{prop}
\begin{proof}
	The flux cohomology class is only dependent on $r_{\epsilon_0}, s_{\epsilon_0}$, and not the extension of these to families $r_{\epsilon}, s_{\epsilon}$, as we can interpolate between any two choices of extensions, and the Flux class is invariant under isotopies of isotopies relative ends. 
	We now give an explicit computation of this Flux class in terms of the profile function.
	Recall that $L_0\#_U L_1$ is parameterized  the $d(s_{\epsilon}\circ f)$ and $d(r_{\epsilon}\circ f)$ charts which are identified with $\RR^n\setminus \{f_1< c_\epsilon\}$.
	We parameterize the neck of the surgery $N$   by specifying maps into these two charts. 
	Let 
	\begin{align*} \alpha_s :\partial U \times [-t_0, 0]\to\{c_\epsilon\leq f_1 \leq 2\epsilon\} && \alpha_r: \partial U \times [0, t_0]\to\{c_\epsilon\leq f_1 \leq 2\epsilon\}
	\end{align*}
	be diffeomorphisms so that 
	\begin{align*}
		\psi_\epsilon:\partial U\times [-t_0, t_0]\to L_0\#_U L_1\\
		(q, t)\mapsto \left\{ \begin{array}{cc}
			\alpha_s(q, t) & \text{ if $t\leq 0$}\\
			\alpha_r(q, t) & \text{ if $t >0$}
		\end{array}
		\right.
	\end{align*}
	gives us a parameterization of the neck of the surgery. 
	With this parameterization,
	\begin{align*}
		(f_1\circ \psi_\epsilon)|_{-t_0}=&2c_{\epsilon }\\
		(f_1\circ \psi_\epsilon)|_{t_0}=&2c_\epsilon \\
		(f_1\circ \psi_\epsilon)|_{0}=&c_\epsilon. 
	\end{align*}
	As the isotopy $i_\epsilon:L_0\#_U L_1\times (0, \epsilon)\to X$ is constant outside of surgery neck $N$, the Flux class descends to a class in $H^1(N, \partial N)$.
	In dimension greater than 2, $U$ is the zero set of a convex function and therefore $\partial U$ is simply connected.
	In this case, $H^1(N, \partial N)$ is generated by a single curve $\gamma:[-t_0, t_0]\to N$ with $\gamma(-t_0)\in \partial U\times \{-t_0\}$ and $\gamma(t_0)\in \partial U\times \{t_0\}$. 
	In dimension equal to 2, there is another generator $c\in H^1(N, \partial N)$, however the integral of the Louiville form on this cycle vanishes and so $\Flux_{i_\epsilon}(c)=0$.
	Therefore, the flux is characterized by $\Flux_{i_\epsilon}(\gamma)$. 

	Since we are in the exact setting (with $\omega=d\eta$), the flux on a path $c:[0, 1]\to L$ can be computed instead as $\Flux_{i_\epsilon}(\gamma):=\int_{[0, 1]}(i_{\epsilon_0}\circ c)^*\eta - \int_{[0, 1]}(i_{0}\circ c)^*\eta$.
	When the Lagrangian is parameterized by primitive $g$, so that $i(x)=(x, dg)$, then $\int_{i\circ c}\eta=g(c(1))-g(c(0))$.

	Using the decomposition of $N$ into two charts parameterized by $r_\epsilon \circ f_1\circ \psi_\epsilon|_{t\geq  0}$ and $s_\epsilon \circ f_1 \circ \psi_\epsilon|_{t\leq 0}$, we obtain 
	\begin{align*}
		\Flux_{i_\epsilon}(\gamma)=& \int_{[-t_0, t_0]}(i_{\epsilon_0}\circ \gamma)^*\eta - \int_{[-t_0, t_0]}(i_{0}\circ \gamma)^*\eta \\
		=&((r_{\epsilon_0} \circ f_1\circ \psi_\epsilon(t_0)-r_{\epsilon_0} \circ f_1\circ \psi_\epsilon(0))
		+(s_{\epsilon_0} \circ f_1\circ \psi_\epsilon(0)-s_{\epsilon_0} \circ f_1\circ \psi_\epsilon(-t_0)))\\
		&-((r_{0} \circ f_1\circ \psi_\epsilon(t_0)-r_{0} \circ f_1\circ \psi_\epsilon(0))
		+(s_{0} \circ f_1\circ \psi_\epsilon(0)-s_{0} \circ f_1\circ \psi_\epsilon(-t_0)))\\
		=&((r_{\epsilon_0}(2c_{\epsilon_0})-r_{\epsilon_0}(c_{\epsilon_0}))+(s_{\epsilon_0} (c_{\epsilon_0})-s_{\epsilon_0}(2c_\epsilon)))\\
		&-((r_{0}(2c_{\epsilon_0})-r_{0}  (0))+(s_{0} (0) -s_{0}  (2c_{\epsilon_0})))\\
		=&-r_{\epsilon_0}(c_{\epsilon_0})+s_{\epsilon_0}(c_{\epsilon_0})=\nw_{r_\epsilon, s_\epsilon}.
	\end{align*}
\end{proof}
\begin{corollary}
	If $L_0\#_U L_1$, $L_0\#'_U L_1$ are two surgeries defined with profiles $(r_\epsilon, s_\epsilon)$ and $(r'_\epsilon, s'_\epsilon)$ of matching neck width, then $L_0\#_U L_1$ is Hamiltonian isotopic to $L_0\#'_U L_1$.  
\end{corollary}
\begin{proof}
	This follows from the observation that the set of profiles  $(r_{\epsilon}, s_{\epsilon})$ of a fixed neck width is path connected, so we may always find a flux-free isotopy between two surgery profiles of the same neck width. 
\end{proof}
As the Hamiltonian isotopy class of the surgery is only dependent on the neck width of the surgery, we will write $L_0\#_U^\nw L_1$ to denote a surgery of $L_0$ and $L_1$ at $U$ with a choice of profile functions for the surgery with neck width $\nw$. 
By iterating \prettyref{prop:surgeryprofile} at each intersection point we get the following statement about symmetric differences of Lagrangians.
\begin{corollary}
	Let $L_0$ and $L_1$ be two Lagrangian submanifolds of $X$.
	Suppose that for each connected component $U_k$ of the intersection $U=L_0\cap L_1$, there is a neighborhood $U_k\subset V_k\subset L_0$ which may be identified with a subset $V_k\subset \RR^n$. Consider the Weinstein charts $B^*_\epsilon V_k\subset X$. 
	Suppose that $L_1$ restricted to this chart $B^*_\epsilon V_k$ is the graph of an exact differential form $df_k: V_k\to B^*V_k$ which vanishes on $U_k$.
	Suppose additionally that the primitives $f_k$ are all convex functions on $V_k$.

	There exists a Lagrangian  $L_0\#_{U_k}^{\underbar w} L_1$ in a small neighborhood of the symmetric difference
	\[
		L_0\#_{\{U_k\}}^{\underbar w} L_1 \subset B_\epsilon( (L_0\sqcup L_1)\setminus (L_0 \cap L_1)).
	\]
	and a Lagrangian cobordism $K:(L_0,L_1) \rightsquigarrow L_0\# L_1$.
\end{corollary}
Here, $\underbar w=\{\nw_k\}$ is the sequence of neck widths chosen at for each of the surgeries.

\begin{example}
	We compare our Lagrangian surgery with fixed neck in the non-compact setting to ordinary Lagrangian surgery as drawn in \prettyref{fig:differentsurgeries}. Let  $L_0, L_1 \subset T^*S^1$ be a cotangent fiber and its image under inverse Dehn twist around the zero section (see \prettyref{fig:differentsurgieries1} and \prettyref{fig:differentsurgieries2}).
	An application of \prettyref{prop:surgeryprofile} shows that $L_0\sqcup L_1$ is cobordant to the zero section of $T^*S^1$ by applying surgery on the overlapping regions outside a neighborhood of the zero section (see \prettyref{fig:differentsurgieries4}) \label{exam:basicsurgery}

	Let us compare this to the surgery obtained by first  perturbing $L_1$ by the wrapping Hamiltonian $\theta$ and then taking the Lagrangian connect sum. Then $L_0$ and $\theta(L_1)$ intersect at two points, which we can resolve in the usual way.
	The resulting Lagrangian $ L_0\#(\theta(L_1))$ has three connected components, two of which are non-compact (see \prettyref{fig:differentsurgieries3}). 
	Despite this, $L_0\#(\theta(L_1))$ and $L_0\#_UL_1$ agree as objects of the Fukaya category, as an additional argument shows that the non-compact components of $L_0\#(\theta(L_1))$ are trivial as objects of the Fukaya category.
	This example will become the simplest example of a construction of a tropical Lagrangian submanifold. 
	\label{exam:simplesurgery}
\end{example}
\begin{figure}
	\centering
	\begin{subfigure}{.24\linewidth}
		\centering
		\includegraphics{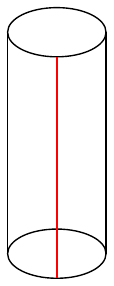}
		\caption{$L_0$}
		\label{fig:differentsurgieries1}
	\end{subfigure}
	\begin{subfigure}{.24\linewidth}
		\centering
		\includegraphics{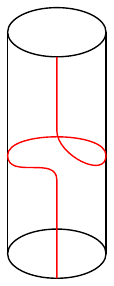}
		\caption{$L_1$}
		\label{fig:differentsurgieries2}
	\end{subfigure}	
	\begin{subfigure}{.24\linewidth}
		\centering
		\includegraphics{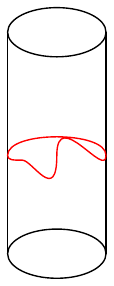}
		\caption{$L_0\#_U L_1$}
		\label{fig:differentsurgieries4}
	\end{subfigure}
	\begin{subfigure}{.24\linewidth}
		\centering
		\includegraphics{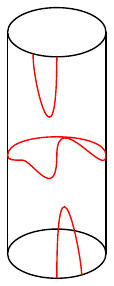}
		\caption{$L_0\#(\theta(L_1))$}
		\label{fig:differentsurgieries3}
	\end{subfigure}	
	\caption{The difference between surgery with neck $U$ and ordinary Lagrangian surgery.}
	\label{fig:differentsurgeries}
\end{figure}

In the setting of the cotangent bundle $T^*\RR^n$, the connect sum has image under the projection to $\RR^n$ which lives in a neighborhood of the complement of the regions $U$.
We now examine what the projection to the fiber of $T^*_0\RR^n$ of this surgery looks like.
Let $\arg: T^*\RR^n\to T^*_0\RR^n$ be projection to the cotangent fiber of zero, and for any set $C$, let $\arg(C)$ be the image of this set under the projection.
Suppose that $f$ has minimal value $0$. 
By the convexity of $f$, whenever $\{x\;|\; df(x)=0\}$ is compact and $c> 0$, then $0$ is an interior point of $\arg(\{df(x)\;|\; f(x)\leq c\})$, which is a ball containing the origin.

\begin{prop}
	Let $L_0$ and $L_1$ be sections of $T^*\RR^n$ as in \prettyref{prop:surgeryprofile}.
	Suppose that $\{x\;|\; df_1(x)=0\}$ is compact.
	Then for a choice of surgery parameter $\epsilon$ sufficiently small
	\[
		\arg(L_0\#_U L_1 )= \arg (L_1).
	\]
	\label{prop:argument}
\end{prop}
\begin{proof}
	For this proof, let $f=f_1$, $r=r_\epsilon, s=s_\epsilon$.
	Recall that we have normalized $f$ so that $df(x)=0$ if and only if $f(x)=0$.  
	The decomposition of $L_0\#_U L_1$ into the $r$ and $s$ charts breaks $\arg(L_0\#_U L_1)$ into two components,
	\[
		\arg(\{d(r\circ f)(x)\;|\;\epsilon \leq f(x)\})\cup \arg(\{d(s\circ f)(x)\;|\;\epsilon \leq f(x)\}).
	\]
	Choose $c$ small enough so that $B_c(0)$ is an open ball in $\arg(df(\RR^n))$.  
	Choose $\epsilon$ small enough so that $f(x)\leq 2\epsilon$ implies $df(x)\subset B_c(0)$. 
	Consider the following three parameterized subsets of $T^*_0\RR^n$:
	\begin{align*}
		B:\{x\;|\; f(x)\leq 2\epsilon\}\xrightarrow{\arg(df)}                           & T^*_0\RR^n \\
		C_r:\{x\;|\; \epsilon \leq f(x)\leq 2\epsilon\}\xrightarrow{\arg(d(r \circ f))} & T^*_0\RR^n \\
		C_s:\{x\;|\; \epsilon \leq f(x)\leq 2\epsilon\}\xrightarrow{\arg(d(s\circ f))} & T^*_0\RR^n
	\end{align*}
	
	Since $r\circ f=f$ outside of $f(x)\leq 2\epsilon$,  to prove the proposition it suffices to show that  $B\subset C_r\cup C_s \subset B_c(0)$.

	Since $s', r' \leq 1$, the chain rule for the compositions $r\circ f, s\circ f$ gives the inclusion $( C_r \cup C_s)\subset B_c(0)$. 
	As topological chains, $C_r$ and $C_s$ have boundary components corresponding to where $f(x)= \epsilon$  and $f(x)=2\epsilon$. 
	The boundary components have the following identifications:
	\begin{align*}
		\partial_{2\epsilon} C_r= \partial_{2\epsilon} B = &\arg(df)|_{ f(x)=2\epsilon}           \\
		\partial_{\epsilon} C_r = \partial_{\epsilon} C_s =& \frac{1}{2} \arg(df)|_{f(x)=\epsilon} \\
		\partial_{2\epsilon} C_s=&0
	\end{align*}
	As the inner boundaries of $C_r$ and $C_s$ match, we may glue these two charts into a single chain $C_{r+s}$, with boundary $\partial(C_{r+s})= \partial B+ \{0\}$.
	The chain $C_{r+s}$ provides a contraction of the boundary sphere of $B$ to the point $0\in B$. 
	We therefore obtain that $B\subset C_{r+s}$, completing the proof.
\end{proof}
In the setting where $U$ is non-compact, we still obtain that for sufficiently small surgery parameters	
\[
	\arg(L_0\#_U L_1 )\subset \arg (L_1).
\]
In the non-compact setting one can obtain the result of \prettyref{prop:argument} by imposing additional control on the argument of $f$ over the non-compact region.  
Let $f: \RR^n\to \RR$ be a convex function.
Suppose that there exists a non-compact convex polytope $V_\infty\subset \RR^n$ which has a single vertex $v_\infty$.
We furthermore suppose that the vertex $v_\infty\in f^{-1}(0)$. 
We denote the faces of the polytope as $F< V_\infty$, and we will also consider $V_\infty$ as a face of this polytope.
The polytope $V_\infty$ will provide additional control on the argument of $f$ over the region $V_\infty$ by imposing the following additional two conditions:
\begin{itemize}
	\item We require that the intersection of every face $F\cap f^{-1}(\{x<2\epsilon\})$ is compact.
	\item At each face $F\leq V_\infty$, we require that $df|_{NF}=0$.
\end{itemize}
We will denote by $\arg(T^*_{v\infty}F)\subset T^*_0\RR^n$ the subspace parallel to $F$ in $\RR^n$. 
See \prettyref{fig:noncompact} for our setup.
For each $F\in V_\infty$ define the chains with domains restricted to faces by
\begin{align*}
	B_F:	\{x\;|\; f(x)\leq 2\epsilon, x\in F\}\xrightarrow{\arg(df)}                           & T^*_0\RR^n \\
	C_{r, F}:\{x\;|\; \epsilon \leq f(x)\leq 2\epsilon, x\in F\}\xrightarrow{\arg(d(r \circ f))}& T^*_0\RR^n \\
	C_{s, F}:\{x\;|\; \epsilon \leq f(x)\leq 2\epsilon, x\in F\}\xrightarrow{\arg(d(s \circ f))}& T^*_0\RR^n .
\end{align*}
If $F\neq v_\infty$, then $F$ contains a point on which $f>2\epsilon$, and so $C_{r, F}, C_{s, F}$ are nonempty. 

In addition to the analogues of the boundary components $\partial_{2\epsilon}C_{r, F}, \partial_{\epsilon} C_{r, F},\partial_{2\epsilon}C_{s, F}, \partial_{\epsilon} C_{s, F},$ and $\partial_{2\epsilon} B_F$ defined in \prettyref{prop:argument} the chains $B_F, C_{r, F}$ and $C_{s, F}$ have boundary components corresponding to the boundaries of $F$. 
\begin{align*}
	\partial_\infty(C_{r, F}):= \bigcup_{G< F} C_{r, G} &&
	\partial_\infty(C_{s, F}):= \bigcup_{G< F} C_{s, G} && \partial_\infty(B_{f}):= \bigcup_{G<F} B_G.
\end{align*}
Since $f|_F$ is a convex function, $\{x\;|\; f(x)\leq 2\epsilon, x\in F\}$ is a convex region on $F$, and $df|_F\in\arg(T^*_{v_\infty}F)$, the same argument from \prettyref{prop:argument} shows that $C_{r, F}+C_{s, F}$ is contained within $B_F$, and that $B_F$ is convex hull of its boundary.
To replicate the argument of \prettyref{prop:argument}, we need to show that the boundary of $C_{r, F}+C_{s, F}$ 
\[\partial_{2\epsilon} C_{r, F}+ \partial_{2\epsilon} C_{s, F} + \left(\sum_{G<  F} \left(C_{r, G}+ C_{s, G}\right)\right)\]
agrees with the boundary of $B_F$.  
\begin{prop}
	For functions $f$ and polytopes $V_\infty$ satisfying the requirements of the above discussion, and  $F\leq V_\infty$ a face with $\dim F\geq 1$ the chains $C_{r, F}+C_{s, F}$ and $B_F$ are homotopic inside $\arg(T^*_{v_\infty} F)$ and have matching image.
	\label{prop:noncompactargument}
\end{prop}
\begin{figure}
	\centering
	\begin{subfigure}{.25\linewidth}\centering
	\includegraphics{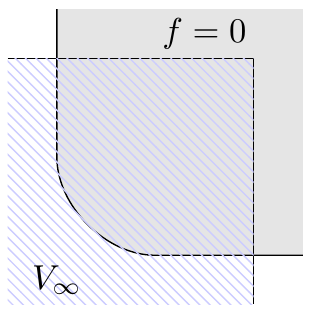}
	\caption{Set-up for \prettyref{prop:noncompactargument}.}
	\label{fig:noncompact}
	\end{subfigure}
	\begin{subfigure}{.7\linewidth}
		\centering
	\includegraphics{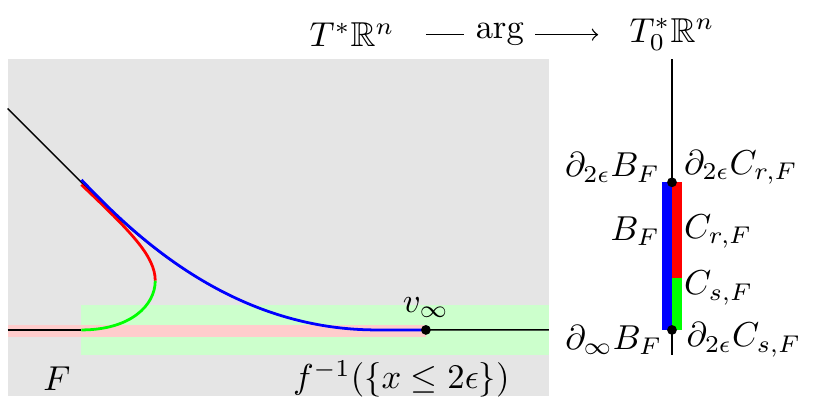}
	\caption{1-dimensional case of \prettyref{prop:noncompactargument}.}
	\label{fig:noncompact1dim}
	\end{subfigure}
	\caption{}
\end{figure}
\begin{proof}
As noted above, the image of $C_{r, F}+C_{s, F}$ is contained within $B_F$, so we need to prove the other inclusion. 
We argue by induction on the dimension of the faces.
The condition that $df|_{NF=0}$ implies that at each face, $B_F, C_{r, F}, C_{s, F}\subset \arg(T^*_{v_\infty} F)\subset \RR^n$, which will allow us to run our induction.

For the 1-dimensional case, let $F$ be any edge. Let $v\in F$ be the point where $f(v)=2\epsilon$.
The boundaries of the chains $B_F, C_{r, F}$ and $C_{s, F}$ are:
\begin{align*}
	\partial_\infty B_F= \arg(df(v_\infty))=0 && \partial_\infty C_{r, F}=\emptyset && \partial_\infty C_{s, F}=\emptyset\\
	\partial_{2\epsilon} B_F = \arg(df(v)) && \partial_{2\epsilon} C_{r, F}=\arg(df(v)) && \partial_{2\epsilon} C_{s, F}= 0.
\end{align*}
Both $C_{r, F}+ C_{s, F}$ and $B_F$ are paths which link $0$ and $df(v)$ inside of $\arg(T^*_{v_\infty}F)$, so they match after reparameterization. See \prettyref{fig:noncompact1dim}.

Now suppose that the induction hypothesis holds for all faces $G< F$. 
By the induction hypothesis the images of $B_{G}$ and $C_{r, G}+C_{s, G}$ match for all $v_\infty < G< F$. Additionally, $B_{v_\infty}=0$ has image matching $\partial_{\infty} C_{s, F}$. We obtain that the following chains have matching image:
\begin{align*}
	\partial_\infty B_F=\partial_\infty(C_{r, F}+C_{s, F})+\partial_{2\epsilon} C_{s, F}& & \partial_{2\epsilon} B_F = \partial_{2\epsilon} C_{r, F}.
\end{align*}
Therefore, the images  of $\partial(C_{r, F}+C_{s, F})$ and $\partial B_{F}$ agree. Furthermore, these cycles are homotopic through their common image.
 $C_{r, F}+C_{s, F}$ is a contraction of a sphere which is the boundary of a convex set $B_F$, and therefore the image of $B_F$ is contained in the image of $C_{r, F}+C_{s, F}$. 
\end{proof}
\subsection{Tropical Lagrangian Sections}
There is a nice collection of admissible Lagrangians inside of $\Fuk_{\Delta_\Sigma}(X, W_\Sigma)$ called the tropical Lagrangian sections of $X\to Q$ which we will use as building blocks in our construction. These were introduced in \cite{abouzaid2009morse}. For our model of the Fukaya Seidel category, we use the monomial admissible Fukaya Seidel category  $\Fuk_{\Delta_\Sigma}(X, W_\Sigma)$. This version of the FS category is due to \cite{hanlon2018monodromy}.

Let $\phi: Q \to \RR$ be a tropical polynomial.
Choose $\epsilon$ small enough so that for any point $q\in \underbar U_{\{v_i\}}$ the convex hull of $d\phi(B_\epsilon(q))$  is either all of $\underbar U^{\{v_i\}}$, or contains $\underbar U^{\{v_i\}}$ as a boundary component.
This means that $\epsilon$ is small enough that for all $q\in Q$ the induced stratification of  $B_\epsilon(q)$ from the tropical variety has at most one vertex. See \prettyref{fig:smallepsilon} for a non-example.

\begin{figure}
	\centering
	\includegraphics{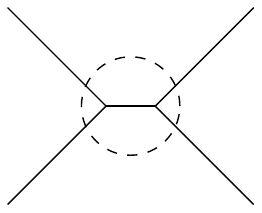}
	\caption{The kind of behavior we wish to rule out by making $\epsilon$ small.}
	\label{fig:smallepsilon}
\end{figure}
Define $\tilde \phi^\epsilon $ to be the smoothing of $\phi$ by convolution with a symmetric non-negative bump function $\rho_\epsilon$ with support $B_\epsilon(0)$, a small ball around the origin.

When we have fixed a size $\epsilon$, we will simplify notation and refer to this as a smoothing $\tilde \phi$. The smoothing $\tilde \phi$ enjoys many of the same properties of $\phi$.
\begin{prop}[(Properties of $\tilde \phi$)]
	The function $\tilde \phi$ satisfies the following properties.
	\begin{itemize}
		\item
		      (Nearly Tropical) The function  $\tilde \phi$ and $\phi$ agree outside of an $\epsilon$ neighborhood of $V(\phi)$
		\item
		      (Concavity)   $\tilde \phi$ is a concave function
		\item
		      (Newton Polytope) $\arg(d\tilde \phi(\RR^n))= \Delta_\phi\subset T^*\RR^n.$
	\end{itemize}
\end{prop}
\begin{proof}
	The first property comes from the preservation of linear functions under symmetric smoothing.

	The second property is a general statement about convolutions of concave functions against non-negative functions.
	Let $\psi$ be a smooth concave function.
	Then $A=\Hess(\psi)$ is the matrix with entries given by $a_{ij}=(\partial_i\partial_j \psi)$.
	Concavity of $\psi$ is equivalent to checking that $A$ is negative semi-definite i.e. for all vectors $v$, $v^T\cdot A \cdot v\leq 0$.  
	Let $\tilde \psi = \rho_\epsilon * \psi$, and let $\tilde A = \Hess(\tilde \psi)$. 
	Then $\tilde a_{ij}=\partial_i\partial_j(\rho_\epsilon*\psi)=\rho_\epsilon*(\partial_i\partial_j\psi)=\rho_\epsilon*(a_{ij})$. 
	By linearity, $v^T\cdot \tilde A \cdot v = \rho* (v^T\cdot \tilde A \cdot v)$. 
	Since $\rho\geq 0$, and the convolution of non-negative functions is again non-positive, we obtain that $\tilde \psi$ is concave. 
	\footnote{In the 1-dimensional setting, whenever $f(t)$ is a function of a real variable and $f''(t)\geq 0$, then $(f*\rho)''(t)=(f''*\rho)(t)\geq 0$. }
	
	The third property requires a lemma about convolution and the argument of a function.	
	\begin{lemma}
		Let $\phi:\RR^n \to \RR$ be a smooth function. 
		Let $\rho_\epsilon: B_\epsilon(0)\to \RR$ be a non-negative smoothing kernel.
		Then $d(\rho_\epsilon*\phi)(p_0)\in \overline \Hull(\arg(\{d\phi(x)\;|\;x\in B_\epsilon(p_0)\}))$. 
		\label{lemma:weightedaverage}
	\end{lemma} 
	We delay the proof of the lemma. 
	With the lemma, $\arg(d\tilde \phi(\RR^n))\subset \Hull(\arg(d\phi(x)))= \Delta_\phi$. 
	For the reverse direction, we remark that for every $v\in \Delta^\ZZ_\phi$, the set $U_{\{v_i\}}$ on which $d\tilde \phi=v$ is non-empty.
	Therefore, $\Delta_\phi=\Hull(\Delta^\ZZ_\phi)\subset \Hull(\{U_{v_i}\})\subset \arg(d\tilde \phi(\RR^n)).$
\end{proof}

\begin{proof}[of \prettyref{lemma:weightedaverage}]
	Let $A\subset \RR^n$ be a compact set. 
	Let $\rho:\RR^n\to \RR$ be a function with $\int_{\RR^n}\rho d\vol =1$ and $\rho\geq 0$. 
	The $\rho$-weighted center of mass of $A$ is $\avg_\rho(A):=(\int_A x_1\rho d\vol, \ldots , \int_A x_n \rho d\vol)$.
	Let $\overline{\{\avg_\rho(A)\}}$ be the closure over all $\rho$-weighted centers of mass over all $\rho$. 
	We first sketch an argument showing that  
	\[\Hull(A)=\overline{\{\avg_\rho(A)\}}\]
	To show that $\Hull(A)\subset \overline{\{\avg_\rho(A)\}}$, we use Carath\'eodory's theorem to write any $p\in \Hull(A)$ as $p=\sum_{i=1}^{n+1} \alpha_i p_i$ as a convex combination of points $p_i\in A$. 
	Let  $\rho_{B_\epsilon(x)}$ be a non-negative smooth bump function with support on $B_\epsilon(x)$ and total integral 1.  
	Let $\rho_{\epsilon, \alpha_i, p_i} :=\sum_{i=1}^{n+1} \alpha_i \rho_{B_\epsilon(p_i)}$ be a sum of such bump functions.
	Then $\lim_{\epsilon\to 0} \avg_{\rho_{\epsilon, \alpha_i, p_i}}(A)=p.$
	
	The reverse inclusion is obtained by taking a sequence of approximations for the weighting $\rho$ by 
	\[\rho|_A\approx \sum_{i=1}^k \alpha_i \rho_{B_\epsilon(p_i)}.\]
	for which 
	\[\avg_\rho(A)\approx  \sum_{i=1}^k \alpha_i \int x_i \rho_{B_\epsilon(p_i)} \approx  \sum_{i=1}^k \alpha_i p_i\]
	
	We now prove a generalization.
	Given $A\subset \RR^n$ a compact set, and $\eta: A\to \RR^k$ a function, define $\Hull(\eta(A))$ to be the convex hull of the image of $A$ under $\eta$. 
	We expand $\eta$ in coordinates as $\eta=(\eta_1, \ldots, \eta_k)$.
	Let $\rho:\RR^n\to \RR$ be a function with $\int_{\RR^n}\rho d\vol=1$.
	The $\rho$-weighted average of $\eta$ over $A$ is the value 
	\[
		\avg_\rho(\eta(A)):=\left(\int_A \eta_1\rho d\vol, \ldots, \int_A \eta_k \rho d\vol \right).
	\]
	We now show that $\overline{\Hull(\eta(A))}= \overline{\{\avg_\rho(\eta(A))\}}.$ 
	Let $p=\sum_{i=1}^k \alpha_i \eta(p_i)$ be a convex linear combination of points. 
	Then by again taking $\rho_{\epsilon, \alpha_i, p_i} :=\sum_{i=1}^{n+1} \alpha_i \rho_{B_\epsilon(p_i)}$ we obtain a sequence of $\rho$-weighted averages of $\eta$ which converges to $p$. 
	For the reverse inclusion, we approximate a weighting $\rho$ by a sum of bump functions with small support. 

	With this generalization we can prove \prettyref{lemma:weightedaverage}, as 

	\begin{align*}
		d(\rho_\epsilon*\phi)(p_0)=&\sum_{i=1}^n \frac{\partial (\rho_\epsilon*\phi)}{\partial x_i} dx_i
		=\sum_{i=1}^n \rho_\epsilon*\frac{\partial \phi}{\partial x_i} dx_i\\
	\end{align*}
	is a $\rho_\epsilon$ weighted average of the function 
	\begin{align*}
		\eta: \RR^n\to \RR^n\\
		x\mapsto \sum_{i=1}^n\frac{\partial \phi}{\partial x_i}dx_i.
	\end{align*}
	Therefore, $d(\rho_\epsilon*\phi)(p_0)\in \overline{\Hull(\arg(\{d\phi(x)\;|\;x\in B_\epsilon(p_0)\}))}.$
\end{proof}
For each $\{v_i\} \subset \Delta_\phi$ , let
\[
	U_{\{v_i\}}:=\{q \;|\; d(\tilde\phi)(q)\subset \text{Interior of the Convex Hull of $\{v_i\}$}\}.
\]
Note that if $\{v_i\}$ is just a single point $v$, then $U_v=\{q\;|\; d\tilde\phi(q)=v\}$. 
Each $ U_{\{v_i\}}$ is $\epsilon$-close to  $\underbar U_{\{v_i\}}$ discussed in \prettyref{subsec:tropicalgeometry}. These sets $U_{\{v_i\}}$ can be characterized in terms of the smoothing ball $B_\epsilon$ and the combinatorics of $\phi$.
\begin{prop}
	The set $ U_{\{v_i\}}$ is the set of all points $q\in Q$ such that $d\phi|_{B_\epsilon(q)}$ belongs to $\underbar U^{\{v_i\}}$, and is not contained in the boundary of $\underbar U^{\{v_i\}}$.
	\label{prop:intersectioncharacterization}
\end{prop}
\begin{proof}
	By \prettyref{lemma:weightedaverage}, if $d\phi|_{B_\epsilon(q)}$ belongs to  $\underbar U^{\{v_i\}}$ then  $d\tilde \phi(q)\in \underbar U^{\{v_i\}}$. 
	The only concern may be that $\tilde d\phi|_q$ is not in the interior of $\underbar U^{\{v_i\}}$, however the requirement that $d\phi|_{B_\epsilon(q)}$ is not contained in the boundary of $\underbar U^{\{v_i\}}$ rules out this possibility.
	Therefore, $q\in U_{\{v_i\}}$.

	Suppose now that $q\in U_{\{v_i\}}$.
	We would like to show that within an $\epsilon$-radius of $q$, $d\phi(q)$ belongs to $\underbar U^{\{v_i\}}$, and that at least one point is not contained in the boundary.
	We first show containment. Suppose for contradiction that $d\phi|_{B_\epsilon(q)}\not\subset \underbar U^{\{v_i\}}$. Then take additional vertices $w_k$ so that $d\phi|_{B_\epsilon(q)}\subset U^{\{v_i\}\cup\{w_k\}}$. We break into two cases.
	\begin{itemize}
		\item
		      The set $\{w_k\}$ only contains one element. In this case, the weighted average over arguments defining $d\phi(q)$ cannot possibly lie in $\underbar U^{\{v_i\}}$.
		\item
		      The set $\{w_k\}$ contains at least 2 elements. This contradicts our assumption on the size of $\epsilon$, as we now see top dimensional strata corresponding to two different boundaries of $\underbar U_{\{v_i\}}$.
	\end{itemize}
	This proves that $d\phi|_{B_\epsilon(q)}\subset \underbar U^{\{v_i\}}$.
	As the value of $d\tilde \phi(q)$ is an interior point of $\underbar U^{\{v_i\}}$ by the definition of $q\in U_{v_i}$, it cannot be the case all of $d\phi|_{B_\epsilon(q)}$ is contained in the boundary of $\underbar U^{\{v_i\}}$.
\end{proof}
This proposition gives us a  clean description of the sets $U_{\{v_i\}}$, and additional information on the restriction of $\phi$ to each of these subsets.
In the setting of top dimensional strata, we have an inclusion $ U_v\subset \underbar U_v$, and $d(\tilde \phi)|_{U_v}=v$.
The strata $U_v$ contains an open set if and only if $d(\tilde \phi)|_{U_v}\in \Delta^\ZZ_\phi$.

The graph of $d\tilde \phi$ is a Lagrangian in $T^*Q$ rather than in $X=(\CC^*)^n$, but after periodizing the cotangent bundle, we get sections of the SYZ fibration.
\begin{definition}[\cite{abouzaid2009morse}]
	The \emph{tropical Lagrangian section }$\sigma_\phi^\epsilon: Q\to X$ associated to $\phi$ is  the composition
	\[
		\begin{tikzcd}
			T^*Q\arrow{r}{ / T^*_\ZZ Q } & X\\
			Q\arrow{u}{d\tilde \phi ^\epsilon }
		\end{tikzcd}.
	\]
	\label{def:tropicallagrangiansection}
\end{definition}
When the smoothing radius $\epsilon$ is not important, we will drop it and simply write $\sigma_\phi$.
The key observation is that for each lattice point $v\in \Delta^\ZZ_\phi$
\begin{equation}
	\sigma_\phi|_{U_v}= \sigma_0|_{U_v}.
	\label{eq:keyobservation}
\end{equation}
Given a monomial admissibility condition $(W_\Sigma, \Delta_\Sigma)$, we say that $\phi$ is an admissible tropical polynomial if $\sigma_\phi$ is an $\Delta_\Sigma$-monomially admissible Lagrangian.
From here on out, we will only work with admissible tropical polynomials.

Take an admissible perturbation of $\sigma_{-\phi}$ so that $\arg(\sigma_{-\phi})$ is locally a diffeomorphism onto its image.
This can be arranged by adding a small strictly convex function to $\tilde\phi$.
After incorporating this Hamiltonian perturbation, the intersections between $\sigma_{-\phi}$ and $\sigma_0$ become transverse.
The intersections are in bijection with the points where $\arg(\sigma_{-\phi})=0$. 
Since $\arg(\sigma_{-\phi})$ is (after considering the small Hamiltonian perturbation) a slight enlargement of the Newton polytope $\Delta_\phi$, the intersections between $\sigma_{\phi}$ and $\sigma_0$ are identified with the lattice points of $\Delta_\phi$.

This observation, along with a characterization of holomorphic strips on convex functions allows one to describe the Floer cohomology of the tropical Lagrangian sections combinatorially.
\begin{theorem}[\cite{abouzaid2009morse,hanlon2018monodromy}]
	Let $\check X_\Sigma$ be a toric variety, and let $(X, W_\Sigma)$ be its mirror Landau-Ginzburg model.
	Let $\Delta_\Sigma$ be a monomial admissibility condition.
	Let $\phi_1, \phi_2$ be the support functions for line bundles $\mathcal O(\phi_1), \mathcal O(\phi_2)$ on $\check X_\Sigma$ (see \prettyref{subsec:toricbackground}).
	Then after appropriately localizing, there is a quasi-isomorphism
	\[
		\CF(\sigma_{\phi_1}, \sigma_{\phi_2})= \hom(\mathcal O(\phi_1), \mathcal O (\phi_2)).
	\]
	Furthermore, the $A_\infty$ structure on the subcategory of the Fukaya-Seidel category generated by tropical Lagrangian sections is quasi-isomorphic to the dg-structure on the dg-enhancement of the derived category of coherent sheaves on $\check X_\Sigma$.
	\label{thm:hmstorics}
\end{theorem}
Assuming $\check X_\Sigma$ is smooth and projective, line bundles generate the derived category of coherent sheaves on $\check X_\Sigma$.
This proves that the subcategory of $\Fuk_{\Delta_\Sigma}(X, W_\Sigma)$ generated by tropical Lagrangian sections is equivalent to $D^b\Coh(\check X_\Sigma)$.

\subsection{Tropical Lagrangian Hypersurfaces}
For this section, we fix $(W_\Sigma, \Delta_\Sigma)$ some monomial division and work in the setting where \prettyref{thm:hmstorics} holds.

It is usually desirable for Lagrangians to have transverse intersections.
However, the highly non-transverse configuration of unperturbed tropical Lagrangian sections will work in our favor as locally the intersection of $\sigma_0$ and $\sigma_{-\phi}$ is given by the graphs of one-forms with convex primitives. This allows us to apply our surgery profile from  \prettyref{prop:surgeryprofile}.
\begin{prop}
	Let $\phi$ be a tropical polynomial on $\RR^n$. 
	The connected components of $\sigma_0\cap \sigma_{-\phi}$ are in bijection with the lattice points of the Newton polytope, $\ZZ^n\cap \Delta_\phi$.
	The connected components which contain an open ball are in bijection with the top-dimensional linearity strata $\Delta^\ZZ_\phi$. 
	\label{prop:latticepointsareintersections}
\end{prop}
\begin{proof}
	The set $\sigma_0\cap \sigma_{-\phi}$ is homeomorphic to $\arg(-d\tilde \phi)^{-1}(\ZZ^n)$.
	Since $-\tilde\phi^{-1}$ is convex, we can compute the set on which it has a fixed derivative as the minimal locus of an appropriately shifted convex function
	\[\arg(-d\tilde\phi)^{-1}(v)=\{x\;|\;-\tilde\phi-x^v \text{ is minimal}\}.\]
	Since the minimal locus of a convex function is contractible, $\arg(-d\tilde\phi)^{-1}(v)$ is a contractible set.
	Therefore, $\arg(-d\tilde \phi)^{-1}(\ZZ^n)$ is a union of disjoint contractible sets, one for each lattice point contained in the image of $\arg(-d\tilde \phi)$. 

	To show that the connected components which contain an open ball are in bijection with the top dimensional linearity strata, we use \prettyref{prop:intersectioncharacterization}.
\end{proof}
\begin{definition}
	Let $\phi$ be a tropical polynomial on $\RR^n$. Let $\{U_v\;|\; \dim U_v=n\}$ be the collection of intersections between $\sigma_{0}$ and $\sigma_{-\phi}$ corresponding to the smooth non-self intersecting monomials of $\phi$, i.e. points $v\in \Delta^\ZZ_\phi$. 
	For choice of sufficiently small $\epsilon$, let $\mathcal D_\epsilon=\{\rho_\epsilon, r_\epsilon, s_\epsilon\}$ be a choice of smoothing kernel and surgery profile curves.
	We define the \emph{ tropical Lagrangian associated to $\phi$ and necks $U_v$} by the surgery
	\[
		L^{\mathcal D_\epsilon}(\phi):=(\sigma_{0})\#_{\{U_v\}}(\sigma_{-\phi}),
	\]
	with the given smoothing kernel and surgery profile curves.
	\label{def:tropicallagrangian}
\end{definition}
\begin{example}
	The simple example of a surgery given in \prettyref{exam:simplesurgery} is the tropical Lagrangian. 
	In this setting $X=\CC^*$, and $Q=\RR$. 
	The tropical hypersurfaces of $Q$ are simply points, and so we expect that the tropical Lagrangian associated to a point will be an SYZ fiber. 
	The two sections $\sigma_0$ and $\sigma_{-(0\oplus x_1)}$ are the line and twisted line of \prettyref{fig:differentsurgeries} which are surgered together. 
	The resulting Lagrangian is Hamiltonian isotopic to the SYZ fiber with valuation 0. 
\end{example}
We will later show that the choice of data $\mathcal D_\epsilon$ does not change the exact isotopy class of $L$, and will therefore usually write $L(\phi)$ instead.
In the definition of a tropical Lagrangian submanifold, we've taken the connect sum along each strata of $U_v$ corresponding to the non-self-intersection strata of the tropical polynomial $\phi$.
As a result, a tropical variety with self-intersections only lifts to an immersed Lagrangian.
As an example, the Lagrangian lift of the tropical polynomial $\phi_{T^2}^0(x_1,x_2) =x_1\oplus x_2\oplus (x_1x_2)^{-1}$ drawn in \prettyref{fig:nonsmoothtropicalelliptic} is an immersed sphere with 3 punctures and 1 transverse self intersection.

These intersections may be transverse, but they need not be --- an example is $\phi=x_1\oplus x_2 \oplus(x_1x_2)^{-1}$ as a tropical function on $Q=\RR^3$, where the self-intersection  is a clean intersection $\RR\subset L(\phi)$.

The distinction in terminology between the smoothness and self-intersections becomes important here.
For example, the tropical curve exhibited in \prettyref{fig:smoothtropicalcross} lifts to an embedded tropical Lagrangian submanifold, as the tropical curve has no self-intersections. 
However the tropical curve from \prettyref{fig:smoothtropicalcross} is not an example of a smooth tropical curve.

\begin{theorem}
	Let $\phi$ be a tropical polynomial. The tropical Lagrangian hypersurface $L^{\mathcal D_\epsilon}(\phi)$ and corresponding cobordism $K^{\mathcal D_\epsilon}(\phi)$ satisfy the following geometric properties:
	\begin{itemize}
		\item
		      \textbf{Independence from Choices:} Different choices of parameters ${\mathcal D_\epsilon}$ in the construction of these Lagrangians produce exactly isotopic tropical Lagrangians.
		\item
		      \textbf{Valuation Projection:} As $\epsilon$ approaches zero, the valuation of our tropical Lagrangian $\val(L^{\mathcal D_\epsilon}(\phi))$ approximates the tropical hypersurface $V(\phi)$.
		\item
		      \textbf{Argument Projection:} The argument projection $\arg(L^{\mathcal D_\epsilon}(\phi))$ is the Newton polytope $\Delta_\phi$ associated to the line bundle $\mathcal O(-D)$.
		\item
			  \textbf{Topology:} When $V(\phi)$ is a smooth tropical variety, $L^{\mathcal D_\epsilon}(\phi)$ is embedded.
	\end{itemize}
	Additionally, the tropical Lagrangian submanifolds satisfy these technical requirements giving them well defined Floer cohomology.
	\begin{itemize}
		\item
		      \textbf{Admissibility:} Let $(W_\Sigma, \Delta_\Sigma)$ be a monomial admissibility condition as defined in \cite{hanlon2018monodromy}. Then whenever $\sigma_{-\phi}$ is monomially admissible, so are $L^{\mathcal D_\epsilon}(\phi)$ and $K^{\mathcal D_\epsilon}(\phi)$.
		\item
		      \textbf{Unobstructedness:} Assuming \prettyref{assum:welldefinedfloer} and \prettyref{assum:bottleneckideals}, there exists a bounding cochain for the Lagrangian cobordism $K(\phi):(\sigma_0, \sigma_{-\phi})\rightsquigarrow L(\phi)$
	\end{itemize}
	\label{thm:tropicalLagrangians}
\end{theorem}
\begin{proof}
	The proof of independence of choice is given by \prettyref{prop:Hamiltonianisotopy}. 
	The valuation projection is shown in  \prettyref{prop:surgeryprofile} and \prettyref{prop:intersectioncharacterization}.
	Embeddedness of liftings of smooth tropical varieties follows from \prettyref{def:tropicallagrangian}, as in this setting all intersections of $\sigma_0$ and $\sigma{-\phi}$ are surgered away.

	To prove that the argument projection matches, we use \prettyref{prop:argument} at each surgery portion where the surgery region is compact, and \prettyref{prop:noncompactargument} on the non-compact regions.
	To apply \prettyref{prop:noncompactargument}, we first assume we are in the case where the Newton polytope of $\phi$ is full dimensional (otherwise, we quotient out by translation symmetry to reduce to this case).
	Let $U_v$ be a non-compact intersection region.
	The vertex $v$ belongs to the boundary of the Newton polytope. Let $V_{\infty, v}$ to be the cone over a small neighborhood of $-v$ in $-\Delta_\phi$. This contains a single vertex $-v$.
	After taking an appropriate translation of $V_{\infty, v}$ placing the vertex $-v\in U_v$, each strata of $V(\phi)$ meets the dual face in $V_{\infty, v}$ orthogonally.
	Therefore at each face $F$, we have  $d\tilde \phi|_{NF}=0$ (as drawn in \prettyref{fig:noncompact}) and may apply \prettyref{prop:noncompactargument}.

	The the proof of admissibility is \prettyref{prop:admissibility}.
	The proof of unobstructedness is left to \prettyref{sec:unobstructedness}. 
\end{proof}

\subsubsection{Independence from choices}

In our definition of $L^{\mathcal D_\epsilon}(\phi)$, we've made  choices for the data $\mathcal D_\epsilon=(\rho_\epsilon, r_\epsilon, s_\epsilon)$.
Fortunately, these choices do not modify the exact isotopy class of $L(\phi)$.
\begin{prop}
	Assume that $V(\phi)$ has no self-intersections.
	For parameters $\epsilon_1, \epsilon_2$ sufficiently small, choose data $\mathcal D_{\epsilon_1}=\{\rho_{\epsilon_1}, r_{\epsilon_1}, s_{\epsilon_1}\}$ and $\mathcal D_{\epsilon_2}=\{\rho_{\epsilon_2}, r_{\epsilon_2}, s_{\epsilon_2}\}$.
	The tropical Lagrangians $L^{\mathcal D_{\epsilon_1}}(\phi)$ and $ L^{\mathcal D_{\epsilon_2}}(\phi)$	are Hamiltonian isotopic.
	\label{prop:Hamiltonianisotopy}
\end{prop}
\begin{proof}
	These two Lagrangians are Lagrangian isotopic as we can smoothly interpolate between the parameters $\mathcal D_{\epsilon_1}$ and $\mathcal D_{\epsilon_2}$ in our construction.
	Let $L^{\mathcal D_{\epsilon_t}}(\phi)$ be a Lagrangian isotopy between the choices of data.
	Since $(\CC^*)^n$ is exact, we can prove the proposition by showing that the integral of the Liouville form $\eta$ (satisfying $d\eta=\omega$) on cycles of $L^{\mathcal D_{\epsilon_t}}(\phi)$ is independent of the choices that we've made.
	For simplicity, let $L(\phi)$ be the tropical Lagrangian constructed with any choice of data $\mathcal D=\{\rho, s, r\}$, and we'll show that the Liouville form on $H_1(L(\phi))$ only depends on the choice of $\phi$.
	For this computation, we will choose a convenient spanning set for $H_1(L(\phi))$.

	We now assume that $L(\phi)$ is connected, and that $0$ is an lattice point of $\Delta_\phi$.
	To compute $H_1(L(\phi))$, decompose $L(\phi)=L_r\cup L_s$, where $L_r$ and $L_s$ are the charts corresponding to the profiles $r,s$ in \prettyref{fig:surgeryprofile}.
	These charts are homotopic to the tropical amoeba $\RR^n\setminus\{U_v\}$, so $H_1(L_s)=H_1(L_r)=H_1(V(\phi))$.
	The inclusion $i_s: L_s\cap L_r\to L_s$ is also an isomorphism on the first homology. 
	The Meyer-Vietoris  sequence  allows us to decompose the homology of $L(\phi)$ as
	\begin{align*}
	\cdots\to& H_1(L_r\cap L_s)\xrightarrow{(i_r,i_s)_{1}} H_1(L_r)\oplus H_1(L_s)\xrightarrow{(j_r)_{1}-(j_s)_{1}} H_1(L(\phi))\xrightarrow{\partial_{1}}\\&  H_0(L_r\cap L_s)\xrightarrow{(i_r,i_s)_0}  H_0(L_r)\oplus H_0(L_s)\to\cdots
	\end{align*}
	The map $(i_r,i_s)_{1}$ is injective, so the kernel of $(j_r)_{1}-(j_s)_{1}$ has dimension $b_1(L_r\cap L_s)$. 
	Therefore, $\im((j_r)_{1}-(j_s)_{1})=\im(j_r\circ i_r)_1$, identifying a copy of $H_1(L_r\cap L_s)\subset H_1(L(\phi))$.
	To characterize $H_1(L(\phi))/H_1(L_r\cap L_s)$, note that $\dim \ker((i_r,i_s)_0)= b_0(L_r\cap L_s)-1$, so $\dim(\im(\partial_1))=b_0(L_r\cap L_s)-1$. 
	Therefore, $\dim(H_1(L(\phi))/H_1(L_r\cap L_s))=b_0(L_r\cap L_s)-1$. 

	The connected components of  $L_r\cap L_s$ are in bijection with the surgery regions $\{U_v\}$ for $\sigma_0$ and $\sigma_{-\phi}$, which (by \prettyref{prop:latticepointsareintersections}) are in bijection with the non-self intersection lattice points $\Delta^\ZZ_\phi$. 
	This makes $\{e_v\}_{v\in \ZZ^n\cap \Delta_\phi}$ a basis for $H_0(L_r\cap L_s)$.
	The kernel of $(i_r, i_s)_0$, which may be thought of as the image of $\partial_1$, is spanned by pairs $\{e_{v_i}-e_{v_j}\}$.
	
	At each top strata $\underbar U_{\{v_i,v_j\}}$ of our tropical hypersurface, let $c_{\{v_i, v_j\}}$ be conormal fiber to the $\underbar U_{\{v_i, v_j\}}$ strata of our tropical Lagrangian.
	The cycles $c_{\{v_i, v_j\}}\subset H_1(L(\phi))$ and satisfy the property that $\partial_1(c_{\{v_i, v_j\}})=e_{v_i}-e_{v_j}$.
	Since the $\partial_1 (c_{\{v_i, v_j\}})$ span $\im(\partial_1)$, they generate $H_1(L(\phi))/H_1(L_r\cap L_s)$.

	We take the cycles from $H_1(L_r\cap L_s)$ and cycles $\{c_{\{v_i, v_j\}}\}$ to be a generating set for $H_1(L(\phi))$, and compute the integral of $\eta$ on these generators.

	To compute the integral of the Liouville form on a cycle $c$ included via  $i:H_1(L_r\cap L_s)\into H_1(L(\phi))$, we may always select a representative for $c$  which lies completely inside of the chart $L_r$. Since $L_r$ is exact, the Liouville form vanishes, $\eta(c)=0$ and we may conclude $\eta(i(H_1(V(\phi)))=0$.
	Therefore, the integral of the Liouville form on this portion of homology is independent of the choices.

	To compute the integral of the Liouville form on cycles which are of the form $c_{\{v_i, v_j\}}$, we give a local description of $L(\phi)$ containing the cycle $c_{\{v_i, v_j\}}$. We now need to make an assumption on the size of $\epsilon$,  so that at each facet $\underbar U_{\{v_i, v_j\}}$ there exists a point $p\in U_{\{v_i, v_j\}}$ with a sufficiently small neighborhood
	\[
		B_{2\epsilon}(p)\subset \left(\underbar U_{\{v_i, v_j\}}\cup\underbar U_{\{v_i\}}\cup\underbar U_{\{v_j\}} \right).
	\]
	We assume that $\epsilon_1, \epsilon_2$ are chosen to be smaller than this $\epsilon$.

	When restricted to $B_\epsilon(p)$, the tropical polynomial $\phi$ has only 2 domains of linearity, and so the restriction may be written as $\phi|_{B_\epsilon}= a_{v_i}x^{v_i}\oplus a_{v_j}x^{v_j}$.
	Because of the extra $\epsilon$ of room from the  larger neighborhood $B_{2\epsilon}(p)$, the smoothing $\tilde \phi|_{B_\epsilon}$ is only dependent on the local combinatorics of our strata.
	The function $\tilde \phi|_{B_\epsilon}$ is invariant with respect to translations in the  $(v_i-v_j)^\bot$ hyperplane.
	This means that the function $\tilde \phi|_{B_\epsilon}$ factors as $\tilde \psi(x^{v_i-v_j})$ where $\tilde\psi$ is the smoothing of a tropical function $a_0\oplus a_1t$.
	This function has the symmetry
	\begin{equation}
		d\tilde \psi(t)=- d\tilde \psi(-t)+1
		\label{eq:sectionsymmetry}
	\end{equation}

	We now will construct a cycle representing the homology class $[c_{\{v_i, v_j\}}]$. Consider the holomorphic cylinder $C$ which is the lift of the line
	\[
		\ell=t(v_i-v_j)+p\subset B_\epsilon
	\]
	to a complex curve $C=T^*\ell/T^*_{\ZZ}\ell\subset \val^{-1} B_\epsilon$.
	The intersection of $C$ with $L(\phi)|_{B_\epsilon}$ is a representative for the class $c_{\{v_i, v_j\}}$.
	Consider now the cycle $\hat {c}_{\{v_i, v_j\}}\subset C$ which is given by the  intersection  $C \cap N^*_{U_{\{v_i, v_j\}}}/(N^*_{U_{v_i,v_j}})_\ZZ$.
	If we can show that $\eta(c_{\{v_i, v_j\}})=\eta(\hat c_{\{v_i, v_j\}})$ , we will be finished as we will have shown that $\eta(c_{\{v_i, v_j\}})$ is independent of choices of $\alpha_i, \beta_i$.
	The difference between these two quantities $\int_{c_{\{v_i, v_j\}})-\hat c_{\{v_i, v_j\}}}\eta$ is now equated with the symplectic area between these two cycles on the holomorphic cylinder $C$.
	\begin{figure}
		\centering
		\begin{subfigure}{.4\linewidth}
			\includegraphics{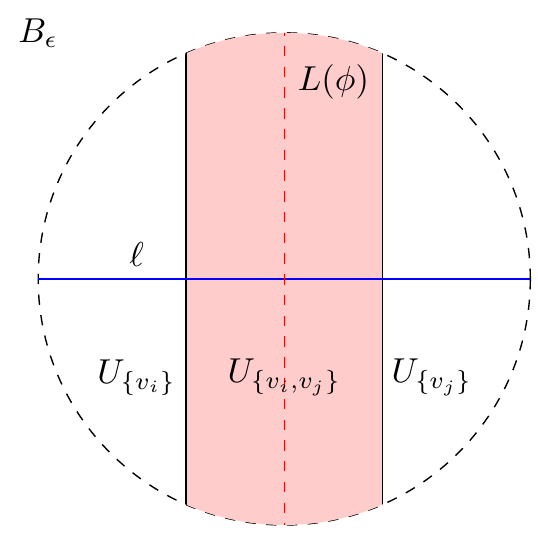}
			\caption{The projection of $L(\phi)$, $N^*U/N^*_\ZZ U$ and $C$ to the base $B_\epsilon$.}
		\end{subfigure}
		\begin{subfigure}{.4\linewidth}
			\includegraphics{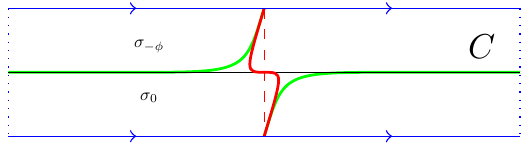}
			\caption{The restriction of $\sigma_0, \sigma_{-\phi}, c_{\{v_i, v_j\}}$ and $\hat c_{\{v_i, v_j\}}$ to the curve $C$. The curve $c_{\{v_i, v_j\}}$ is drawn in red, and $\hat c_{\{v_i, v_j\}}$ is dashed in red.} 
		\end{subfigure}
		\caption{Distinguished cycles for computing flux.
		}
		\label{fig:localsymmetry}
	\end{figure}
	We now recenter our coordinates so that $p$ is at the origin. In these coordinates, the symmetry from equation (\ref{eq:sectionsymmetry}) translates into the odd symmetry of the tropical section
	\[
		\sigma_\phi(t\cdot(v_i-v_j))=-\sigma_\phi(-t(v_i-v_j)).
	\]
	The cycle $c_{\{v_i, v_j\}}\subset C$ also inherits this symmetry.
	We may use this symmetry to decompose the integral
	\[
		\int_{c_{\{v_i, v_j\}}-\hat c_{\{v_i, v_j\}}}\eta
	\]
	into two equal cancelling components, and conclude that  $\eta(c_{\{v_i, v_j\}})=\eta(\hat c_{\{v_i, v_j\}})$.
\end{proof}
	A more detailed computation would show that by following the proof of \prettyref{prop:neckwidth} one can choose different surgery profiles at each region as long as all of the surgery profiles chosen have the same neck width.

\prettyref{prop:Hamiltonianisotopy} shows we can take an exact isotopy to bring  $\val (L(\phi))$ arbitrarily close to $V(\phi)$ if desired.
A difference between the valuation of a complex hypersurface (an amoeba) and the valuation of $L(\phi)$ is that the tentacles of the Lagrangian do not winnow off to zero radius as we increase valuation.

\subsubsection{Admissibility of Tropical Lagrangians}
In order for us to study this object in the Fukaya category, we will have to show that it satisfies the admissibility conditions of $\Fuk_{\Delta_\Sigma}(X, W_\Sigma)$, and show that this Lagrangian is unobstructed.
\begin{prop}
	Suppose that the section $\sigma_{-\phi}$ is $\Delta_\Sigma$ admissible. Then $L(\phi)$ is $\Delta_\Sigma$ admissible.  \label{prop:admissibility}
\end{prop}
\begin{proof}
	To show that $L(\phi)$ is admissible, we need show that over each region $C_\alpha$ from \prettyref{def:admissiblitycondition}, the argument of $c_\alpha z^\alpha$ is zero outside of a compact set.
	If we associate the exponent $\alpha$ to a vector $\alpha\in \ZZ^n$, the admissibility condition can be restated as the argument of $L$ lying inside the $T^{n-1}_{\alpha^{\bot}}$ subtorus of the fibers $F_q$ of $(\CC^*)^n$ over the regions of $C_\alpha$.
	Both $\sigma_{-\phi}$ and $\sigma_0$ are constrained within this subtorus as $\arg(\sigma_{-\phi}), \arg(\sigma_0)\subset T^{n-1}_{\alpha^\bot}$ on the region $C_\alpha$.
	\prettyref{prop:argument} allows us to conclude that $\arg(\sigma_{0}\#_U\sigma_{-\phi}	)= \arg(L(\phi))$ is similarly contained within the region $C_\alpha$.
\end{proof}

\section{Unobstructedness of Tropical Lagrangians}
\label{sec:unobstructedness}
The Lagrangian $L(\phi)$ is not necessarily monotone.
However, as a corollary of \prettyref{prop:Hamiltonianisotopy}, the only disks which may appear on our tropical Lagrangian have Maslov index 0.
\begin{corollary}
	The Maslov class of $L(\phi)$ is trivial in $H^1(L, \ZZ)$.
\end{corollary}
\begin{proof}
	The cycles in homology which arise from $H_1(V(\phi))$ will all live in the zero section, and near these cycles $L(\phi)$ agrees with the zero section and therefore has Maslov class zero. For the other generators of homology, take a Hamiltonian isotopy of our tropical Lagrangian so that in a neighborhood of a cycle $c_{\{v_i, v_j\}}$ the Lagrangian $L(\phi)$ agrees with $N^*U_{\{v_i,v_j\}}/N^*_\ZZ U_{\{v_i, v_j\}}$ (see the proof of \prettyref{prop:Hamiltonianisotopy}). This is a special Lagrangian, and so the Maslov class of $c_{\{v_i, v_j\}}$ is zero as well.
\end{proof}
Hence, every disk with boundary on $L(\phi)$ has Maslov index zero.
As a result, in low dimensions disks will show up in negative dimensional families, and therefore do not appear for regular $J$.
\begin{corollary}
	For $n=1, 2$, the Lagrangians $L(\phi)\subset (\CC^*)^n$ bound no holomorphic disks for regular perturbations.  \label{cor:lowdimensionnodisk}
\end{corollary}
In the general case, we cannot hope for this kind of unobstructedness result.
In \cite{hicks2019tropical}, we show that there exists almost complex structures so that $L(\phi)$ bounds disks of Maslov index 0 by finding a mutation structure on Lagrangian submanifolds.
See \cite[Remark 1.4]{sheridan2018lagrangian} for a similar discussion on the presence of holomorphic disks on tropical Lagrangians.
Instead, we show that the pearly $A_\infty$ algebra of this Lagrangian cobordism is unobstructed by bounding cochain.

\subsection{Pearly Model for Lagrangian Floer Theory}
\label{subsec:assumptions}
To define a bounding cochain, one needs to associate to a Lagrangian $L\subset X$ a filtered $A_\infty$ algebra $\CF(L)$.
We choose to use the pearly model for the $A_\infty$ algebra, where the Morse $A_\infty$ algebra $\CM(L, h)$ is deformed by insertion of holomorphic disks at the vertices of flow trees \cite{fukaya1993morse}.
The construction of this algebra has been carried out with a variety of conditions placed for Lagrangians $L$ and symplectic manifolds $X$ \cite{biran2008lagrangian,charest2015floer,lipreliminary}. 
Unfortunately, the setting which we are interested in --- pearly Floer cohomology for non-monotone and non-compact Lagrangian submanifolds --- has not to our knowledge been constructed in the literature. 
Rather, existing definitions of the pearly complex each cover a portion of these requirements. 

\subsubsection*{Existing Constructions and Their Properties:}
We give an example of a model where all the properties of the Floer cohomology we will need have been checked, albeit with stricter requirements on the Lagrangian submanifolds than we will have in our application.
In \cite{biran2008lagrangian} and the related work \cite{biran2013lagrangian}, a quantum homology algebra $QH_\bullet(L, h)$ is constructed for Lagrangians $L$ equipped with Morse functions $h$.
We will use cohomological indexing $QH^\bullet(L, h)$ as opposed to homological indexing to remain consistent with the notations of this paper.
These Lagrangians $L\subset X$ are allowed to be non-compact, provided that the Lagrangians $L$ are monotone and convex at infinity.
The  Lagrangian Floer theory that they construct is well defined in Lefschetz fibrations, and exhibits a decomposition rule.
\begin{definition}
	Let $W:X\to \CC$ be a symplectic Lefschetz fibration, and $L\subset X$ a Lagrangian submanifold.
	We say that $L$ is \emph{bottlenecked} at $z_0\in \CC$ if
	\begin{itemize}
		\item $W(L)\cap B_\epsilon(z_0)$ is an embedded curve in $B_\epsilon(z_0)\subset \CC$ which passes through $z_0$,
		\item The fiber $L_{z_0}:=L\cap W^{-1}(z_0)$ disconnects $L$, so that we may write $L\setminus L_{z_0}=L^-\sqcup L^+$. 
		Furthermore, the projections of these components are disjoint:
		\[W(L^-)\cap W(L^+)=\emptyset\]
	\end{itemize}
	We say that a Morse function $h:L\to \RR$ has a bottleneck at $z_0$ if, after identifying $L\cap W^{-1}(B_\epsilon(z_0))= L_{z_0}\times I$,
	\begin{itemize}
		\item The gradient of $h$ points outward along the boundaries of $L_{z_0}\times I\subset L$,
		\item The critical points of $h$ contained in the bottleneck region all project to the bottleneck value, \[W(\Crit(h)\cap(L_{z_0}\times I))=z_{0},\]
		\item Every flow line $\gamma(t)$ of $\grad(h)$ either flows out of the pinched region,
			\[\lim_{t\to\infty} W(\gamma(t))\not\in B_\epsilon(z_0)\]
			or has constant projection under $W$,
			\[W(\gamma(t))=z_0.\]
	\end{itemize}
\end{definition}
Bottlenecks constrain holomorphic disks which pass through the bottleneck region to lie in fibers of $W$, and bottlenecked Morse functions similarly have flow lines constrained to either $L_{z_0}, L^+$ or $L^-$.
\begin{lemma}[(Paraphrasing Lemma 4.2.1 of \cite{biran2013lagrangian})]
		Suppose that $W:X\to \CC$ is a symplectic Lefschetz fibration. 
		Let $L\subset X$ be a monotone Lagrangian submanifold with bottleneck at  $z_0$.
		There exists a regularizing perturbation for the Cauchy-Riemann equation so that all solutions to the perturbed pseudo-holomorphic disk equation $u: (D^2, \partial D^2)\to (X, L)$ are regular and will either have  $W(u)\subset \CC\setminus B_\epsilon(z_0)$, or $W(u)=z$, a point. 
		\label{lemma:constrainedtofiber}
\end{lemma}
The proof uses the open mapping principle.
In the setting where one uses a Hamiltonian perturbed Cauchy-Riemann equation, one must show that this perturbation can be chosen to preserve the open mapping principle in a region of the bottlenecks. 
\begin{corollary}
	If $L\subset X$ is compact, monotone, and bottlenecked at $z_0$ so that $L=L_-\sqcup L^+\sqcup L_{z_0}$, there exists a choice of  Hamiltonian perturbation so that every perturbed pseudo-holomorphic disk with boundary on $L$ is regular and the boundary is contained in $L_-, L_{z_0}$ or $L_+$. 
	\label{cor:bottleneckdiskboundary}
\end{corollary}

The application of this lemma is two-fold.
The first is to show that the pearly-complex for $QC^\bullet(L,h)$ is well defined for an admissible class of non-compact Lagrangian submanifolds and Morse functions.

\begin{definition}
	Let $W:(\CC^*)^n\to \CC$ be a symplectic Lefschetz fibration.
	A Lagrangian $L$ and Morse function $h:L\to \RR$ is \emph{$W$-admissible} if 
	\begin{itemize}
	\item There exists a compact subset $V\subset L$ outside which $W(L\setminus V)\subset \RR_{\geq 0}$. 
	\item The Morse function $h$ is bottlenecked at $z_\infty\in W(L\setminus V)$
	\item $W(\Crit(h))\subset W(V)\cup\{z_\infty\}$.
	\end{itemize}
	\label{def:wadmissible}
\end{definition}
The corollary can be used to prove a Gromov-compactness result for $W$-admissible monotone Lagrangian submanifolds of $X$, as every pearly trajectory with output on a critical point will have holomorphic disk insertions with boundary contained in $W(V)\cup[0, z_\infty]$.
\begin{theorem}[(Paraphrasing Theorem 2.3.1 of \cite{biran2008lagrangian})]
	For $W$-admissible monotone Lagrangian submanifolds, the quantum chain complex $QC^\bullet(L, h)$ is a deformation of the Morse complex $CM^\bullet(L, h)$.
	The homotopy type of the complex is independent of choices made during the construction of $QC^\bullet(L, h)$, including the Morse function.
\end{theorem}
The second application of bottlenecks is an algebraic decomposition of the Lagrangian quantum cohomology.
Let $L$ be a $W$-admissible Lagrangian submanifold bottlenecked at $z_0$.
The condition of being bottlenecked means that our Lagrangian $L$ looks like the concatenation of two Lagrangian cobordisms in a neighborhood of $z_{0}$.
Let $L\setminus L_{z_0}= L^-\cup  L^+$.
Now, pick $h$ a Morse function for $L$  bottlenecked at $z_0'\in W(L^-)\cap B_\epsilon(z_0)$.
Recall that if $L^+$ is a manifold whose closure has boundary, the Morse cohomology relative to the boundary of the closure is computed by taking a Morse function whose flow is transverse and points inwards from the boundary $\partial \bar  L^+$. 
The Morse complex of  $(L, h)$ splits as a vector space
\[
	\CM(L, h)=\CM(L^-, h|_{L^-})\oplus \CM(L^+, \partial L^+,h|_{L^+}).
\]
In this setup, the flow of $h$ is transverse to the fiber $L_{z_0}$, and flows from the $L^-$ component to the $L_+$ component. 
As a result, $\CM(L^+, \partial L^+)$ is a subcomplex (in fact, an $A_\infty$ ideal) of $\CM(L)$, and
\[
	\CM(L^-, h|_{L^-})= \CM(L, h)/\CM(L^+, \partial L^+, h|_{L^+}).
\]
Because \prettyref{cor:bottleneckdiskboundary} constrains disk boundaries inside $L^+$, every treed disk which intersects $L^+$ must have boundary completely contained within $L^+$.
This extends the Morse decomposition of the Lagrangian quantum cohomology :
\begin{prop}
	Let $L\subset X$ be a monotone $W$-admissible Lagrangian submanifold.
	Then for the Morse function $h$ described above we have a splitting of the pearly complex as a vector space:
	\[
		QC^\bullet(L, h)=QC^\bullet(L^-, h|_{L^-})\oplus QC(L^+, \partial L^+, h|_{L^+}).
	\]
	Furthermore, $QH^\bullet(L^+, \partial L^+,h|_{L^+})$ is an ideal of $QH(L,h)$. 
	\label{prop:knownsplitting}
\end{prop}
\begin{remark}
	In fact, the decomposition gives a slightly stronger result.
	The restriction $L_{z_0}\subset W^{-1}(z_0)$ is a Lagrangian submanifold of the fiber, and the decomposition of \prettyref{cor:bottleneckdiskboundary} states that there is an inclusion of chain complexes $i_{z_0'}: QC^\bullet(L_{z_0}, h|_{L^-})\to QC^\bullet(L, h)$, which induces a map on the quantum cohomology algebras.
\end{remark}

With these examples from the monotone setting in mind, we take a look at constructions of the pearly model in the non-monotone setting. These constructions are generally more involved, as global geometric perturbations of Floer data are insufficient to achieve transversality of the perturbed Cauchy-Riemann equation.
We will look at the approach to the pearly model which achieves regularity of moduli spaces of holomorphic treed disks through the use of domain dependent perturbations. 

\begin{theorem}(Paraphrasing Theorem 4.1 of \cite{charest2015floer})
	Let $X$ be a compact symplectic manifold with rational symplectic form. Let $L\subset X$ be an embedded Lagrangian submanifold, and $h:L\to \RR$ a Morse function.
	There exists a filtered $A_\infty$ algebra $\CF(L, h)$ whose construction involves choices of stabilizing divisor and domain dependent perturbation data.
	Different choices of perturbation data and divisors produce filtered $A_\infty$ homotopic algebras.
	\label{thm:compactfloerdefinition}
\end{theorem}
A key ingredient in the construction is the stabilizing divisor $D\subset X\setminus L$, which has the property that every disk with positive symplectic area and boundary on $L$ intersects the divisor $D$.
This provides interior marked points on every disk with boundary on $L$, which can be used to stabilize the domain of the disk and construct a domain dependent perturbed Cauchy-Riemann equation.
This allows \cite{charest2015floer} to achieve transversality, as the space of domain dependent perturbations is much larger (and in particular can handle the cases of non-injective points). 

\subsubsection*{Needed constructions and properties of $\CF(L, h)$:}
However, to our knowledge, an analogue of \prettyref{lemma:constrainedtofiber} has not been proved in the non-monotone setting, nor has stabilizing divisor perturbation method been extended to the non-compact setting.
As the construction of this category is beyond the scope of this paper, we will provide a set of properties of the pearly model which we will be using throughout this paper, and an argument sketching how to augment the proof of \cite[Theorem 4.1]{charest2015floer} to these assumptions.

\begin{assumption}(Extension of Theorem 4.1 of \cite{charest2015floer})
	Let $(X, W)$ be a symplectic Lefschetz fibration, and let $(L, h)$ be $W$-admissible. 
	The pearly Floer complex is a filtered $A_\infty$ algebra (\prettyref{def:filteredainfty}) which satisfies the following properties:
	\begin{itemize}
		\item $\CF_{=0}(L, h)$, the zero-valuation portion of the $A_\infty$ algebra, is chain isomorphic to $\CM(L, h)$, the Morse cochain complex of $L$.
		The valuation of this deformation is bounded below by the symplectic area of the smallest holomorphic disk with boundary on $L$ for the choice of regularizing perturbation.
		\item If $(L', h')$ is also $W$-admissible, and $L'$ is Hamiltonian isotopic to $L$, then there exists a continuation $A_\infty$ homotopy equivalence $f:\CF(L, h)\to \CF(L', h')$, which is a deformation of the continuation map for Morse cohomology.
		The valuation of this deformation is bounded below by the symplectic area of the smallest holomorphic disk with boundary on $K_H$, the suspension cobordism of the Hamiltonian isotopy between $L$ and $L'$. 
		\item The $A_\infty$ homotopy equivalence class of $\CF(L, h)$ is independent of choices made in the construction for domain dependent perturbation datum. 
	\end{itemize}
	\label{assum:welldefinedfloer}
\end{assumption}

\begin{assumption}(Extension of \prettyref{prop:knownsplitting} )
	Suppose that $L\subset X$ is a $W$-admissible Lagrangian, with bottleneck and Morse function chosen as in \prettyref{prop:knownsplitting}.
	Then there exists a splitting of vector spaces (arising from the decomposition of Morse theory)
	\[
		\CF(L, h)=\CF(L^-, h|_{L^-})\oplus \CF(L^+, \partial L^+,h|_{L^+}).
	\]
	Furthermore, $\CF(L^+, \partial L^+,h|_{L^+})$ is an $A_\infty$ ideal of $\CF(L,h)$. 
	\label{assum:bottleneckideals}
\end{assumption}

\emph{Justification of \prettyref{assum:welldefinedfloer}.}
We outline how the compactness result  \cite[Theorem 4.27]{charest2015floer} could be extended to the case where $(L, h)$ is admissible and $X$ is non-compact.
Pick a connected  subset $V'\subset \CC$ such that $V\subset V'$, $z_\infty\in V'$, and the gradient of $h$ points outward along $\partial V$. 
A \emph{perturbation system} \cite[following Definition 4.26]{charest2015floer} for a combinatorial type $\Gamma$ of flow tree is a choice of domain-dependent Morse functions
\[F_\Gamma:\overline{\mathcal T}_{\circ, \Gamma}\times L\to \RR\]
and domain-dependent almost complex structures 
\[J_\Gamma:\overline{\mathcal TS}_{\circ, \Gamma}\times L\to \End(TX).\]
We say that a perturbation system is $W$-admissible if (in addition to satisfying the requirements of \cite[Definition 4.10]{charest2015floer}) these perturbations are trivial outside of the compact set $V'$, 
\begin{align*}  
	F_\Gamma((C,x),z)|_{W(z)\not\in V'}=\pi_2^*h && J_\Gamma((C, x),z)|_{W(z)\not\in V'}=\pi^*_2J.
\end{align*}
where $\pi_2$ is projection onto the $L$-factor of $\overline{\mathcal T}_{\circ, \Gamma}\times L$ or $\overline{\mathcal TS}_{\circ, \Gamma}\times L$.
The space of such perturbations is denoted  $\mathcal P^l_\Gamma(X, D)$.
As outside of the region $V'$ the perturbations are not domain dependent, the open mapping principle may still be applied to show that the  disk components of treed disks may not escape the region $V'$.
As the flow of $h$ points outward of $V'$, we obtain that every $W$-admissibly perturbed holomorphic treed disk with boundary on $L$ has projection contained within $V'$. 

It remains to show that we can extend the results of \cite[Theorem 4.19]{charest2015floer} to find a co-meager subset of $W$-admissible perturbation data for which the domain-perturbed Cauchy-Riemann equations are regular.
The idea, following the original proof, is to consider the space of maps of treed disks and $W$-admissible perturbation datum,
\[\mathcal B^i_{k,p,l, \Gamma}=\mathcal M_\Gamma^i\times \Map^{k,p}_\Gamma(C, X, L, D)\times \mathcal P^l_\Gamma(X, D)\]
where $\Map^{k,p}_\Gamma(C, X, L, D)$ is the space of maps with prescribed incidences at boundaries and interior marked points and $W^{k,p}$ Sobelov regularity. 
$\mathcal M_\Gamma^i$ is an open set of $\mathcal M_\Gamma$, the moduli space of treed disks.
We then consider a Banach bundle $\mathcal E^i_{k,p,l, \Gamma}\to \mathcal B^{i}_{k,p,l,\Gamma}$ with a section $\bar\partial_\Gamma: \mathcal B^{i}_{k,p,l,\Gamma}\to  \mathcal E^{i}_{k,p,l,\Gamma}$ encoding the $(F, J)$-perturbed Cauchy-Riemann and Morse equations. 
The \emph{local universal moduli space} \cite[Equation 4.15]{charest2015floer} is the zero locus of $\bar \partial^{-1}(\mathcal B^{i}_{k,p,l,\Gamma})$.
The requirement that our perturbations be $W$-admissible implies that all points $(m, u, J, F)\in \bar \partial^{-1}(\mathcal B^{i}_{k,p,l,\Gamma})$ will correspond to maps $u$ with image   $W(u)\subset V'$ by the open mapping principle.
Regularity for a co-meager set of perturbations is equivalent to the local universal moduli space being cut out transversely. 
At points $(m, u, J, F)\in \bar \partial^{-1}(\mathcal B^{i}_{k,p,l,\Gamma})$ the check of transversality proceeds as in the original proof of \cite[Theorem 4.19]{charest2015floer}, as  achieving surjectivity of the linearized $\bar \partial$ operator is possible using perturbations with support in a neighborhood of the image of $u$.
\endproof

Provided that $W$-admissible perturbations can be used to give an extension of \cite[Theorem 4.1]{charest2015floer}, we can give a version of \prettyref{cor:bottleneckdiskboundary}.
\begin{definition}
	Consider a $W$-admissible Lagrangian submanifold $L\subset X$, with choice of Morse function as in the setup of \prettyref{assum:bottleneckideals}. We say that a labelling $\{y;x_1, \ldots, x_k\}\subset \Crit(h)$ is a \emph{split labelling} if $y\in L^-$ and at least one of the $x_i$ are in $L^+$.
\end{definition}
\begin{prop}
	There exists a choice of coherent domain-dependent perturbations so that whenever $\underbar x = \{y;x_1, \ldots, x_k\}$ is a split labelling for a treed disk $\Gamma$, the moduli space of treed disks $\mathcal M_\Gamma(L, D)$ is empty.
\end{prop}

\emph{Sketch of Proof.}
	Suppose that $\Gamma$ has a split labelling. Then, the open mapping principle again implies that for the standard choice of Morse function and almost complex structure, $\mathcal M_\Gamma(L, D)$ is empty, and (therefore trivially!) cut out transversely.
	It remains to show that this can be extended to a coherent choice of perturbations for the entire moduli space. 
	The construction of a coherent perturbation is iterative, using the order induced on the $\Gamma$'s by graph morphisms, (see statement of \cite[Theorem 4.19]{charest2015floer}).
	Whenever $\Gamma'\prec \Gamma$ in this ordering, a split labelling on $\Gamma$ induces a split labelling on $\Gamma'.$
	The most important of these cases proves that perturbation produces coherent boundary strata: if $\Gamma$ is obtain from $\Gamma'$ by gluing at a breaking, either the portion of $\Gamma'$ above the break or the portion of $\Gamma'$ below the break has split label.  
	Since the split label condition is downward closed under the iteration order used to construct coherent perturbation schemes, we can construct a coherent perturbation by initially choosing the trivial perturbation for all the split labels, and then proceeding to choose perturbations for the trees whose labels are not split using \cite[Theorem 4.19]{charest2015floer}.
\endproof
We will use this proposition as a substitute for the stronger statement \prettyref{cor:bottleneckdiskboundary} to prove an analogue of \prettyref{prop:knownsplitting}.

\emph{Justification of \prettyref{assum:bottleneckideals}.}
	The condition that $\CF(L^+, \partial L^+, h|_{L^+})$ be an $A_\infty$ ideal (\prettyref{def:ideal}) is that $m^k(x_1, \ldots, x_k)\in \CF(L^+, \partial L^+, h|_{L^+})$ whenever at least one of the $x_i\in \CF(L^+, \partial L^+, h|_{L^+})$. 
	This can be rephrased as
	\[\langle m^k(x_1, \ldots, x_k), y\rangle=0\]
	whenever $\{y;x_1, \ldots, x_k\}$ is a split label. 
	This is equivalent to the condition that the moduli space of treed disks with split labels is empty. 
\endproof
\begin{remark}
On the abstract perturbation side, the work of \cite{lipreliminary} and \cite{fukaya2000Lagrangian} construct methods for regularizing the moduli space of holomorphic disks using polyfolds and Kuranishi structures respectively. 
Both of these approaches associate to a general compact Lagrangian submanifold $L\subset X$ an $A_\infty$ algebra.
In the second approach, the $A_\infty$ homotopy class of the algebra is shown to be free of choices and invariant under Hamiltonian isotopy. 
In both constructions, the $A_\infty$ relations are a result of a coherent choice of abstract perturbations. 
As in the stabilizing divisor case, these coherent abstract perturbation data are constructed iteratively (see, for instance, \cite[Section 6.2]{filippenko2018polyfold} or \cite[Section 11.5]{lipreliminary},) which should allow the strategy of proof outlined above for the domain-dependent case to be extended to these regularization techniques as well.
The employment of these techniques to prove \prettyref{assum:welldefinedfloer} and \prettyref{assum:bottleneckideals} is an ongoing project of the author. 
\end{remark}
When the Morse function is unimportant, we will simply write $\CF(L)$.

\subsection{Proof of unobstructedness}
As noted before, to prove that the pearly algebra is unobstructed, we will require the extension of the pearly model to the non-compact setting with properties satisfying \prettyref{assum:welldefinedfloer}. 
We will also need the extension of the splitting of Floer theory in symplectic fibrations, as stated by \prettyref{assum:bottleneckideals}.
\begin{prop}
	$\CF(L(\phi))$, the Floer algebra defined by \prettyref{assum:welldefinedfloer} and satisfying \prettyref{assum:bottleneckideals},  is unobstructed by bounding cochain.
	\label{prop:unobstructed}
\end{prop}
The proof is in three steps. 
We first describe a geometric relation between $L(\phi)$ and $L^{tr}(\phi)$, the Lagrangian obtained by taking a surgery between $\sigma_0$ and $\sigma_{-\phi}$ after applying a perturbation to make intersections transverse (see \prettyref{prop:ltr}).
We then show that $\CF(L(\phi))$ may be expressed as a quotient of the Floer theory of $\CF(L^{tr}(\phi))$.
Finally, we show that $\CF(L^{tr}(\phi))$ is unobstructed by bounding cochain.
\subsection*{Unobstructedness: some tools}
Before proving \prettyref{prop:unobstructed}, we look at three tools largely independent from the discussion of tropical Lagrangians.
The first is a comparison between our tropical Lagrangians and the geometry of symplectic fibrations.
The second is a statement about bottlenecked Lagrangians and their bounding cochains.
The third is an existence result for bounding cochains on sequences of Lagrangians converging to tautologically unobstructed Lagrangians.
\subsubsection*{Tropical Symplectic Fibrations.}
\label{par:tropicalsymplectic} 
We now summarize a discussion in \cite{hanlon2018monodromy} relating the monomial admissibility condition (\prettyref{def:admissiblitycondition}) the bottleneck condition \prettyref{def:wadmissible}  by a construction from \cite{abouzaid2009morse} called the tropically localized superpotential. 
Starting with an ample divisor $D=\sum v_\alpha D_\alpha$ on $\check X_\Sigma$ with polytope $P=\Delta_{\phi_D}$, the Newton polytope of the support function $\phi_D$, Abouzaid constructs a family of superpotentials $W_{t, 1}: (\CC^*)^n\to \CC$ for $t\geq 1$ and $s\in [0,1]$ such that $W_{\Sigma}=W_{1, 0}$. 
The fiber $M_{t, 1}:=W_{t, 1}^{-1}(1)$ has valuation projection $\val(M_{t, 1})$ which lives close to a tropical variety for $t$ sufficiently large.
Additionally, $Q\setminus \val(M_{t, 1})$ has a distinguished connected component $\mathcal P_{t, 1}$ which can be rescaled to lie close to $P$. 
Near the boundary of $P$, the tropically localized superpotential can be explicitly written as
\[
	W_{t, 1}=\sum_{\alpha} t^{-v_\alpha}(1-\rho_\alpha(z))z^\alpha
\]
where the $\rho_\alpha$ are smooth non-negative real-valued functions. The functions $\rho_\alpha(z)$ only depend on $\val(z)$, and are constructed so that 
\begin{itemize}
	\item Whenever $\val(z)$ is outside a small neighborhood of the dual facet of $\alpha$, $\rho_\alpha(z)=1$. 
	\item Whenever $\val(z)$ is nearby the dual face of $\alpha$, $\rho_\alpha(z)=0$. 
\end{itemize}
The upshot of working with the tropically localized superpotential is the following.
With the usual superpotential, the fiber $W_\Sigma^{-1}(1)$ should roughly have a decomposition into regions where the subsets of monomials of $W_\Sigma$ codominate the other terms. On each of these regions, $W_\Sigma$ is approximately equal to the dominating monomials. 
With the tropically localized superpotential the hypersurface $M_{t, 1}$ similarly admits a decomposition, however $W_{t, 1}$ honestly matches the dominating monomials on each region of domination.
We call $W_{t, 1}$ the tropically localized $W_\Sigma$. 

The monomial admissibility condition for $W_\Sigma=\sum c_\alpha z^\alpha$ only requires that each monomial term $z_\alpha$ dominates in the region $C_\alpha$ after possibly being raised to some power $k_\alpha$. We may assume that the $k_\alpha$ are rational, and therefore find an integer $N$ and rescalings $c_\alpha^N$ of $c_\alpha$ defining a new Laurent polynomial 
\[
	\tilde W_{N\Sigma} := \sum_{\alpha} c_\alpha^N z^{\alpha \cdot k_\alpha\cdot  N}.
\]
Associated to this $\tilde W_{N\Sigma}$,  we obtain a Newton polytope $P_{N}\subset Q$ containing the valuation of points $\val(\tilde W_{N\Sigma}^{-1}(B_1(0)))$.
As we increase $N$, the polytopes $P_N$ scale to cover all of $Q$. 
Therefore, we may additionally assume that $N$ is chosen large enough so that a given monomially admissible Lagrangian $L$ satisfies the monomial admissibility condition in a neighborhood of the boundary of $\val^{-1}(P_N)$. 

Set $W_{t, 1}$ to be the tropically localized $W_{N\Sigma}$. Since the tropically localized superpotential only involves the monomials $z^\alpha$ for which $\val(z)\in C_\alpha$, and $z^\alpha(L)\in \RR_+$ over the region where $L$ meets $\val^{-1}(P_N)$, we may conclude:

\begin{lemma}[(Section 4.4 of \cite{hanlon2018monodromy})]
	Suppose that $L$ is a monomial admissible (in the sense of \prettyref{def:admissiblitycondition}) submanifold.
	Then $L\cap \val^{-1}(P)$ is a $W_{t, 1}$-admissible (in the sense of \prettyref{def:wadmissible}) submanifold with boundary on $M_{t, 1}$. 
	\label{lemma:equivadmissibility}
\end{lemma}
See \prettyref{fig:tropicaladmissibility} for a diagram of $W_{t, 1}(\sigma_{-\phi})$ and $W_{t, 1}(L(\phi))$. 

\subsubsection*{Bottlenecked Lagrangians and Bounding Cochains.}
Bottlenecks not only provide a method for producing admissible Lagrangians in the non-compact setting, but they also give a several useful decompositions of the Floer cohomology which can be used to produce bounding cochains.
This is where we will employ the expected property of the pearly algebra stated in \prettyref{assum:bottleneckideals}.
Here are three observations on bottlenecked Lagrangians and bounding cochains. 
\begin{prop}
	Suppose that $L\subset X$ is a $W$-admissible Lagrangian, bottlenecked at $z\in \CC$, giving us a decomposition of $L=L_-\cup L_{z_0}\cup L_+$. 
	Take a Morse function as described in \prettyref{assum:bottleneckideals} which gives us a decomposition 
	\[
		\CF(L, h)=\CF(L^-, h|_{L^-})\oplus \CF(L^+, \partial L^+,h|_{L^+}).
	\]
	Suppose that $L^+=L_{z_0}\times \RR_{>0}$. Then $\CF(L, h)/\CF(L^+, \partial L^+,h|_{L^+})$ is isomorphic to $\CF(L, h)$.
	\label{prop:bottlenecktrivialneck}
\end{prop}
\begin{proof}
	If $L^+=L_{z_0}\times \RR_{>0}$, the $L^+$ can be given a Morse function with no critical points.
\end{proof}
\begin{prop}
	Let $L_1, L_2$ be two Lagrangian submanifolds which are bottlenecked at the same point $z$. 
	Suppose that $L_1^-=L_2^-$, and the Morse functions $h_1, h_2$ for $L_1$ and $L_2$ agree over that region. 
	Then
	\[
		\CF(L_1, h_1)/ \CF(L^+_1, \partial L^+_1,h_1|_{L_1^+})= \CF(L_2, h_2)/ \CF(L^+_2, \partial L^+_2,h_2|_{L^+_2}).
	\]
	\label{prop:agreebottleneck}
\end{prop}
\begin{proof}
	Following the justification of \prettyref{assum:bottleneckideals}, pick perturbations for both $\CF(L^1)$ and $\CF(L^2)$ by first picking perturbations for the non-split labels, and then for the split labels.
	Since $L_1, L_2$ and $h_1, h_2$ match completely upon restriction to $L^-_1=L^-_2$, we can choose regularizing perturbations for non-split labels belonging completely to $\Crit(h)|_{L^-_1}=\Crit(h)|_{L^+_1}$ so that the moduli spaces of disks exactly match.
	Therefore, the $A_\infty$ operations agree for $\CF(L_1, h_1)$ and $\CF(L_2, h_2)$ when restricted to inputs which lie entirely in the negative components. 
\end{proof}
\begin{prop}
	Suppose that $L\subset X$ is bottlenecked at $z\in \CC$. 
	Then if $L$ is unobstructed, $\CF(L, h)/ \CF(L^+, \partial L^+,h|_{L^+})$ is unobstructed as well.
\end{prop}
\begin{proof}
	Follows from \prettyref{lemma:pushforwardmad}, which states that if the domain of an $A_\infty$ filtered homomorphism in unobstructed, then the codomain can be unobstructed by a pushforward bounding cochain. 
	Applying this to the projection map from $\CF(L_1, h_1) \to \CF(L_1, h_1)/ \CF(L^+, \partial L^+_1,h_1|_{L_1^+})$ proves the proposition.
\end{proof}
\begin{corollary}
	Suppose that $L$ is bottlenecked, and $\CF(L)$ admits a bounding cochain.
	Then $\CF(L^-)$ admits a bounding cochain. 
	\label{cor:bottleneckunobstructed}.
\end{corollary}

\subsubsection*{Eventually Unobstructed Lagrangians and Bounding Cochains.}
This is the portion of the proof where we use the continuation map property of \prettyref{assum:welldefinedfloer}.
\begin{definition}
	Let $\{L_\alpha\}_{\alpha\in \NN}$ be a sequence of Hamiltonian isotopic Lagrangian submanifolds.
	We say that this sequence is \emph{eventually unobstructed} if for every energy level $\lambda$, there exists $\alpha_\lambda$ so that $\beta\geq \alpha_\lambda$ implies that $L_{\beta}$ bounds no holomorphic disks of energy less than $\lambda$ belonging to treed disks contributing to $\CF(L_\beta)$.
\end{definition}
\begin{lemma}
	Let $\{L_\alpha\}_{\alpha\in \NN}$ be a sequence of Hamiltonian isotopic Lagrangian submanifolds, and let  $\{K_{\alpha,\alpha+1}\}_{\alpha\in \NN}$ be the sequence of suspension cobordisms corresponding to choices of Hamiltonian isotopies between $L_\alpha$ and $L_{\alpha+1}$.
	Suppose both $\{L_\alpha\}_{\alpha\in \NN}$ and $\{K_{\alpha,\alpha+1}\}_{\alpha\in \NN}$ are eventually unobstructed sequences of Lagrangian submanifolds. 
	Then $L_0$ is unobstructed by bounding cochain.
	\label{lemma:unobstructedinlimit}
\end{lemma}
\begin{proof}
	The suspension cobordism $K_{\alpha, \beta}$ given by the concatenation of our Lagrangian cobordisms correspond to continuation maps 
	\[
		f_{\alpha,\beta}:\CF(L_\alpha)\to \CF(L_\beta)
	\]
	as defined in \prettyref{assum:welldefinedfloer}, which satisfy the property that
	\[
		f_{\beta,\gamma}\circ f_{\alpha,\beta}= f_{\alpha,\gamma}.
	\]
	We note that the valuation of $f_{\alpha+1, \alpha}^0$ goes to infinity. 
	This is because the  valuation of $f_{\alpha+1, \alpha}^0$ can be bounded below by the energy of the smallest holomorphic disk which occurs in the Hamiltonian suspension cobordism $K_{\alpha,\alpha+1}$ between $L_\alpha$ and $L_{\alpha+1}$. 
	By hypothesis of the lemma, the minimal energy of these holomorphic disk goes to $\infty$ as $\alpha$ goes to $\infty$.

	For Lagrangian $L_\alpha$, we let $m^k_\alpha$ be the $A_\infty$ structure on $\CF(L_\alpha)$.
	For a deforming chain $d\in \CF(L_\alpha)$, we let $(m^k_\alpha)_d$ be the deformed curved $A_\infty$ structure.
	By our assumption, for each $\lambda$ there is a $\alpha_\lambda$ so that $\beta\geq \alpha_\lambda$ implies that  $\val(m^0_\beta)$ will be greater than $\lambda$. 

	Given a deformation $b_\alpha\in \CF(L_\alpha)$ and a filtered $A_\infty$ homomorphism $f_{\alpha,\beta}$ as above, we get a pushforward map on deformations
	\[
		b_\beta=(f_{\alpha,\beta})_*(b_\alpha):= \sum_{k\geq 0} f^k_{\alpha,\beta}(b_\alpha^{\tensor k}).
	\]
	If $b_\alpha$ is a bounding cochain for $\CF(L_\alpha)$, then this pushforward is again a bounding cochain.
	The same is true for deformations which are bounding cochains up to a low valuation.
	\begin{prop}
		Suppose that $(m_\alpha)_{b_\alpha}^0$, the $b_\alpha$ deformed curvature term of $\CF(L_\alpha)$, has valuation greater than $\lambda$. Let $b= (f_{\alpha,\beta})_*(b_\alpha)$.  Then $(m^0_\beta)_{b}$ has valuation greater than $\lambda$.
	\end{prop}
	In the simplest example, we define $b_\alpha=(f_{\alpha,0})_*(0)\in CF^\bullet(L_0)$ to be the pushforward of the trivial deformation of $L_\alpha$. This deformation may be rewritten using the quadratic $A_\infty$ relations as
	\begin{align*}
		(m_0)_{b_\alpha}^0=\sum_{k}m^k_0(((f_{\alpha, 0})_*(0))^{\tensor k}= \sum_k m^k_0((f^0_{\alpha,0})^{\tensor k}) = f^1_{\alpha, 0} m^0_\alpha.
	\end{align*}
	The condition that our Lagrangians successively only bound disks of increasing energy means that the $m^0_\alpha$ have increasing valuation, so the sequence of cochains $(m_0)^0_{b_\alpha}=(f^0_{\alpha,0})_*(0)$ unobstruct $CF(L_0)$ to higher and higher valuations.
	We now show that this sequence $\{f^0_{\alpha,0}\}$ of deforming cochains converge to an actual bounding cochain.

	From the quadratic relation for composition of filtered $A_\infty$ homomorphisms $f_{\alpha+1, 0}= f_{\alpha, 0}\circ f_{\alpha + 1, \alpha}$ we obtain
	\begin{align*}
		f^0_{\alpha+1, 0}= \sum_{k\geq 0} f^k_{\alpha, 0}((f_{\alpha+1, \alpha}^0)^{\tensor k})
		= & f^0_{\alpha, 0}+\sum_{k\geq 1} f^k_{\alpha, 0}((f_{\alpha+1, \alpha}^0)^{\tensor k}).
	\end{align*}
	To prove the convergence, it suffices to show that the differences 
	\[
		f^0_{\alpha+1,0}-f^0_{\alpha, 0}= \sum_{k\geq 1} f^k_{\alpha, 0}((f_{\alpha+1, \alpha}^0)^{\tensor k})
	\]
	converge to zero (as we are proving convergence in an ultrametric space). 
	As the valuation of the $f_{\alpha, \alpha+1}^0$ goes off to infinity,
    the sequence of cochains $(f_{\alpha, 0})_*(0)$ converge in $\CF(L_0)$. 
\end{proof}

\subsection{Unobstructedness: returning to the proof}

We now compare the surgery profile defined in \prettyref{prop:surgeryprofile}, and the standard transversal surgery. 
\begin{prop}
	Suppose that $U= L_0\cap L_1$ is a compact convex region, satisfying the conditions for taking the generalized Lagrangian surgery as in \prettyref{prop:surgeryprofile}.
	Then there exists another Lagrangian $L_1^1$ which  intersects $L_0$ transversely at a unique point $q$.
	Furthermore, $L_0\#_q L_1^1$ is Lagrangian isotopic to $L_0\#_U L_1$. 
	\label{prop:surgerycomparison}
\end{prop}
\begin{proof}
We first describe a family of Lagrangians $L^\alpha_1$. Let $U_1\supset U_0\supset U$ be small collared neighborhoods of $U$, and let $h^\alpha:[0, 1]_\alpha\times U_1\to \RR$ be a family of smooth functions satisfying the following:
\begin{itemize}
	\item 
	As a section, $dh^0= L_1$ on all of $U_1$. 
	\item
	As a section, $dh^\alpha=L_1$ on $U_1\setminus U_0$ for all $\alpha$.
	\item 
	$dh^1(x)= (\text{dist}_q(x))^2$ in a small neighborhood of $q$. 
	\item
	$h^\alpha$ is convex for every $\alpha$. 
\end{itemize}	
Let $L_1^\alpha$ be the Lagrangian obtained by removing the portion of $L_1$ which lives above $U_0$, and gluing in $dh^\alpha$ instead.
Clearly $L_1^1$ and $L_1$ are Lagrangian isotopic.
By construction $L_1^1$ and $L_0$ intersect transversely.
For each $\alpha$, the Lagrangians $L_0$ and $L_1^\alpha$ have convex intersection region $U_\alpha$.
We may construct the surgeries $L_0\#_{U^\alpha} L_1^\alpha$ in a smooth family.
Since (with appropriate choices of surgery neck) $L_0\#_{U^1} L_1^1=L_0\#_q L_1^1$, we may conclude that there is Lagrangian isotopy between our generalized Lagrangian surgery and the standard transverse surgery. 
\end{proof}

We will use this comparison for our tropical Lagrangians. 
We define the set
\[C_{1}:=W^{-1}_{t, 1}(\{z \text{ such that } |z|\leq 1\}).\]
By \prettyref{lemma:equivadmissibility} the monomially admissible Lagrangian $L(\phi)$ is bottlenecked by the symplectic fibration $W_{t, 1}$.
See \prettyref{fig:bottleneck3}.
The negative portion of this Lagrangian is  $L^-(\phi)= L(\phi)\cap C_1$. 
The positive portion of the bottleneck in $L^+(\phi)=L(\phi)\setminus C_1$.
Topologically, $L^+(\phi)$ is $(W^{-1}_{t, 1}(1)\cap L(\phi))\times \RR\geq 1$.
\begin{prop}
	$\CF(L(\phi))$ and  $\CF(L^-(\phi))$ are isomorphic as curved $A_\infty$ algebras.
	\label{prop:trivialbottleneck}
\end{prop}
This follows from observing that the gradient flow of a monomially admissible Morse function at the boundary $\partial L^-(\phi)$ points outward, and that one can pick Morse function for $L(\phi)$ which only has critical points in the overlapping region with $L^-(\phi)$ and applying \prettyref{prop:bottlenecktrivialneck}.

\begin{prop}
	Suppose that $\phi$ is a smooth tropical polynomial. Let $W_{t, 1}$ be a tropically localized superpotential so that $L(\phi)|_{P_N}$ is $W_{t, 1}$ admissible.
	There exists a monomial admissible Hamiltonian wrapping isotopy (see \prettyref{def:admissiblitycondition}) $\theta$ so that $\sigma_0$ and $\theta(\sigma_{-\phi})$ have transverse intersections $q_v$ for each $v\in \Delta_\phi\cap \ZZ^n$.  
	Furthermore there exists a Lagrangian $L^{tr}(\phi)$ satisfying the following properties:
	\begin{itemize}
		\item
		$L^{tr}(\phi)$ is admissibly Lagrangian isotopic to $\sigma_0\#_{\{q_v\}_{v\in \Delta_\phi^\ZZ}} (\theta(\sigma_{-\phi}))$.
		Note here that we are only performing Lagrangian surgery at the transverse intersections corresponding to non-self intersection points of $\phi$.
		\item 
		$L^{tr}(\phi)$ agrees with $L(\phi)$ on the set $C_{ 1}$. 
	\end{itemize}
	\label{prop:ltr}
\end{prop}
\begin{proof}
	First observe that $L^{tr}(\phi)$ and $L(\phi)$ \emph{do not} have the same topology! 
	Consider $\theta_{C_1}$, the monomial admissible Hamiltonian wrapping isotopy which has been modified to be the identity on $C_1$. 
	Outside of $C_1$ the monomial admissible Hamiltonian wrapping isotopy increases the controlled arguments over each monomial admissibility region, so that $\sigma_0$ and $\theta_{C_1}(\sigma(\phi))$ are disjoint outside a sufficiently large compact region.
	As a result, $\sigma_0$ and $\theta_{C_1}(\sigma_{-\phi})$ now intersect at compact regions $\sigma_0\cap \theta_{C_1}(\sigma_{-\phi})=\bigsqcup_{v\in \Delta^\ZZ_\phi} \tilde U_v$, which agree with our original intersections over the compact set $C_1$,
	\[\tilde U_v\cap C_1=U_v.\]
	Consider the Lagrangian 
	\[\sigma_0\#_{\tilde U_v}\theta_{C_1}(\sigma_{-\phi})\]
	which agrees with $L(\phi)$ over the set $C_1$. 
	We then define $L^{tr}(\phi)$ using \prettyref{prop:surgerycomparison}.
\end{proof}
The Lagrangians $L^{tr}(\phi)$ and $L(\phi)$ are compared in 
\prettyref{fig:surgerytruncation}.
\begin{figure}
	\centering
	\begin{subfigure}{.4 \linewidth}
		\centering
		\includegraphics[scale=.6]{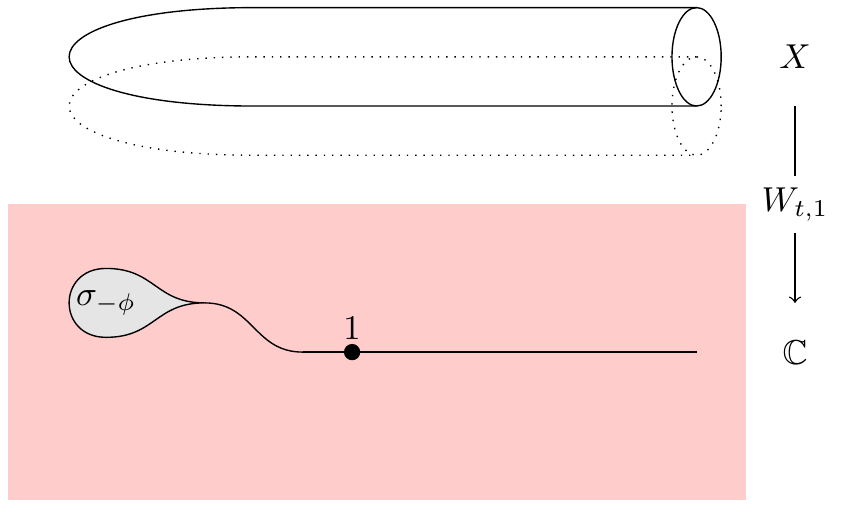}
		\caption{}
		\label{fig:tropicaladmissibility}
	\end{subfigure}\;\;\;\;\;
	\begin{subfigure}{.4 \linewidth}
		\centering
		\includegraphics[scale=.6]{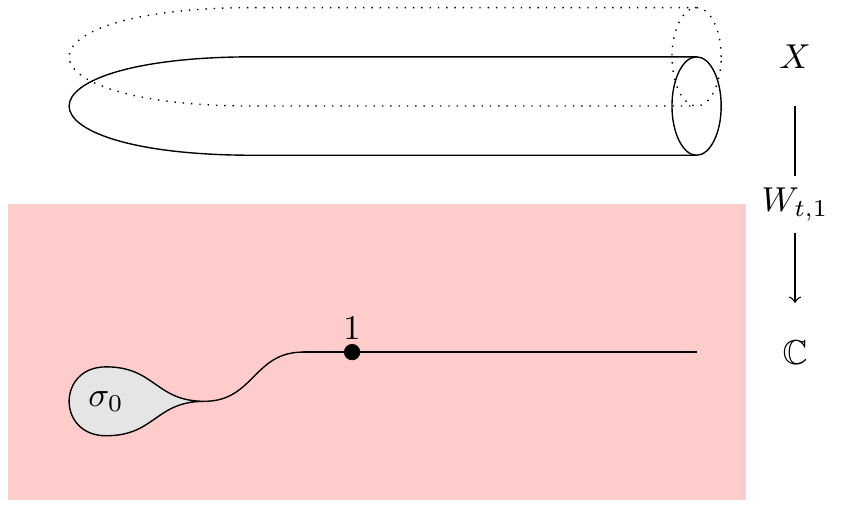}
		\caption{}
	\end{subfigure}\\
	\begin{subfigure}{.4 \linewidth}
		\centering
		\includegraphics[scale=.6]{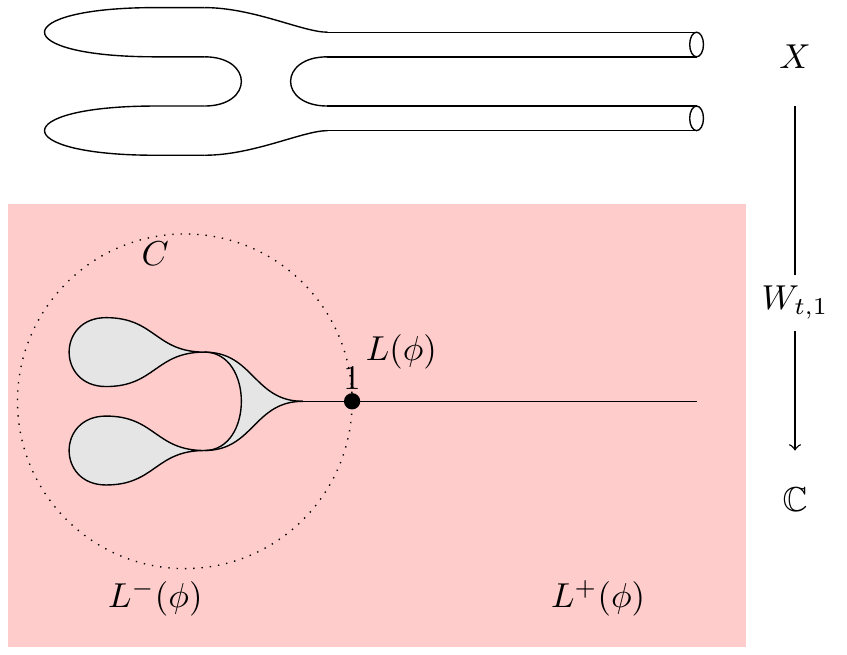}
		\caption{}
		\label{fig:bottleneck3}
	\end{subfigure}\;\;\;\;\;
	\begin{subfigure}{.4 \linewidth}
		\centering
		\includegraphics[scale=.6]{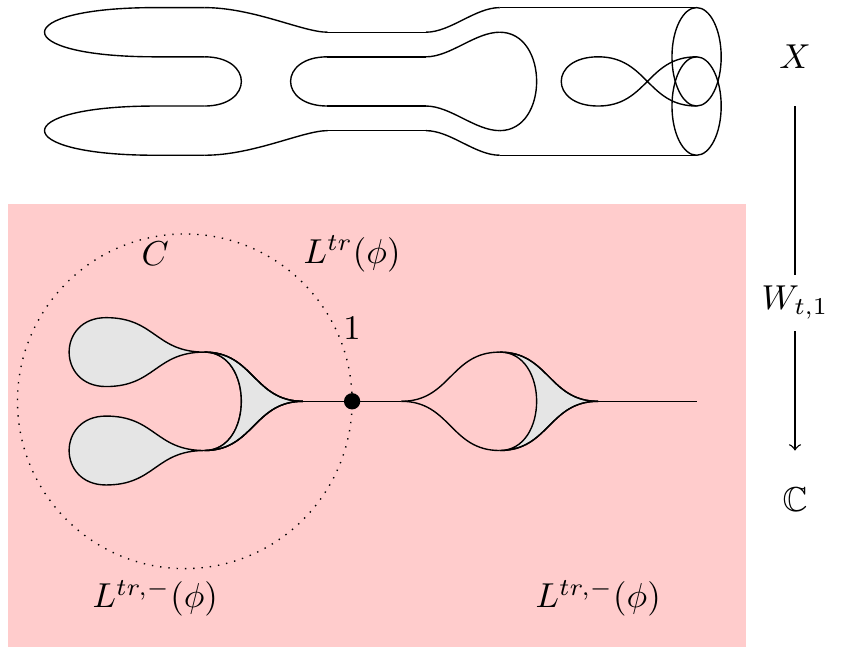}
		\caption{}
	\end{subfigure}
	\caption{The image of some of our Lagrangian submanifolds under the projection $W_{t, 1}$.
	Here, the sections are assumed to intersect at 2 points. 
	The tropical sections $\sigma_{-\phi}$ and $\sigma_0$ fiber over the real axis past the bottleneck. 
	The second pair of pictures compare $L^{tr}(\phi)$ and $L(\phi)$ under the projection of $W_{t, 1}$. 
	The surgery regions to construct $L(\phi)$ lie over real axis (as this is where $\sigma_{-\phi}$ and $\sigma_{0}$ intersect).
	$L(\phi)$ is cylindrical past the bottleneck.
	If the Lagrangians $\sigma_{-\phi}$ and $\sigma_0$ are made transverse before performing the surgery, the resulting Lagrangian has the topology drawn in last figure.
	Over the bottleneck region, this transverse surgery can be made to match the one defining the tropical Lagrangian.
	However, past the bottleneck region, this matches $\sigma_0$ and $\sigma_{-\phi}$. }
	\label{fig:surgerytruncation}
\end{figure}

Since $L^{tr}(\phi)$ matches $L(\phi)$ on the compact set $C_1$,  $L^{tr}(\phi)$ is similarly bottlenecked by this symplectic fibration at the point $z_{0}=1$.
Let $L^{tr, -}(\phi)$ and $L^{tr,+}(\phi)$ be the negative and positive ends of the bottleneck.
Since $L^{tr,-}\subset C_1$, we obtain that $L^{tr,-} = L^-(\phi)$. 
\begin{lemma}
	If $\CF(L^{tr}(\phi))$ has a bounding cochain, then so does $\CF(L(\phi))$. 
\end{lemma}
This is virtue of the curved $A_\infty$ homomorphisms
\[
\CF(L^{tr}(\phi))\xrightarrow{\text{Assum \ref{assum:bottleneckideals}}} \CF(L^{tr,-}(\phi))\xrightarrow{\text{Prop. \ref{prop:agreebottleneck}}}\CF(L^-(\phi))\xrightarrow{\text{Prop. \ref{prop:trivialbottleneck}}} \CF(L(\phi)).
\]
It remains to prove that $L^{tr}(\phi)$ is unobstructed by bounding cochain. 
We do this by constructing an eventually unobstructed sequence starting at $L^{tr}(\phi)$. 
We now describe a sequence of Hamiltonian isotopic Lagrangian submanifolds $\{L^{tr}_\alpha\}_{\alpha\in \NN}$, with $L^{tr}_0=L^{tr}(\phi)$.

For notation, we denote the union of two tropical sections which have been made transverse by an infinitesimal wrapping Hamiltonian as  $L^{tr}_\infty: = \sigma_0 \cup  (\theta'(\sigma_{-\phi}))$. 
For each $v\in\Delta^\ZZ_\phi$, let $q_v\in L^{tr}_\infty$ be the corresponding self-intersection point. 
Around each $q_v$ there is a standard symplectic neighborhood  $B_\epsilon(q_v)$, which we identify with a neighborhood of the origin in $\CC^n$. 
We take a Hamiltonian isotopy of $L^{tr}_\infty$ so that its restriction to each $B_\epsilon(q_v)$ matches $\RR^n\cup i\RR^n$. 
The sequence of Hamiltonian isotopic Lagrangian submanifolds $L^{tr}_\alpha$ are constructed by replacing 
\[
	L^{tr}_\infty\cap \left( \bigcup_{v\in \Delta^\ZZ_\phi} B_\epsilon(q_v)\right)
\]
with a standard surgery neck of radius $r_\alpha$. The constants $r_\alpha$ are chosen so that $\lim_{\alpha\to\infty} r_\alpha=0$.
In order to make this a  sequence of Hamiltonian isotopic Lagrangian submanifolds, we cancel out the small amount of Lagrangian flux swept out by the surgery necks with an equal amount of Lagrangian isotopy on $L^{tr}_\infty\setminus\left( \bigcup_{v\in \Delta^\ZZ_\phi} B_\epsilon(q_v)\right)$.
These Hamiltonian isotopies are chosen so that $L^{tr}_\alpha \setminus\left( \bigcup_{v\in \Delta^\ZZ_\phi} B_\epsilon(q_v)\right)$ converges smoothly. 

By \prettyref{prop:surgerycomparison}, the first member of this sequence  $L_0^{tr}$ can be constructed in such a way that it is Hamiltonian isotopic to $L^{tr}(\phi)$.
\cite{fukaya2007chapter10} gives us a relation between disks on the $L^{tr}(\phi)$ and the disks on $\sigma_0 \cup  (\theta'(\sigma_{-\phi}))$.
\begin{prop}
	If there exists a sequence of holomorphic disks 
	\[u_\alpha: (D, \partial D)\to (X, L_\alpha^{tr})\]
	 contributing to the $A_\infty$ structure on $\CF(L^{tr}_\alpha)$, then there exists a holomorphic polygon or disk 
	 \label{prop:sequenceonsurgery}
	 \[u_\infty: (D, \partial D)\to (X, L^{tr}_\infty).\] 
\end{prop}
\emph{Idea of Proof}
	The proof follows the methods used in \cite[Section 6]{palmer2019invariance}, \cite[Theorem 1.2]{rizell2018refined}, or \cite{fukaya2007chapter10}, which all make comparisons between disks with boundary on the surgery to polygons with boundary on transversely intersecting Lagrangians.
	Let $\{u_\alpha\}$ be a sequence of holomorphic disks of bounded energy and boundary on $L^{tr}_\alpha$ contributing to $\CF(L^{tr}_\alpha)$. 
	Then the images of the $\{u_\alpha\}$ are mutually contained within a compact set of $X$.
	We would like to apply a Gromov-compactness argument on the sequence of $u_\alpha$  but cannot as the family $L^{tr}_\alpha$ does not converge to $L^{tr}_\infty$ in a strong enough sense.
	However, it is the case that $L_\alpha\setminus B_\epsilon(q_v)$ does converge to $L_\infty \setminus B_\epsilon(q_v)$ uniformly.
		
	In \cite[Section 62]{fukaya2007chapter10} it is shown that for such a sequence of disks $u_\alpha: (D^2, \partial D)\to (X, L^{tr}_\alpha)$, one may construct a family of approximate solutions $u_{\alpha, app}:(D^2, \partial D)\to (X, L^{tr}_\infty)$ by replacing the regions of the curve $u_\alpha$ which intersect $B_\epsilon(q_v)$ with holomorphic corners based on a standard model from \cite[Section 59]{fukaya2007chapter10}. 

	The following neck stretching argument is used to show that these approximate solutions approach an honest solution. 
	As $\alpha\to\infty$, the restriction  $L_\alpha \cap (B_\epsilon(q_v)\setminus \{q_v\})$ approaches the cylindrical Lagrangian  $L^{tr}_\infty \cap (B_\epsilon(q_v)\setminus \{q_v\})$. As a result, the holomorphic maps $\{u_\alpha\}$ converge to cylindrical maps in the neck region $B_\epsilon(q_v)\setminus \{q_v\}$ \cite[Section 62.4]{fukaya2007chapter10}.
	This provides an error bound on the failure of $u_{\alpha, app}$ to being a holomorphic polygon.
	As the $\{u_\alpha\}$ converge to cylindrical maps this error approaches zero.

	Since the $\{u_\alpha\}$ have images confined  a compact set of $X$,the maps $\{u_{\alpha, app}\}$ are similarly constrained.
	We can apply Arzela-Ascoli to take a subsequence of $\{u_{\alpha, app}\}$ which converge to a holomorphic map $u_\infty$ with boundary on $L^{tr}_\infty$.
\endproof

In this case, we can rule out the existence of holomorphic polygons with boundary on $L^{tr}_\infty$. 
\begin{prop}
	If we are working in complex dimension greater than $1$, there are no holomorphic polygons with boundary on $L_\infty^{tr}=\sigma_0 \cup  (\theta'(\sigma_{-\phi}))$.
	\label{prop:nodiskindex}
\end{prop}
\begin{proof}
	This follows from an index computation.
	A holomorphic polygon with boundary contained in $\sigma_0\cup (\theta'(\sigma_{-\phi}))$ has $2k-1$ inputs and $1$ output.
	The inputs must alternate between being an element of $\CF(\sigma_0, \theta'(\sigma_{-\phi}))$ and $\CF( \theta'(\sigma_{-\phi}), \sigma_0)$. 
	We will look at the case where output $p$ lies in $p\in \CF(\sigma_0, \theta'(\sigma_{-\phi}))$ and the inputs $x_i, y_j$ are in
	\begin{align*}
		y_j\in \CF( \theta'(\sigma_{-\phi}), \sigma_0) & & 1 \leq i \leq k-1\\
		x_i\in \CF(\sigma_0, \theta'(\sigma_{-\phi})) & & 1\leq i \leq k.
	\end{align*}
	The dimension of moduli space of regular polygons with these boundary conditions can be explicitly computed based on the index of the points $x_i$ and $y_j$. 
	The degree of the input intersections of the form $x_i$ is $n$, and the degree of each intersection of the form $y_j$ is $0$.
	The output intersection $p$ has degree $n$.
	The dimension of this space of disks is 
	\begin{align*}(2k-1)- 2+\deg (p) - &\left( \sum_{i=1}^{k} \deg (x_i) + \sum_{i=1}^{k-1} \deg(y_i)\right)\\
		=&(2-n)k-3+n
	\end{align*}
	which is negative whenever $n\geq 2$. 

	The argument for when the output marked point $p$ is in $\CF(\theta'(\sigma_{-\phi}), \sigma_0)$ is the same.
\end{proof}
By  \prettyref{prop:sequenceonsurgery} and \prettyref{prop:nodiskindex}, the sequence of Lagrangians submanifolds $L^{tr}_\alpha$ is eventually unobstructed. 

We additionally need to prove that the Lagrangians $K^{tr}_{\alpha, \alpha+1}\subset X\times \CC$ given by the suspension of the Hamiltonian isotopy between $L_\alpha^{tr}$ and $L^{tr}_{\alpha+1}$ are an eventually unobstructed sequence. 
This follows from the same argument. 
A sequence of holomorphic disks with boundary on $K^{tr}_{\alpha, \alpha+1}$ produces a holomorphic disk with boundary on $K_\infty = L^{tr}_\infty\times \RR$. 
Since $L^{tr}_\infty\times \RR$ is a trivial cobordism and the complex structure was chosen to be the standard product structure, every holomorphic disk with boundary on $L^{tr}_\infty\times \RR$ gives us a holomorphic disk with boundary on $L^{tr}_\infty$. 
By \prettyref{prop:nodiskindex}, there are no such disks. Therefore, the Lagrangian cobordisms $K_{\alpha, \alpha+1}$ are eventually unobstructed.

As both $\{L^{tr}_{\alpha}\}_{\alpha\in \NN}$ and $\{K^{tr}_{\alpha,\alpha+1}\}_{\alpha\in \NN}$ are eventually unobstructed sequences of Lagrangians, it follows from \prettyref{lemma:unobstructedinlimit} that  $L^{tr}_0=L^{tr}(\phi)$ is unobstructed, completing the proof of \prettyref{prop:unobstructed}. 
\begin{remark}
	Note that in dimension 1, the Lagrangian sections $\theta'(\sigma_{-\phi})\cup \sigma_0$ may still bound interesting holomorphic disks.
	In dimension 1, see \prettyref{fig:diskwhichsurvives} for an example of a disk which has boundary on tropical sections.
	We now provide some evidence that these disks correspond to higher genus open Gromov-Witten invariants of the tropical Lagrangian. 
	In the 1 dimensional example, the disk in \prettyref{fig:diskwhichsurvives} becomes a holomorphic annulus with boundary on $L(x_1^2)$.

	We can replicate this phenomenon in higher dimensions.
	Let $\phi_{E}(x_1, x_2)$ be the tropical polynomial describing a tropical elliptic curve $V(\phi_E)$ as drawn in \prettyref{fig:internalannuli}. Then the Lagrangian $L(\phi)$ bounds holomorphic annuli which are modelled on the previous example in one dimension higher.
	See \prettyref{fig:internalannuli} for this example (in red) and an example of holomorphic genus 0 curve with four boundaries on $L(\phi_{E})$.
	At this point, it is unclear what the presence of these higher genus open Gromov-Witten invariants entail.
\end{remark}
\begin{figure}
	\centering
	\begin{subfigure}{.4\linewidth}
		\centering
		\includegraphics[scale=.7]{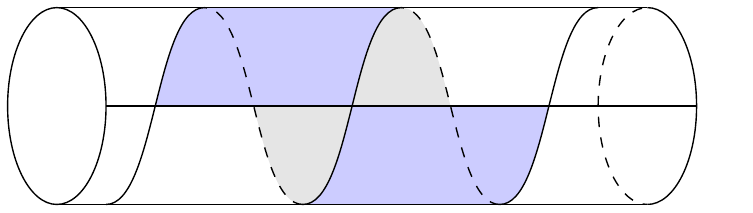}
		\caption
		{
			An example of a disk with boundary on $\sigma_0$ and $\sigma_{x^2}$ in $\CC^*$, corresponding to a holomorphic annulus after applying surgery at 3 points. 
		}
		\label{fig:diskwhichsurvives}
	\end{subfigure}
	\begin{subfigure}{.4\linewidth}
		\centering
		\includegraphics[scale=1]{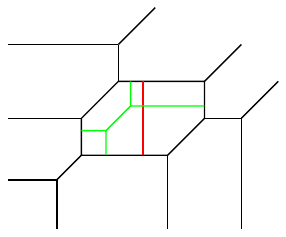}
		\caption
		{
			Two examples of higher genus open curves with boundary on $L(\phi_{E})$.
		}
		\label{fig:internalannuli}
	\end{subfigure}
	\caption{Some speculation on higher genus OGW invariants}
\end{figure}

\section{Homological Mirror Symmetry for $L(\phi)$}
\label{sec:tropicalmirrorsymmetry}
In this last section we look at some applications of our construction to homological mirror symmetry. Our Lagrangian submanifolds will have the additional structure of a Lagrangian brane, meaning that they are equipped with a choice of Morse function, spin structure, and bounding cochain.
To simplify notation, we will often refer to data of a Lagrangian brane by the Lagrangian submanifold $L$.
Since the Lagrangians $L(\phi)$ that we study are not exact, they do not fit into the framework of \cite{hanlon2018monodromy}, and additionally we are required to work with the Fukaya category defined over Novikov coefficients.
\begin{assumption}
	We assume that the monomially admissible Fukaya Seidel category can be extended to include unobstructed Lagrangian submanifolds, and that the appropriate analogues of \prettyref{thm:cobordisms} and \prettyref{thm:mirrorfortoric} hold in this setting.
\end{assumption}
The mirror to the Landau-Ginzburg model $(X, W_\Sigma)$ is the rigid analytic space $\check X^\Lambda_\Sigma$.
The intuition for our constructions should be understood independently of the requirements of Novikov coefficients.

The Novikov toric variety $\check X^\Lambda_\Sigma$  comes with a valuation map $\val: \check X^\Lambda_\Sigma\to Q$ using the valuation on the Novikov ring.
A difference between complex geometry and geometry over the Novikov ring is that in the non-Archimedean setting the valuation of a divisor is described exactly by its tropicalization, as opposed to living in the amoeba of the tropicalization.
\subsection{Mirror Symmetry for Tropical Lagrangian hypersurfaces}
\begin{theorem}
	Let $D_{\check f}\subset \check X_\Sigma^\Lambda$ be a divisor transverse to the toric divisors, defined by the equation $\check f=0$. Let $\phi$ be the tropicalization of $\check f$. The tropical Lagrangian brane $L(\phi)$ is mirror to $\mathcal O_{D'}$, with $D'$ rationally equivalent to $D_{\check f}$.
	\label{thm:tropicalLagrangianhms}
\end{theorem}

We first give a family Floer argument motivating this mirror statement.
From SYZ mirror symmetry we know that the mirror to a point in the complement of the anticanonical divisor $z\in \check X_\Sigma^\Lambda\setminus D$ is a fiber of the SYZ fibration equipped with local system.
One method to compute the mirror sheaf to $L(\phi)$ is to compute $\CF(L(\phi), F_q)$ and assemble this  data into a sheaf over $X$ using techniques from family Floer theory.
This line of proof is rooted in a long-known geometric intuition for mirror symmetry via tropical degeneration (see \prettyref{fig:amodelbmodel}).
However, the precise computation of the support is difficult due to the need to count holomorphic strips contributing to the Floer differential. In \cite{hicks2019tropical}, we compute the support in a few fundamental examples.
\begin{figure}
	\centering
	\includegraphics{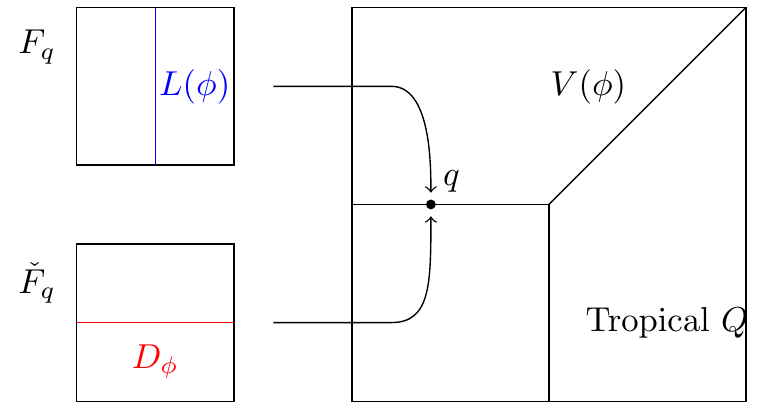}
	\caption[Swapping complex and symplectic geometry]{SYZ mirror symmetry predicts that Lagrangians are swapped with complex subvarieties by fiberwise duality over a tropical curve in the base.}
	\label{fig:amodelbmodel}
\end{figure}
\begin{proof}[of \prettyref{thm:tropicalLagrangianhms}]
	We use Lagrangian cobordisms to prove this theorem. The function $\check f$ associated to the effective divisor $D_{\check f}\subset \check X^\Lambda_\Sigma$ defines a section of the line bundle $\mathcal O_{\check X_\Sigma}(D_{\check f})$, giving us an exact triangle
	\begin{equation}
		\mathcal O_{\check X^\Lambda_\Sigma}(-D_{\check f})\xrightarrow{\check f} \mathcal O_{\check X^\Lambda_\Sigma}\to \mathcal O_{D_{\check f}}.
		\label{eq:bexacttriangle}
	\end{equation}
	This gives us a description of $\mathcal O_{D_{\check f}}$ in terms of line bundles on $\check X^\Lambda_\Sigma$.
	By \prettyref{thm:mirrorfortoric}, we have an identification of $\Fuk((\CC^*)^n, W_\Sigma)$ with $D^b\Coh(\check X_\Sigma^\Lambda)$ giving us the following mirror correspondences between sheaves and Lagrangian submanifolds:
	\begin{align*}
		\mathcal O_{\check X_\Sigma^\Lambda} \leftrightarrow \sigma_0 &  & \mathcal O_{\check X_\Sigma^\Lambda}(-D_{\check f})\leftrightarrow \sigma_{-\phi_{[D]}}.
	\end{align*}
	where $\phi_D$ is the support function of the divisor $D_{\check f}$. The Lagrangians $\sigma_{-\phi}$ and $\sigma_{-\phi_{[D]}}$ are Hamiltonian isotopic. 
	Using an extension of \prettyref{thm:cobordisms} to the unobstructed setting, we obtain an exact triangle
	\begin{equation}
		\sigma_{-\phi}\xrightarrow{\check g}\sigma_{0}\to L(\phi)
		\label{eq:mappingcone}
	\end{equation}
	for some map $\check g$.  $L(\phi)$ is therefore identified under the mirror functor to a sheaf $\mathcal O_{D_{\check g}}$ supported on $D_{\check g}$, an effective divisor for the bundle $\mathcal O(D_{\check f})$.
	The divisors $D_{\check g}$ and $D_{\check f}$ are rationally equivalent.
\end{proof}

If we wish to prove   that $D_{\check g}$ and $D_{\check f}$ match up exactly, we need to better understand the map ${\check g}$ in  equation (\ref{eq:mappingcone}).
Though we cannot determine this map without making a computation of holomorphic strips with boundary on the cobordism $K$, we conjecture
\begin{conjecture}
	Let $\phi$ be the tropicalization of $\check f$. There exists a choice of bounding cochain making the Lagrangian brane $L(\phi)$  mirror to $\mathcal O_{D_{\check f}}$.
	\label{conj:mirror}
\end{conjecture}

\subsection{Twisting by Line Bundles}
Let $\psi:Q\to \RR$ be a piecewise linear function, and let $\theta_\psi$ be the time 1 Hamiltonian flow associated to the pullback of the smoothing, $\tilde \psi \circ \val: X\to \RR$.
This Hamiltonian flow can be compared to fiberwise sum \cite{subotic2010monoidal} with the section $\sigma_\psi$
\[
	\theta_\psi(L)= L+\sigma_\psi.
\]
While wrapping is not an admissible Hamiltonian isotopy, it still sends admissible Lagrangian branes to admissible Lagrangian branes, giving an automorphism of the Fukaya category.
\begin{theorem}[\cite{hanlon2018monodromy}]
	Let $\psi:Q\to \RR$ be the support function for a line bundle $\mathcal L_\psi$.
	The functor on the Fukaya category $L\mapsto L+\sigma_\psi$ is mirror to the functor $\mathcal F_L\mapsto \mathcal F_L\tensor \mathcal L_{\psi}$.
	\label{thm:tensorfiber}
\end{theorem}
Provided that $V(\phi)$ and $V(\psi)$ have a pair of pants decompositions with no codimension 2 strata intersecting,  the Lagrangian $L(\phi)+\sigma_\psi$ can be given an explicit description in terms of the pair of pants decomposition.
We outline this construction in complex dimension 2, but the higher dimensional constructions are analogous.
When $V(\phi)$ and $V(\psi)$ have locally planar intersection, the support of  $V(\psi)$ is contained in the cylindrical region between each of the pants in the decomposition of $V(\phi)$.
This means that if the smoothing and construction parameters for the tropical Lagrangians are chosen small enough, the strata
\begin{align*}
	U_{\{v_i, v_j\}}^\psi \cap U_{\{w_i, w_j, w_k\}}^\phi= & \emptyset \\
	U_{\{v_i, v_j\}}^\phi \cap U_{\{w_i, w_j, w_k\}}^\psi= & \emptyset
\end{align*}
are disjoint from one another.
Therefore, the Lagrangian $L(\phi)$ matches $L(\phi)+\sigma_\psi$ over the charts near the vertices of the tropical curve, $U_{\{v_i, v_j, v_k\}}^{\phi}$.

To construct $L(\phi)+\sigma_\psi$ from this pair of pants decomposition, it suffices to modify the cylinders living over regions $U_{\{v_i, v_j\}}^\phi$.
We construct this  modification in a local model where $\phi=0\oplus x_1$ and $\psi= 0 \oplus x_2$.
Topologically  $L(\phi)+ \sigma_\psi|_{U_{\{v_i, v_j\}}^\phi}$ is a cylinder, with an additional twist in the argument direction perpendicular to $V(\psi)$ at the point of intersection between the two tropical varieties, as drawn in \prettyref{fig:3pointtwist}, which shows $L(\phi_{E})+\sigma_{\phi_{pants}}$, where $V(\phi_E)$ is the tropical elliptic curve, and $V(\phi_{pants})$ is a tropical pair of pants meeting the tropical elliptic transversely at 3 points. 
This kind of modification to our tropical Lagrangian was remarked upon in \cite[Remark 5.2]{matessi2018Lagrangian} as a more general way to construct tropical Lagrangians.
This discussion shows that if $L(\phi)$ is mirror to $\mathcal O_D$, then the twisted Lagrangian $L(\phi)+\sigma_\psi$  is mirror to $ \mathcal O_D\tensor \mathcal L_\psi$. This can also be understood as the mirror to the pushforward of the pullback of $L_\psi$.

From the pair of pants description, it is clear that we can ``twist'' our Lagrangian in the argument along edges in ways that do not arise from adding on a section $\sigma_\psi$--- see for instance, \prettyref{fig:1pointtwist}.
These too should be mirror to pushforwards of line bundles on $D$, however these line bundles are not the pullbacks of line bundles on $X$.

This local twisting can be defined more rigorously by working with the sheaf of affine differentials. On $U\subset \RR^n$ these are the sections $\sigma: U\to \val^{-1}(U)$ which are locally described as the differential of tropical polynomials. However, such a section need not be defined globally as the differential of a tropical polynomial.

\begin{definition}
	Let $L$ be a Lagrangian, with $\val(L)\subset U$. Let $\sigma$ be a tropical section defined over the subset $U$. Define  \emph{the tropical Lagrangian  twisted by $\sigma$} to be the Lagrangian submanifold
	\[
		L(\phi;\sigma):=L(\phi)+\sigma.
	\]
\end{definition}

\begin{figure}
	\begin{subfigure}{.48\linewidth}
		\centering
		\includegraphics{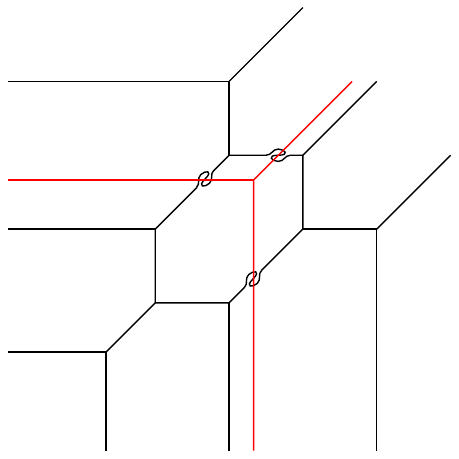}
		\caption{Twisting by a line bundle}
		\label{fig:3pointtwist}
	\end{subfigure}
	\begin{subfigure}{.48\linewidth}
		\centering
		\includegraphics{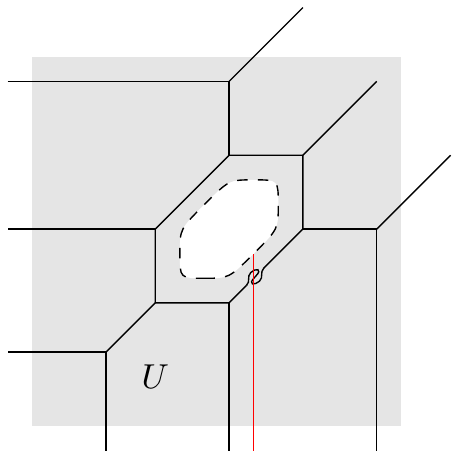}
		\caption{A twist which does not extend to a global section}
		\label{fig:1pointtwist}
	\end{subfigure}
	\caption{Inserting twists into tropical Lagrangians}
\end{figure}
As an example, in the mirror to $\CP^2$  we consider the open set $U$ as drawn in \prettyref{fig:1pointtwist}. There exists a tropical differential on $U$ whose critical locus intersects the tropical elliptic curve $V(\phi_E)$ at a single point.
The Lagrangian given by twisting along this tropical differential is expected to be mirror to the direct image of a degree 1 line bundle on $E$, an elliptic curve whose tropicalization is $V(\phi_E)$.
We expect that we can understand these twistings by employing tropical geometry on the affine structure of $\val(\phi_E)$ itself.
This tropical differential does not extend to a section over the entire base, as degree 1 line bundles on $E$ do not arise from pullback of a line bundle on $\CP^2$.

\begin{conjecture}
	The twisted tropical Lagrangians $L(\phi;\sigma)$ are mirror to the direct image of line bundles on the mirror divisor $D$.
	\label{conj:twisting}
\end{conjecture}

\appendix
	
\section{$A_\infty$ algebras and Bounding Cochains}
\label{app:ainftyrefresher}
In this appendix we review some statements on filtered $A_\infty$ algebras.
In \prettyref{subsec:ainfinitybasics}, we fix notation for filtered $A_\infty$ algebras.
\prettyref{subsec:ainftymorphism} looks at properties of filtered $A_\infty$ homomorphisms, and \prettyref{subsec:ainftydeformations} reviews the definition and construction of bounding cochains.
\subsection{Notation: Filtered $A_\infty$ algebras}
\label{subsec:ainfinitybasics}
We  review curved $A_\infty$ algebras.
In order to ensure convergence of the deformations we develop,  we work with \emph{filtered} $A_\infty$ algebras.
This will mean working over the Novikov field.
\begin{definition}[\cite{fukaya2010Lagrangian}]
	Let $R$ be a commutative ring with unit. The \emph{universal Novikov ring} over $R$ is the set of formal sums
	\[
		\Lambda_{\geq 0}:=\left\{\sum_{i=0}^\infty a_i T^{\lambda_i} \;|\;\lambda_i\in \RR_{\geq 0}, n_i\in \ZZ, \lim_{i\to\infty} \lambda_i=\infty\right\}.
	\]

	Let $k$ be a field. The \emph{Novikov Field} is the set of formal sums
	\[
		\Lambda := \left\{\sum_{i=0}^\infty  a_i T^{\lambda_i}  | \lambda \in \RR, n_i\in \ZZ, \lim_{i\to\infty} \lambda_i = \infty\right \}.
	\]
	This is a non-Archimedean field with the same valuation.
	An \emph{energy filtration} on a graded $\Lambda$-module $A^\bullet$  is a filtration $F^{\lambda_i}A^k$ so that
	\begin{itemize}
		\item
			Each $A^k$ is complete with respect to the filtration, and has a basis with zero valuation over $\Lambda$.
		\item Multiplication by $T^\Lambda$ increases the filtration by $\lambda$.
	\end{itemize}
\end{definition}
The energy filtration will play an important role in the algebraic setting where many of our constructions will either induct on the energy filtration, and in order to obtain some kind of convergence we will have to use the energy filtration.
\begin{definition}
	Let $A^\bullet$ have an energy filtration. A filtered $A_\infty$ structure  $(A, m^k)$ is an enhancement of $A^\bullet$ with $\Lambda_{\geq 0}$ multilinear graded higher products for each $k \geq 0$
	\[
	m^k:A^{\tensor k}\to A[2-k]
	\]
	satisfying the following properties:
	\begin{itemize}
		\item
		      \emph{Energy:} The product respects the energy filtration in the sense that :
		      \[m^k(F^{\lambda_1}A, \cdots , F^{\lambda_k}A)\subset F^{\sum_{i=1}^k \lambda_i}A.\]
		\item
		      \emph{Non-Zero Energy Curvature:}  The obstructing curvature term has positive energy, $m^0\in F^{\lambda>0}C$.
		\item
		      \emph{Quadratic $A_\infty$ relations:} For each $k \geq 0$ we require the following relation hold
		      \[
		      \sum_{j_1+i+j_2=k} (-1)^\clubsuit m^{j_1+j_2+1}(\id^{\tensor j_1}\tensor m^{i}\tensor \id^{\tensor j_2})=0.
		      \]
		      The value of $\clubsuit$ is determined
		      on an input element $a_1\tensor \cdots \tensor a_k$ by 
		      \[
		      \clubsuit = |a_{k-j_1}|+\cdots + |a_k|-i.
		      \]
	\end{itemize}
	We say that $(A, m^k)$ is uncurved or tautologically unobstructed if $m^0=0$. 
	\label{def:filteredainfty}
\end{definition}
For the purposes of exposition, we will work up to signs from here on out. 

\begin{definition}
	Let $A$ be a filtered $A_\infty$ algebra. An \emph{ideal} of $A$ is a subspace $I\subset A$ so that for every $b\in I$ and $a_1, \ldots, a_{k-1}\in A$, 
	\[
		m^k(a_1\tensor\cdots \tensor  a_j\tensor b \tensor a_{j+1}\tensor \cdots \tensor a_{k-1})\in I.
	\]
	Note that we \emph{do not} require $m^0\in I$. 
	\label{def:ideal}
\end{definition}

The quotient of an $A_\infty$ algebra by an ideal is again a filtered $A_\infty$ algebra.
Given $(A,m^k)$ a filtered $A_\infty$ algebra, define the positive filtration ideal 
\[
	A_{>0}:=\{a\in A\;:\; \val(a)>0\}.
\]
We may recover an uncurved $A_\infty$ algebra by taking the quotient, 
\[
	A_{=0}:= A/ A_{>0}.
\]
This is always uncurved as the $m^0$ term is required to  have positive valuation. 

\subsection{Filtered $A_\infty$ algebra homomorphisms}
\label{subsec:ainftymorphism}
The definition of homomorphisms between filtered $A_\infty$ algebras is similar to the definition of homomorphisms of differential graded algebras, except that the homomorphism relation is relaxed by homotopies.
\begin{definition}
	Let $(A, m_A^k)$ and $(B, m_B^k)$ be $A_\infty$ algebras. A \emph{filtered $A_\infty$ homomorphism} from $A$ to $B$ is a collection of graded maps
	\[
		f^k:A^{\tensor k}\to B
	\]
	for $k\geq 0$ satisfying the following conditions:
	\begin{itemize}
		\item
		      \emph{Filtered:} The maps preserve energy
		      \[
			      f^k(F^{\lambda_1}A, \cdots , F^{\lambda_k}A)\subset F^{\sum_{i=1}^k \lambda_i}B.
		      \]
		      
		\item
			\emph{Quadratic $A_\infty$ relations:} The $f^k, m^k_A$ and $m^k_B$ mutually satisfy the quadratic filtered $A_\infty$ homomorphism relations for $k\geq 0$
		      \[
			      \sum_{\substack{(j_1+j+j_2=k)\\ j, j_1, j_2\geq 0}} \pm f^{j_1+1+j_2}(\id^{\tensor j_1}\tensor m^{j}_A\tensor \id^{\tensor j_2}) = \sum_{\substack{i_1+\cdots + i_l =k\\ i_j \geq 0}} \pm m^{l}_B(f^{i_1}\tensor \cdots \tensor f^{i_l})
		      \]
	\end{itemize}
\end{definition}
\begin{prop}
	Let $f^{ k}:A^{\tensor k}\to B$ and $g^{ k}: B^{\tensor k}\to C$ be two filtered $A_\infty$ homomorphisms.
	Then
	\begin{align*}
		(g\circ f)^k:=&\sum_{\substack{j_1+\cdots+j_l=k\\ j_i \geq 0}} g^l(f^{j_1}\tensor \cdots \tensor f^{j_l})
	\end{align*}
	is an $A_\infty$ homomorphism.
\end{prop}

\subsection{Deformations of $A_\infty$ algebras}
\label{subsec:ainftydeformations}
The presence of higher product structures gives us additional wiggle room to deform the product structures on a filtered $A_\infty$ algebra.
We will be mainly interested in the case when we can deform a given filtered $A_\infty$ algebra into an uncurved one to obtain a well defined cohomology theory.
\begin{notation}
	As a shorthand, we write
	\[
		(\id\oplus a)^{{{n+k}\choose n}_a}=\sum_{j_0+\cdots+j_k=n}  (a^{\otimes j_0}\otimes \id \tensor a^{\tensor j_1}\tensor \id \tensor \cdots \tensor a^{\otimes j_{k-1}}\tensor \id \tensor a^{\otimes j_k})
	\]
	for the sum over all monomials containing $n+k$ terms, $n$ of which are $a$ and $k$ for which are $\id$. 
\end{notation}

\begin{definition}
	Let $a\in A$ be an element of positive valuation.
	Define the $a$-deformed product $m^k_a: A^{\tensor k}\to A$ by the sum
	\[
		m^k_a:= \sum_n m^{k+n}\left((\id\oplus a)^{{n+k\choose n}_a}\right).
	\]
	We call this a \emph{graded deformation} if the element $a$ has homological degree 1.
\end{definition}
The convergence of this sum is guaranteed by the positive valuation of the deforming element.
\begin{prop}
	$(A, m^k_a)$ is again a filtered curved $A_\infty$ algebra.
\end{prop}
We are interested in the cases where $(A, m_a)$  gives us a well defined homology theory even though $A$ itself may be curved.
\begin{definition}
	We say that $a\in A$ is a \emph{bounding cochain} or \emph{Maurer-Cartan solution} if
	\[m^0_a=\sum_k m^k(a^{\tensor k})=0.\]
	If $A$ has a bounding cochain, we say that $A$ is \emph{unobstructed.}
\end{definition}
When $m^0=0$, when we say that $A$ is \emph{tautologically unobstructed} or \emph{uncurved}.
In the unobstructed setting, we have a well defined cohomology theory of $(A, m^k)$.
\begin{definition}
	Let $A$ be an $A_\infty$ algebra. The \emph{space of Maurer-Cartan elements} is defined as
	\[
		\mathcal MC(A) := \{a \in A \;:\; m_a^0 = 0 \}.
	\]
\end{definition}
	The Maurer-Cartan equation is non-linear.
	In the event that the Maurer-Cartan space contains a linear subspace, then $0$ is a Maurer-Cartan element and the algebra $A$ is uncurved.
\begin{lemma}
	Let $f: A\to B$ be a filtered $A_\infty$ homomorphism.
	Then there exists a \emph{pushforward}  map between the bounding cochains on $A$ and the bounding cochains of $B$ given by
	\begin{align*}
		f_*: \mathcal MC (A) \to & \mathcal MC (B)     \\
		b_A\mapsto                         & \sum_{k} f^k(b_A^{\tensor k})
	\end{align*}
	\label{lemma:pushforwardmad}
\end{lemma}
\begin{proof}
	We want to show that $b_B:= \sum_{k} f^k(b_A^{\tensor k})$ satisfies the Maurer-Cartan equation
	\begin{align*}
		\sum_k m^k_B(b_B^{\tensor k})= & \sum_k m^k_B\left(\left( \sum_{j_1} f^{j_1}(b_A^{\tensor j_1})\right)\tensor \cdots \tensor \left(\sum_{j_k} f^{j_k}(b_A^{\tensor j_k})\right) \right)           \\
		=                              & \sum_l \sum_k \left(\sum_{j_1+ \ldots+ j_k=l} m^k_B (f^{j_1}\tensor \cdots \tensor  f^{j_k})  \right)\circ (b_A^{\tensor l})                                                  \\
		=                              & \sum_l \sum_{i_1+j+i_2=l} f^{i_1+i_2+1}(\id^{\tensor i_1}\tensor  m^{j}_A\tensor \id^{\tensor i_2})  \circ  (b_A^{\tensor l})                                    \\
		=                              & \sum_{i_1, i_2} f^{i_1+i_2+1}\left(\id^{\tensor i_1}\tensor \left(\sum_j m^j_A(b_A^{\tensor j})\right)\tensor \id^{\tensor i_2}\right) \circ (b_A)^{\tensor( i_1+i_2)} \\
		=                             0
	\end{align*}
\end{proof}
This gives the following nice characterization of unobstructed $A_\infty$ algebras. 
\begin{corollary}
	Let $B$ be a filtered $A_\infty$ algebra. There exists a filtered $A_\infty$ homomorphism $0: 0\to B$  if and only if $B$ is unobstructed.
\end{corollary}
\begin{proof}
	Since $0$ is unobstructed, the presence of a filtered $A_\infty$ homomorphism implies that $B$ is unobstructed by pushforward.
	
	Suppose now that $B$ is unobstructed by bounding cochain $b$. 
	The filtered $A_\infty$ homomorphism relation for the map $0_b: 0\to B$, whose curvature term is $0^0_b=b$, is exactly the Maurer-Cartan equation for $B$.
\end{proof}
Surprisingly, deformations commute with each other in the following sense:
\begin{prop}
	Let $a_1, a_2$ be elements of $A$. Then $(A, (m_{a_1})_{a_2})=(A, m_{a_1+a_2})$.
\end{prop}
\begin{proof}
	A calculation shows that
	\begin{align*}
		(m_{a_1}^k)_{a_2}= & \sum_n  m^{k+n}_{a_1}(\id\oplus a_2)^{{n+k\choose n}_{a_2}} \\
		=&\sum_m \sum_n m^{k+m+n} (\id\oplus a_1)^{{n+k+m\choose m}_{a_1}}\circ (\id\oplus a_2)^{{n+k\choose n}_{a_2}}\\
		=& \sum_{m+n} m^{k+m+n}(\id\oplus( a_1+ a_2))^{{n+m+k\choose m+n}_{a_1+a_2}}\\
		=& m^k_{a_1+a_2}
	\end{align*}
\end{proof}
\begin{prop}
	Let $f: A\to B$ be a filtered $A_\infty$ algebra homomorphism.
	Let $a\in A$ be a deforming element.
	Then the map
	\begin{align*}
		f_a: (A, (m^k_A)_a)\to& (B, m^k_B)\\
		f^k_a:=& \sum_{n} f^{k+n}\circ (\id\oplus a)^{{n+k \choose n }_a}.
	\end{align*}
	is an $A_\infty$ homomorphism.
\end{prop}
\begin{proof}
	We check the $A_\infty$ homomorphism relations. 
	\begin{align*}
		\sum_{\substack{j_1+j+j_2=k\\ j, j_1, j_2\geq 0}} &f^{j_1+1+j_2}_a (\id^{\tensor j_1}\tensor (m^j_A)_a\tensor \id^{\tensor{j_2}})\\
		=		& \sum_{\substack{j_1+j+j_2=k\\ j, j_1, j_2, n_1, n_2 \geq 0\\ \alpha=j_1+1+j_2 }} f^{\alpha+n_1}\circ (\id\oplus a)^{{\alpha+n_1\choose n_1}_{a}}\circ\left(\id^{\tensor j_1}\tensor\left( m_A^{j+n_2}\circ (\id\oplus a)^{{j+n_2\choose n_2}_{a}} \right)\tensor \id^{\tensor j_2}\right)\\
		\intertext{Collecting all of the $a$ and $\id$ terms so that we compose with them first,}
		=& \sum_{\substack{j_1+j+j_2=k\\ j, j_1, j_2, n_1, n_2\geq 0\\ \alpha=j_1+1+j_2}}\left(f^{\alpha+n_1}\circ (\id^{\tensor j_1}\tensor m_A^{j+n_2} \tensor \id^{\tensor j_2})\right)\circ (\id\oplus a)^{{k+n_1+n_2\choose n_1+n_2}_{a}} \\
		\intertext{Applying the quadratic $A_\infty$ homomorphism relation to the first composition, setting $n_1+n_2=n$}
		=& \sum_{\substack{i_1+\cdots+i_l=k+n\\ i_j, n_1, n_2\geq 0}} m^l_B\circ (f^{i_1}\tensor \cdots \tensor f^{i_l})\circ (\id\oplus a)^{{k+n\choose n}_{a}} 
		\intertext{Pulling back the $a$ and $\id$ terms into a composition with the $f^{i_j}$}
		=& \sum_{\substack{i_1+\cdots+i_l=k+n\\ n_1+ \cdots n_l=n\\ i_j, n_j, n\geq 0}} m^l_B\circ \left(\left(f^{i_1}\circ (\id\oplus a)^{{n_1+j_1 \choose n_1}_a}\right)\tensor \cdots \tensor \left(f^{i_l}\circ (\id\oplus a)^{{n_l+j_l \choose n_l}_a}\right)\right)\\
		=& \sum_{\substack{\hat i_1+\cdots \hat i_l =k\\ \hat i_j, k \geq 0}}m^l_B(f^{\hat i_1}_a\tensor \cdots \tensor f^{\hat i_l}_a).
	\end{align*}
\end{proof}
One may use the previous claim to construct the pushforward map on bounding cochains, as
\[
	f_*(b)=(f_b)_*(0).
\]
\printbibliography

\affiliationone{
   Jeff Hicks\\
   Department of Pure Mathematics and Mathematical Statistics\\
   University of Cambridge\\
   Wilberforce Road\\
   Cambridge CB3 0WB
   \email{jh2234@cam.ac.uk}}

\end{document}